\documentclass[11pt,a4paper]{amsart}
\usepackage{amsmath,amssymb, amsbsy}
\usepackage{color,psfrag}
\usepackage[dvips]{graphicx} 
\usepackage{tikz}
\usepackage{enumerate}
\newcommand{\R}{{\mathbb R}}
\newcommand{\N}{{\mathbb N}}

\newcommand{\e }{\varepsilon}

\renewcommand{\ge }{\geqslant}
\renewcommand{\geq }{\geqslant}

\renewcommand{\leq }{\leqslant}
\def\neweq#1{\begin{equation}\label{#1}}
\def\endeq{\end{equation}}
\def\eq#1{(\ref{#1})}
\newtheorem{theorem}{Theorem}[section]
\newtheorem{proposition}[theorem]{Proposition}
\newtheorem{lemma}[theorem]{Lemma}

\newtheorem{remark}[theorem]{Remark}
\newtheorem{definition}[theorem]{Definition}
\theoremstyle{definition}

\newcommand{\cotan}{\textnormal{cotan}}
\newcommand{\cotanh}{\textnormal{cotanh}}

\begin{document}

\title[Nonlinear non homogeneous multi span beam]{On the stability of a nonlinear \\ non homogeneous multiply hinged beam}
\author[Elvise BERCHIO]{Elvise BERCHIO}
\address{\hbox{\parbox{5.7in}{\medskip\noindent{Dipartimento di Scienze Matematiche, \\
Politecnico di Torino,\\ Corso Duca degli Abruzzi 24, 10129 Torino, Italy. \\[3pt]
\em{E-mail address: }{\tt elvise.berchio@polito.it}}}}}
\author[Alessio FALOCCHI]{Alessio FALOCCHI}
\address{\hbox{\parbox{5.7in}{\medskip\noindent{Dipartimento di Scienze Matematiche, \\
Politecnico di Torino,\\ Corso Duca degli Abruzzi 24, 10129 Torino, Italy. \\[3pt]
\em{E-mail address: }{\tt alessio.falocchi@polito.it}}}}}
\author[Maurizio GARRIONE]{Maurizio GARRIONE}
\address{\hbox{\parbox{5.7in}{\medskip\noindent{Dipartimento di Matematica, \\
Politecnico di Milano,\\ Piazza Leonardo da Vinci, 32 - 20133 Milano, Italy. \\[3pt]
\em{E-mail address: }{\tt maurizio.garrione@polimi.it}}}}}

\date{\today}



\begin{abstract}
The paper deals with a nonlinear evolution equation describing the dynamics of a non homogeneous multiply hinged beam, subject to a nonlocal restoring force of displacement type. First, a spectral analysis for the associated weighted stationary problem is performed, providing a complete system of eigenfunctions. Then,  
%\begin{equation}\label{ab1}
%p(x)\,u_{tt}+u_{xxxx}+\|u\|_{L_p^2(I)}^2 \,p(x)\, u=0, \quad u=u(x, t), \, x \in [-\pi, \pi], \, t > 0,
%\end{equation}
a linear stability analysis for bi-modal solutions of the evolution problem is carried out, with the final goal of suggesting optimal choices of the density and of the position of the internal hinged points in order to improve the stability of the beam. The analysis exploits both analytical and numerical methods; the main conclusion of the investigation is that \emph{non homogeneous density functions improve the stability of the structure.}
\end{abstract}

\maketitle

\section{Introduction}
We consider a nonlinear evolution problem for a multiply hinged symmetric beam made of non homogeneous material, arising when dealing with certain simplified suspension bridge models recently proposed in \cite{GarGazBK}, see also \cite{babefega} for the simply hinged case. The corresponding stationary equation reads as a multi-point problem for a fourth-order ODE (see, e.g.,  \cite{gazgar} and references therein). As far as we are aware, no evolution problem has been studied in this setting besides the ones treated in \cite{GarGazBK}, for a constant density beam. The main goal of this paper is to suggest both density distributions and positions of the hinged points (i.e., of the \emph{piers}) in order to maximize the stability of the beam in a proper sense: intuitively, a beam is as stabler as higher are the energies associated to those pure oscillations for which a relevant energy transfer, towards another mode of oscillation, takes place, see Section \ref{critical energies}. Therefore, we first develop the spectral analysis of the related weighted eigenvalue problem and then we take advantage of the obtained information to formalize this intuitive idea by performing a linear stability analysis. \smallbreak
More precisely, we deal with the nonlinear equation 
\begin{equation}\label{nonlineareintro} 
p(x)\,u_{tt}+u_{xxxx}+\|u\|_{L_p^2(I)}^2 \,p(x)\, u=0,
\end{equation}
where $t > 0$ is the time variable, $x \in \overline{I}=[-\pi, \pi]$ represents the position along the beam and $p=p(x)$ is the non constant density function. The choice of the nonlinear term $\|u\|_{L_p^2(I)}^2:=\int_I pu^2 \,  dx$ 
is extensively explained in \cite[Chapter 3]{GarGazBK}, where for $p \equiv 1$ it is shown that a nonlinearity of this kind is the most suitable in order to describe the energy transfers between modes in models for real structures. Equation \eqref{nonlineareintro} is complemented with suitable initial data and hinged boundary conditions (see Section \ref{setting}). After providing a complete system of eigenfunctions for the related weighted eigenvalue problem (Proposition \ref{general facts}), we investigate the energy transfer between modes by proceeding with an analysis of \emph{bi-modal solutions} of \eqref{nonlineareintro}; such solutions have only two active Fourier components, one embodying the \emph{prevailing mode} - namely an initially ``large'' oscillation - and the other one representing the \emph{residual mode}, i.e., an oscillation which is tiny at the beginning and acquires energy as time evolves. In fact, the abrupt energy exchange phenomenon is a behaviour that can be observed in real structures (as suspension bridges) and turns to be dangerous for their safety; for more details, we refer to \cite {bookgaz} and to \cite[Chapter 4]{GarGazBK}, where the results obtained for homogeneous beams are shown to be in line with those obtained for more realistic models having two degrees of freedom (as the so-called ``fish-bone'' model proposed in \cite{elvisefilippo}). We also point out that the reduction to a finite number of modes is usual in classical engineering literature, since the energy concentrated on high modes is unrealistically large for tiny oscillations. With these preliminaries, the stability analysis is carried out by linearizing the $2\times 2$ system of ODEs describing the time evolution of the coefficients of the active Fourier components around the solution for which the coefficient of the residual mode is identically equal to $0$. This leads to analyze the stability regions for a suitable Hill equation (see eq. \eqref{hill}) obtaining, in dependence on the density and on the positions of the piers, an energy threshold of instability beyond which linear instability appears. The explicit expression of such a threshold (see formula \eqref{ecritica}) depends increasingly on the ratio between the eigenvalues corresponding to the active modes; we thus choose to restrict our attention to the set of densities determined by solving a maximization problem for the ratio of the eigenvalues that we relate to the lowest energy threshold of instability. A theoretical explanation of the importance of having large ratios of eigenvalues to gain stability in Hamiltonian systems was already given in \cite{bgz}, within the classical stability theory of Mathieu equations; see also \cite{bebuga2} and \cite{befa} for the same issue when dealing with plate models. \par
In the final part of our study, we compute numerically the above mentioned energy thresholds for beams having different densities and different positions of the hinged points; the collected data allow us to compare the performances of such beams and to formulate precise suggestions on how to improve their stability (Sections \ref{2step} and \ref{num}). As the main conclusion of our analysis, we infer that 
\begin{center}
\textbf{a non homogeneous beam is stabler than a homogeneous one}, 
\end{center}
as we observe in Section \ref{sumcon}. 
\par
The paper is organized as follows. In Section \ref{setting}, we frame the nonlinear evolution model and the associated stationary eigenvalue problem in a proper variational setting, establishing a well-posedness result for the former and a spectral theorem, together with a regularity result for the eigenfunctions, for the latter. Section \ref{stab} explains in details the linear stability analysis which will be performed, explicitly determining the critical energy threshold for bi-modal solutions; since such an expression depends increasingly on the ratio of the involved eigenvalues, we then show that the maximization problem for the ratio of two eigenvalues, with respect to the density, always has a solution, given by a suitable \emph{bang-bang density}, i.e., a stepwise function. In Section \ref{2step}, we take the simplest example of such a kind of density, a two-step one, determining more explicitly the eigenvalues of the corresponding stationary eigenvalue problem and then providing a first numerical study about the optimal position of the piers at fixed density. Section \ref{num} is then dedicated to more general numerical experiments taking into account the possibility of having bang-bang densities with an arbitrary number of jumps; we draw therein our conclusions about the optimal combination ``position of the piers-density'' in terms of stability (Section \ref{sumcon}). The final section of the article (Section \ref{dimostrazioni}) collects all the proofs of the statements given along the paper.

\section{Setting of the problem }\label{setting}
 \subsection{The nonlinear evolution model}
 %\label{nonlin}
 We consider a model for a multiply hinged beam divided in three adjacent spans (segments): the main (middle) span and two side spans separated by piers. Without loss of generality, we normalize the total length to $2\pi$ and we denote by $I:=(-\pi,\pi)$ the segment corresponding to the beam. We assume that the beam is hinged at the extremal points $\pm \pi$ and in correspondence of two further symmetric points $\pm a \pi$, where $a \in (0, 1)$ is a real parameter determining the relative measure of the side spans with respect to the main span (compare with \cite{gazgar}). Furthermore, we assume that the beam is non homogeneous and we denote its density function by $p=p(x)$. More precisely, for $0<\alpha<1<\beta$ given, we deal with the following class of densities:
	\begin{equation} \label{eq:famiglia}
P_{\alpha, \beta}:=\left\{p\in L^\infty(I):\, \int_I p\,dx=2\pi,\,\,\alpha\leq p\leq\beta\,\text{ and } p(x)=p(-x) \text{ a.e. in } I \, \, \right\} \, .
\end{equation} 
The integral condition in \eqref{eq:famiglia} represents the preservation of the total mass, which is useful in order to compare the results for different weights. The numbers  $\alpha$ and $\beta$ represent the ``limit values'' for the density we want to employ to build the beam, while the symmetry requirement on $p$ means that we focus on designs which are symmetric with respect to the middle of the beam. 
% Notice that the assumption $\alpha<1<\beta$ is not restrictive; if we assume $\beta= 1$, it must be $p = 1$ a.e. in $(-\pi,\pi)$, since otherwise we would have $\int_I p\,dx<2\pi$; similarly, if we consider $\alpha=1$.\par

 In the following, we denote by $\|\cdot\|_q$ the norm in the Lebesgue space $L^q(I)$ ($1\leq q\leq\infty$) and with $(\cdot,\cdot)$ the scalar product of $L^2(I)$. Given $p\in P_{\alpha, \beta}$, we also endow the space $L^2$ with the weighted scalar product $(u,v)_{p}:=(p\, u,v)$, for all $u,v\in L^2(I)$. Let us emphasize that the corresponding norm $\|u\|_{L_p^2(I)}^2=(u,u)_{p}$ is equivalent to the $L^2$-norm but, for the sake of clarity, we maintain the different notations.  Finally, by assuming that the Young modulus of the beam is constant, the total energy of the beam reads:
$$
\mathfrak{E}(u)=\frac{1}{2}\, \int_I p(x)\,u_t^2\,dx+\frac{1}{2}\, \int_I \, u_{xx}^2\,dx+\frac{1}{4}\, \left(\int_I p(x)\,u^2\,dx\right)^2\,,$$
see e.g. \cite[Example VI]{levine}. Namely, we add to the kinetic and to the bending energy a nonlocal term which models a beam whose displacement behaves superquadratically and nonlocally; in other words, when the beam is displaced from its equilibrium position in some point, we assume that this increases, in a way proportional to the density, the resistance to further displacements in all the other points. Then, by proceeding formally, the resulting evolution problem is
 \begin{equation}
\begin{cases}
p(x)\,u_{tt}+u_{xxxx}+\|u\|_{L_p^2(I)}^2 \,p(x)\, u=0& \quad x\in I\,,\, t>0\\
u(x,0)=g(x)\,,\,\,u_t(x,0)=h(x)&  \quad x\in I\,,\\
u(\pm\pi,t)=u''(\pm \pi,t)=0 & \quad  t>0\,,\\
u(\pm a\pi,t)=0&  \quad  t>0\,.
\end{cases}
\label{beamEV}
\end{equation}
 We refer to \cite{GarGazBK} for the treatment of problem \eq{beamEV} when $p\equiv1$ and for an exhaustive explanation of the fact that the linear stability analysis for the nonlinear model \eqref{beamEV} allows to recover several crucial features of the behaviour of real structures such as bridges. This motivates the applicative interest of the study we are going to perform.

Problem \eq{beamEV} is written in strong form but, in general, there is not enough regularity for this formulation. To make precise what we mean by solutions of \eq{beamEV}, we introduce the space 
$$V(I):=\{u\in H^2\cap H^1_0(I): u(\pm a\pi)=0  \}\,.$$
Notice that $V(I)$ embeds into $C^1(\overline I)$, therefore the conditions $u(\pm\pi)=u(\pm a\pi)=0$ can be meant pointwise; moreover, $V(I)$ is a Hilbert space when endowed with the scalar product of $H^2\cap H^1_0(I)$
 $$
 (u,v)_{V}:=(u'',v'') \quad (\text{with associated norm }
  \|u\|_{V}^2=(u,u)_{V}).
$$ 
We denote by $V'(I)$ the dual space of $V(I)$ and with $\langle \cdot,\cdot \rangle_{V}$ the corresponding duality pairing. With these preliminaries, we can formalize the concept of weak solution of \eqref{beamEV}. 
\begin{definition}\label{weaks}
 Let $T>0$, $g\in V(I)$ and $h\in L^2(I)$. We say that $u\in C^0([0,T]; V(I))\cap C^1([0,T]; L_p^2(I))$ such that
 $$\|u(x,t)-g(x)\|_{V(I)}\rightarrow 0 \quad \text{and} \quad \|u_t(x,t)-h(x)\|_{L^2_p(I)}\rightarrow 0 \text{ as } t\rightarrow 0^+\,$$
  is a weak solution to \eq{beamEV} if it satisfies the equation 
\begin{equation}\label{beam_weak}
 \int_0^T (u_t(x,s),\varphi(x))_p\, \chi'(s) \, ds=  \int_0^T \bigg\{ (u_{xx}(x,s),\varphi_{xx}(x)) + \|u(s)\|_{L_p^2(I)}^2\,(u(x,s),\varphi(x))_p\, \bigg\}\, \chi(s) ds
 \end{equation}
for all $\varphi \in V(I)$ and all $\chi \in \mathcal{D}(0, T)$. 
 \end{definition}

 \begin{remark}\label{regw}
\textnormal{It is readily seen from \eq{beam_weak} that $pu_t$ admits distributional derivative in the $t$ variable and that $pu\in C^2([0,T]; V'(I))$, hence \eq{beam_weak} may be rewritten as}
 \begin{equation*}
 \langle p(x)u_{tt},\varphi(x)\rangle_{V}+ (u_{xx},\varphi_{xx}(x))+( \|u\|_{L_p^2(I)}^2\,u,\varphi(x))_p=0 \quad \forall \varphi\, \in  V(I)\,, \forall\, t>0\,.
 \end{equation*}
 \end{remark}
  
Existence and uniqueness for weak solutions of \eq{beamEV} is given in the following result.
 \begin{theorem}\label{Galerkin}
 Let $T>0$ (including $T=\infty$), $g\in V(I)$ and $h\in L^2(I)$. Problem \eq{beamEV} admits a unique (weak) solution $u\in C^0([0,T]; V(I))\cap C^1([0,T]; L_p^2(I))$.
Furthermore, $pu\in C^2([0,T]; V'(I))$ and $u$ satisfies
  \begin{equation}\label{energy}
\frac{1}{2}\, \|u_t(t)\|_{L_p^2(I)}^2\,+\frac{1}{2}\,\|u_{xx}(t)\|_{L^2(I)}^2+\frac{1}{4}\,\|u(t)\|_{L_p^2(I)}^4=\frac{1}{2}\, \|h\|_{L_p^2(I)}^2\,+\frac{1}{2}\,\|g_{xx}\|_{L^2(I)}^2+\frac{1}{4}\,\|g\|_{L_p^2(I)}^4\,,
 \end{equation}
 for all $t>0$.
\end{theorem}

\subsection{The oscillating modes of the multiply hinged beam}
%\label{weighteigenpb}
 The linear stability analysis we plan to perform on problem \eqref{beamEV} is strictly related to the fundamental modes of oscillation of the multiply hinged non homogeneous beam, i.e., to the eigenfunctions of the weighted eigenvalue problem:
\begin{equation}
\begin{cases}
e''''(x)=\lambda \,p(x) e(x) \qquad x\in I:=(-\pi,\pi)\\
e(\pm\pi)=e''(\pm \pi)=0\\e(\pm a\pi)=0\,.
\end{cases}
\label{beam0}
\end{equation}
We say that $e=e_{\lambda}$ solves \eqref{beam0} in weak sense if $e\in V(I)$ and
\begin{equation}
\int_I e''v''\,dx=\lambda\,\int_I p(x)\,ev\,dx\qquad \forall v\in V(I).
\label{beamweak}
\end{equation}
We characterize the solutions of \eqref{beam0} as follows.
\begin{proposition}\label{general facts}
Let $p\in P_{\alpha,\beta}$. Then all the eigenvalues of  \eqref{beam0} are simple and can be represented by means of an increasing and divergent sequence $\lambda_j(p)$ ($j\in \N_+$). Furthermore, the corresponding eigenfunctions $e_j=e_{\lambda_j}$ form a complete system in $L^2_p(I)$ and in $V(I)$. 
 \end{proposition}
About the regularity of the eigenfunctions, we let $I_-:=(-\pi,-a\pi)$, $I_0:=(-a\pi,a\pi)$ and $I_+:=(a\pi,\pi)$, so that $\overline I=\overline I_-\cup \overline I_0 \cup \overline I_+$, and we prove the following.
\begin{proposition}\label{regolarita}
Let $p\in P_{\alpha,\beta}$. If $e=e_{\lambda}\in V(I)$ satisfies \eqref{beamweak}, then it solves \eqref{beam0} a.e. in $I$. Furthermore, writing
\begin{equation}\label{u}
e(x)=\left\{\begin{array}{ll}
e_-(x) &  x\in \overline I_-\\
e_0(x) & x\in \overline I_0\\
e_+(x) & x\in \overline I_+\,,
\end{array}
\right.
\end{equation}
it holds
$$
e\in C^2(\overline I),\qquad e_-\in C^3(\overline I_-), \qquad e_0\in C^3(\overline I_0),\qquad e_+\in C^3(\overline I_+).
$$	
\end{proposition}
\begin{remark}
	If $p\in C^0(\overline I_-)$, we obtain $e_-\in C^4(\overline I_-)$; similarly, if $p\in C^0(\overline I_0)$ (resp., $p\in C^0(\overline I_+)$), then $e_0\in C^4(\overline I_0)$ (resp., $e_+\in C^4(\overline I_+)$). Anyway, if $p\in C^0(\overline I)$ we do not obtain more than $e\in C^2(\overline I)$, see \cite{holubova}.
\end{remark}

\section{The linear stability analysis}\label{stab}

In this section, we proceed with the linear stability analysis for bi-modal solutions of \eqref{beamEV}.
% We then aim at optimizing the stability of the structure, acting on $p$ and $a$ in order to raise, as much as possible, such an energy threshold of instability. 
% in case of a two-step piecewise density.
%Given a density $p \in P_{\alpha, \beta}$, we investigate which is a  is to determine an optimal placement for the two symmetric piers in order to prevent dangerous exchanges of energy between the possible modes of oscillation. As usual in classical engineering literature, we devote our attention
\subsection{The critical energy thresholds}\label{critical energies}
We start by characterizing the uni-modal solutions of \eqref{beamEV}, having the form 
\begin{equation}\label{1modo}
u_\lambda(x,t)=W_\lambda(t)e_\lambda(x), 
\end{equation}
where $\{e_\lambda\}_\lambda$ is the complete system of eigenfunctions
provided by Proposition \ref{general facts}; from now on, we assume that each $e_\lambda$ is \emph{normalized in $L^2_p(I)$}. To simplify our analysis, we restrict our attention to the case of potential initial conditions, namely we will assume henceforth that the function $h$ appearing in \eqref{beamEV} is identically zero; such a choice is not restrictive and does not affect the results (see for instance \cite[Chapter 3]{GarGazBK} for an explanation).
For a fixed eigenvalue $\lambda$ of \eqref{beam0}, we assume that $g(x)=\zeta e_\lambda(x)$; inserting \eqref{1modo} into \eqref{beamEV} and taking the $L^2$-scalar product with $e_\lambda$, we then obtain the Cauchy problem
\begin{equation}\label{duffingc}
\ddot{W}_\lambda(t)+\lambda W_\lambda(t)+W_\lambda(t)^3=0,\quad W_{\lambda}(0)=\zeta,\quad\dot{W}_\lambda(0)=0,
\end{equation}
where the differential equation is of Duffing type. 
It is well-known that all the solutions of such a Cauchy problem can be expressed in terms of Jacobi functions and are periodic, with period given by 
$$
T=T(\zeta)=\frac{4}{b}\int_0^{\pi/2}\frac{d\varphi}{\sqrt{1-\gamma^2\sin^2\varphi}},
$$
where 
$
b=\sqrt{\lambda+\zeta^2}$ and $\gamma=\frac{\zeta}{b}\sqrt{\frac{1}{2}}$.
%being $\zeta=\sqrt{\lambda^2+2\zeta_2^2+\zeta_1^4+2\lambda \zeta_1^2}-\lambda$. 
We call $u_\lambda$ defined in \eqref{1modo} a \emph{$\lambda$-nonlinear-mode} of \eq{beamEV}.
We then consider bi-modal solutions of \eqref{beamEV} - in the sense of Definition \ref{weaks} - having the form 
\begin{equation}\label{form3}
u(x,t)=w(t)e_\lambda(x)+z(t)e_\nu(x),
\end{equation}
where $e_\lambda$ and $e_\nu$ are two different eigenfunctions of \eqref{beam0}, with associated eigenvalues respectively given by $\lambda$ and $\nu$. Incidentally, we notice that a solution of \eqref{beamEV} will indeed have this form if the initial data $g$ and $h$ are concentrated on such two modes. 
Inserting \eq{form3} into \eq{beamEV}, we obtain the nonlinear differential system
\begin{equation}\label{sistema2}
\left\{\begin{array}{l}
\ddot{w}(t)+\lambda w(t)+(w(t)^2+z(t)^2)w(t)=0\\
\ddot{z}(t)+\nu z(t)+(w(t)^2+z(t)^2)z(t)=0,  
\end{array}\right.
\end{equation}
which we consider together with the initial conditions
\begin{equation}\label{iniziali}
w(0)=\zeta>0, \, \dot{w}(0)=0, \qquad z(0)=z_0 \ll \zeta, \, \,\dot{z}(0)=0.
\end{equation}
We are thus choosing $e_\lambda$ as the prevailing mode and $e_\nu$ as the residual mode, in the sense explained in the Introduction. 
If $z_0=0$, then the solution of \eq{sistema2}-\eq{iniziali} is given by the couple $(\bar{W}_{\zeta, \lambda}, 0)$, where $\bar{W}_{\zeta, \lambda}$ is the $T$-periodic solution of \eqref{duffingc}; of course, here no energy transfer between the two modes is observed, since $z$ remains constantly equal to $0$.
We then consider small perturbations of this situation, wondering if $(\bar{W}_{\zeta, \lambda}, 0)$ is linearly stable as a solution of \eqref{sistema2} when $\vert z_0 \vert > 0$ is small. More formally, we say that:
\begin{definition}
%\label{linstab}
The $\lambda$-nonlinear mode is {\bf linearly stable} ({\bf unstable}) with respect to the $\nu$-nonlinear-mode if
$\xi\equiv0$ is a stable (unstable) solution of the linear Hill equation
\begin{equation}\label{hill}
\ddot{\xi}(t)+(\nu + \bar{W}_{\zeta, \lambda}(t)^2) \xi(t)=0.
\end{equation}
\end{definition}
The analysis of the stability for \eqref{hill} was done in details in \cite{GarGazBK}, showing that it is equivalent to the study of the stability of an equation of the kind
$$
\ddot{\xi}(t)+\left(\frac{\nu}{\lambda}+\Psi_\lambda(t)^2\right)\xi(t)=0,
$$
with $\Psi_\lambda$ $T$-periodic, which turns out to be completely characterized. We can then use \cite[Proposition 3.4]{GarGazBK}, which in our framework reads as follows.
\begin{proposition}\label{stabilita}
Let $\lambda\neq\nu$ be two eigenvalues of \eqref{beam0}. The $\lambda$-nonlinear-mode of amplitude $\zeta$ is linearly stable with respect to the $\nu$-nonlinear-mode
if and only if 
$$
\text{either }\lambda>\nu\mbox{ and }\zeta>0\qquad\mbox{or}\qquad\lambda<\nu\mbox{ and }\zeta \leq \sqrt{2(\nu-\lambda)}.
$$
\end{proposition}
%\begin{figure}[ht!]
%\begin{center}
%\includegraphics[scale=1]{Figura.png}
%\caption{Regions of linear stability (white) and instability (gray) for \eqref{hill}.}\label{completa}
%\end{center}
%\end{figure}
% In Figure \ref{completa} we give a complete characterization of the stability, which holds also in our setting; 
From Proposition \ref{stabilita} we finally derive our critical energy thresholds of instability. Consider the solution of \eqref{beamEV} having the form \eqref{form3}; by Proposition \ref{stabilita}, the critical amplitude of instability is then defined as
\begin{equation*}
%\label{defD}
D(\lambda,\nu)= \sqrt{2(\nu-\lambda)},
\end{equation*}
namely the initial amplitude $w(0)$ which leads to enter the region of linear instability. The critical energy of instability is then defined as the (constant) energy of the solution of \eqref{duffingc} when $\zeta=D(\lambda,\nu)$ in \eq{iniziali}:
\begin{equation}\label{ecritica}
E(\lambda,\nu):=\frac{\lambda D(\lambda,\nu)^2}{2}+\frac{D(\lambda,\nu)^4}4=\left(\frac{\nu}{\lambda}-1\right)\lambda\, \nu
\end{equation}
(recall that we start with zero kinetic datum). By \eq{ecritica} we immediately see that a crucial parameter of stability for \eqref{hill} is the ratio between the eigenvalues corresponding to the active modes $e_\lambda$ and $e_\nu$; in the next section, we analytically study this quantity on varying of $p$. 

\subsection{Maximizing the ratio of two eigenvalues}

\label{stablin}
The results of Section \ref{critical energies} suggest that a possible way to prevent the energy transfer from a lower to a higher mode is to look for density functions which increase the energy threshold \eq{ecritica}; since this can be done by increasing the ratios of eigenvalues, it appears natural to study the maximum problem
\begin{equation}\label{opt}
\mathcal{R}_{\nu,\lambda}=\mathcal{R}_{\nu,\lambda}(a)=\sup_{p\in P_{\alpha,\beta}}\dfrac{\nu(p)}{\lambda(p)}\qquad (\nu>\lambda \text{ eigenvalues of \eq{beam0}}).
\end{equation}
For the vibrating string, a similar problem was studied in \cite{keller}, without the integral constraint on $p$, and in \cite{banks-gentry}, assuming $\alpha=0$; in both cases, the authors proved that the maximum of the ratio is achieved by a weight of \emph{bang-bang} type, namely a piecewise constant function, symmetric with respect to the middle of the string and getting the maximum value there. See also \cite{befa}, where partial results were obtained for the same maximum problem for partially hinged plates. Broadly speaking, problem \eq{opt} can be related to the well known \emph{composite membrane} and \emph{composite plate} problems, i.e., the problems of building a body of prescribed shape and mass out of given materials, in such a way that the first frequency of the resulting membrane (or plate) is as small as possible, see e.g. \cite{anedda2,befafega,chanillo,chen,CV2,cox,cuccu22,lapr} and the monograph \cite{henrot}.   Coming back to problem \eqref{opt}, we first notice the following:
 \begin{proposition}\label{existence}
 Let $\nu>\lambda$ be eigenvalues of \eqref{beam0}; then, problem \eqref{opt} admits a solution. 
 \end{proposition}
 
Now we turn to a possible characterization of maximizers of  \eqref{opt}. In this framework, a central role is played by the function
 \begin{equation}\label{gp}
 g(p,x)= g_{\nu,\lambda}(p,x):=\frac{\nu(p)}{\lambda(p)}\big(e_{\lambda}^2(x)-e_{\nu}^2(x)\big)\quad \text{for } \nu>\lambda\,, \, p\in P_{\alpha,\beta}  \text{ and }x\in I \,,
 \end{equation}
 where $e_\lambda$ and $e_\nu$ denote the eigenfunctions associated, respectively, with $\lambda(p)$ and $\nu(p)$. Denoting by $\chi_D$ the characteristic function of a set $D\subset I$ and setting $D^c=I\setminus D$, we prove:
 \begin{theorem}\label{thm-rapp}
Let $\nu>\lambda$ be eigenvalues of \eqref{beam0} and let $\widehat{p}=\widehat{p}(\nu,\lambda)$ denote a maximizer of \eqref{opt}. Then there exists $\widehat t=\widehat t (\nu,\lambda)\in\mathbb{R}$ such that the set
 $$
 \widehat I=\widehat I(\nu,\lambda):=\bigg\{x\in I:  g(\widehat{p},x)\geq \widehat t\bigg\}\,
 $$
satisfies $$|\widehat{I}|=|I|\frac{1-\alpha}{\beta-\alpha}\,.$$
Furthermore, set $A_{\widehat t}:=\big\{x\in I:  g(\widehat{p},x)= \widehat t \, \big\}$, two situations may occur:
 \begin{itemize}
 \item[$(i)$] if $ |A_{\widehat t}|=0$, then
 \begin{equation*}\label{pnu}
 \widehat{p}(x) = \beta \chi_{ \widehat I} (x)+ \alpha \chi_{\widehat I^c}(x)\,\quad \text{for a.e. } x\in I\,;
 \end{equation*}
\item[$(ii)$] if  $ |A_{\widehat t}|>0$, then
$$\widehat{p}(x)=\beta \text{ for a.e. } x\in  \widehat I\setminus A_{\widehat t} \quad  \text{ and } \quad  \widehat{p}(x)=\alpha  \text{ for a.e. } x\in \widehat I^c\setminus A_{\widehat t}\,.$$ 	
\end{itemize}
 \end{theorem}
 \begin{figure}
	\centering
{\includegraphics[width=7cm]{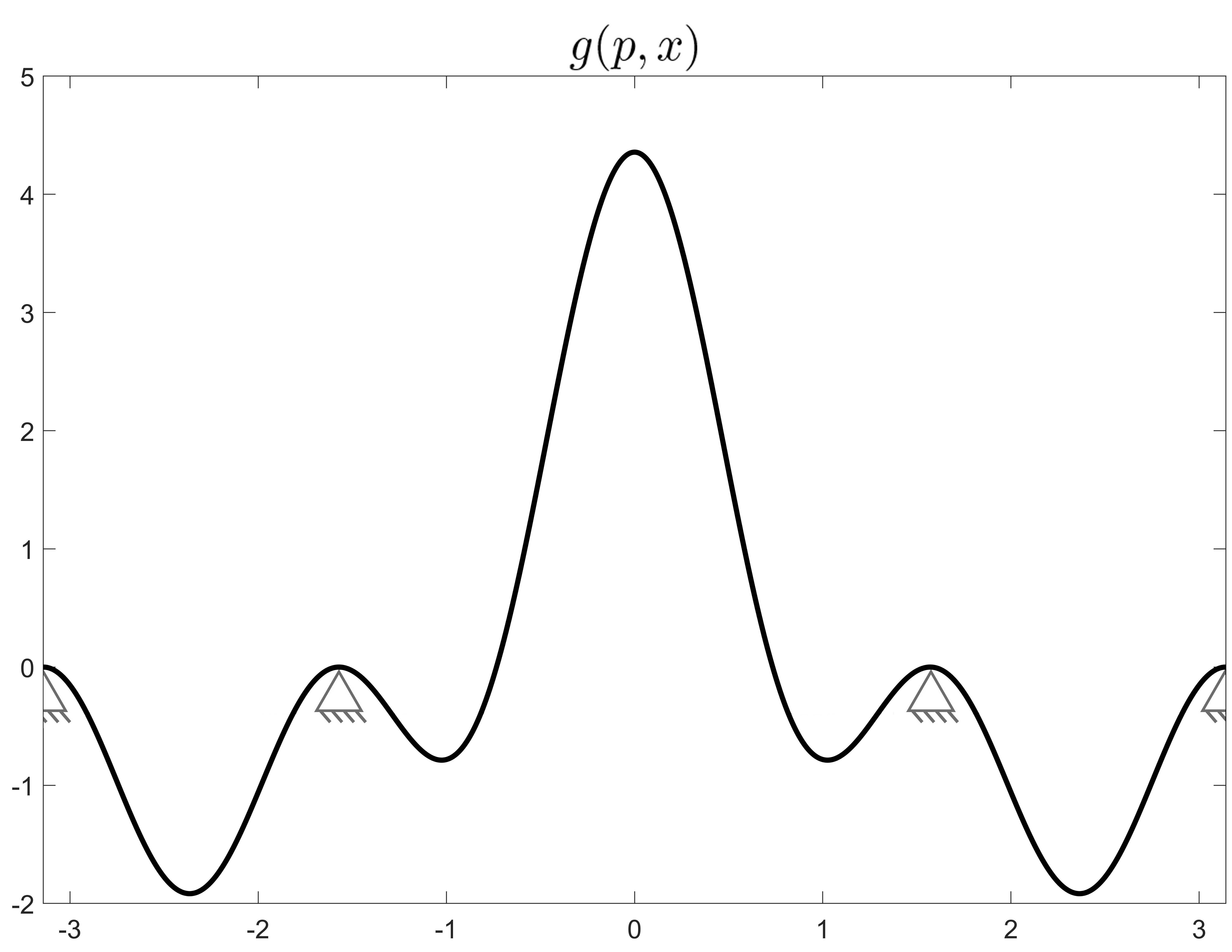}}
\quad
{\includegraphics[width=7.5cm]{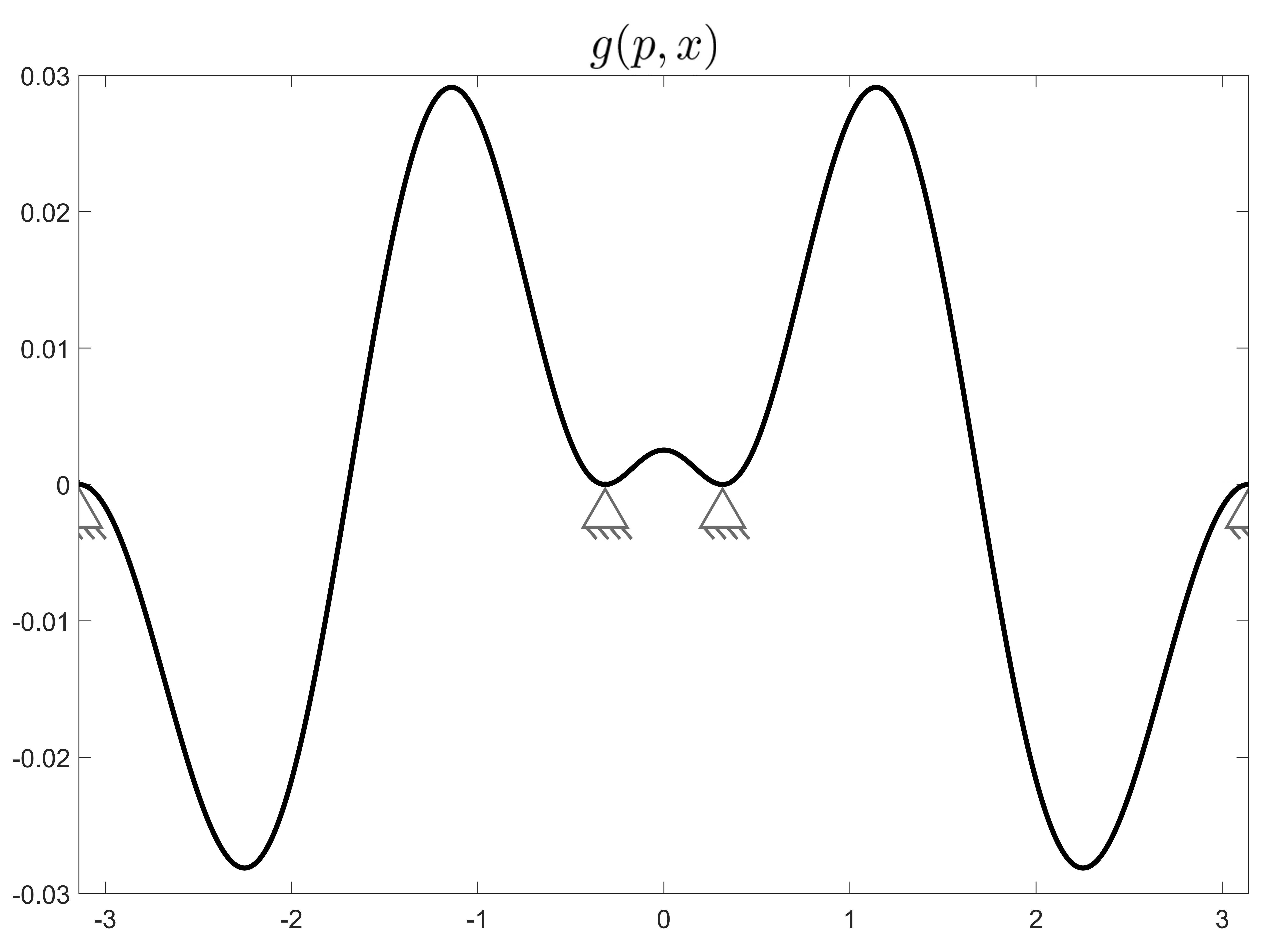}}
\caption{The graphs of $g_{\lambda_2,\lambda_1}(1,x)$ (namely, for $p\equiv1$), with $a=0.5$ (left) and $a=0.1$ (right).}
\label{plotg}
\end{figure}
 
 Theorem \ref{thm-rapp} relates the best location of the materials within the beam, in the sense explained at the beginning of Section \ref{stablin}, to the level sets of the function $g$. In general, it is not easy to study analytically this function; by means of the numerical method outlined in Section \ref{num}, we have plot the graph of $g$ for several choices of $a$ and $p$ and we have always found that its level sets have zero measures, therefore $|A_{\widehat t}|=0$, see e.g. Figure \ref{plotg}. However, we do not have an analytic proof of this fact or, equivalently, of the fact that $(ii)$ does not occur. 
\par
We highlight some properties of the function $g$ in the following proposition, which shows that more information can be desumed when dealing with eigenfunctions $e_\lambda$ and  $e_\nu$ having different parity.
 \begin{proposition}\label{gmax}
 Let $g$ be as in \eqref{gp}. It holds
  $$g(p,\pm \pi)=g'(p,\pm \pi)=0 \quad \text{ and } \quad g(p,\pm a \pi)=g'(p,\pm a \pi)=0\,.$$
Furthermore, if $e_\lambda$ is even (resp., odd) and $e_\nu$ is odd (resp., even), then
 $$g(p,0)>0\,, \; g'(p,0)=0 \,\; \text{and} \,\; g''(p,0)<0\,, \quad \text{(resp., } g(p, 0) < 0 \,, \; g'(p, 0)=0 \,\; \text{and} \,\; g''(p, 0) > 0 \text{)},$$
 i.e., $g$ has a local maximum (resp., minimum) at $x=0$.
 \end{proposition}
Referring to formula \eq{ecritica}, it is clear that the most dangerous situation in terms of stability occurs when $\lambda$ and $\nu$ are as close as possible, i.e., when $\lambda$ and $\nu$ are consecutive eigenvalues; we conjecture that this necessarily implies that the associated eigenfunctions $e_\lambda$ and $e_\nu$ have different parity, whence Proposition \ref{gmax} applies. In order to reach a theoretical proof of this fact, one should probably exploit some general properties of the eigenvalues or, in some particular cases (as the one of piecewise constant densities), one could try to proceed ``by hand'' as in \cite{gazgar}. In any case, there is a strong numerical evidence supporting this conjecture (see Sections \ref{2step} and \ref{num}, in particular Figure \ref{autovalori}).

 \begin{remark}
\textnormal{The qualitative properties of the function $g$ defined in \eqref{gp} vary according to the weight, to the chosen eigenvalues and to the values of the parameters $a$, $\alpha$ and $\beta$; thus, it seems difficult to analytically deduce general information about $g$. Under the assumptions of Proposition \ref{gmax}, if we also know that for the chosen weight $p$ the eigenfunction $e_\lambda$ has a global maximum at $x=0$, then 
	$$g(p,x)\leq \frac{\nu(p)}{\lambda(p)}\, e_{\lambda}^2(x)\leq  \frac{\nu(p)}{\lambda(p)}\, e_{\lambda}^2(0)= g(p,0) \quad \text{for all } x \in I\,,$$
	hence $g(p,x)$ has a global maximum at $x=0$, see Figure \ref{plotg} on the left. However, there may also be cases in which the maximum at $0$ is only local (Figure \ref{plotg} on the right) or, exchanging the parity of $e_\lambda$ and $e_\nu$, the critical point of $g$ at $x=0$ is a minimum. %In this situation, Theorem \ref{thm-rapp} suggests to locate the denser material $\beta$ near $x=0$, i.e. in the middle of the beam, see Figure \ref{plotg} on the left.  %This may be the case also in some cases for which the function $g$ does not have a global minimum at $x=0$, as in the middle picture of Figure \ref{plotg}. 
Anyway, if the weight $p$ is the maximizer provided by Theorem \ref{thm-rapp}, i.e., $p=\widehat{p}$, the shape of $g(\widehat{p}, x)$ suggests where one should place the heavier and the lighter material along the beam. In Section \ref{num}, we will actually proceed iteratively, computing the function $g$ for several different weights, in order to provide suggestions about possible optimal locations of the materials in terms of stability.  
}
	\end{remark}

\section{Two-step piecewise constant densities}\label{2step}

Neglecting case $(ii)$ which never occurs in our numerics, Theorem \ref{thm-rapp} states that the maximizers for the ratio of two eigenvalues are piecewise constant densities; %the comments at the end of the previous section may suggest that the heavier density has to be present in the middle of the beam. 
in this section, we thus start by considering such a kind of weights, focusing for the moment on the simpler case of a density taking only two values and having exactly two symmetric jumps. We will see that this allows a quite explicit characterization of the instability regions. We will always denote by $\alpha$ the density of the lighter material and by $\beta$ the density of the heavier one (being, as in the previous sections, $\alpha < 1 < \beta$).

\subsection{Explicit computation of the eigenvalues}
Let $\alpha, \beta$ be fixed with $\alpha < 1 < \beta$ and let $p \in P_{\alpha, \beta}$ be such that, for every $x \in I$, it holds $p(x)=\alpha$ or $p(x)=\beta$. Assuming that $p$ has only two symmetric discontinuities, the mass constraint on $p$ implies that such jumps will occur at the points $x=\pm \rho \pi$, where $\rho$ is defined by
\begin{equation}\label{ics0}
\rho:=\frac{1-\alpha}{\beta-\alpha}\, ,\quad \textrm{if } p(0)=\beta, \qquad \textrm{or} \qquad \rho:= \frac{\beta-1}{\beta-\alpha}\,, \quad \textrm{if } p(0)= \alpha.
\end{equation}
With such a definition of $\rho$, we are thus considering a density $p$ having the form 
\begin{equation}\label{defp2}
p(x)=\beta \chi_{[-\rho \pi, \rho \pi]}(x) + \alpha \chi_{(0, \pi) \setminus [-\rho \pi, \rho \pi]} (x), \quad \textrm{if } p(0)=\beta, 
\end{equation}
or
\begin{equation}\label{defp22}
p(x)=\alpha \chi_{[-\rho \pi, \rho \pi]}(x) + \beta \chi_{(0, \pi) \setminus [-\rho \pi, \rho \pi]} (x), \quad \textrm{if } p(0)= \alpha.
\end{equation}
We now determine explicitly the eigenvalues of problem \eqref{beam0} for these simplified cases; by the parity of $p$, we can reason only on the right-half $[0, \pi]$ of the beam, which is divided into the three subintervals $I_1, I_2, I_3$ represented in Figure \ref{figbeam}. 
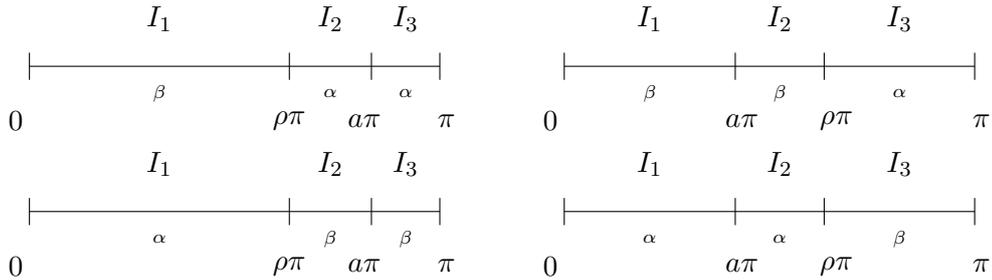
\begin{figure}[!ht]
\begin{center}
%\makebox[\linewidth][c]{
\begin{tikzpicture}[xscale=0.9,yscale=0.9]
\draw (0,0) -- (6,0);
\draw (0, -0.2) -- (0, 0.2);
\draw (3.8, -0.2) -- (3.8, 0.2);
\draw (5, -0.2) -- (5, 0.2);
\draw (6, -0.2) -- (6, 0.2);
\node at (-0.2, -0.8){$0$};
\node at (3.8, -0.8){$\rho \pi$};
\node at (4.9, -0.8){$a\pi$};
\node at (6.1, -0.8){$\pi$};
\node at (1.9, -0.4){\tiny{$\beta$}};
\node at (4.4, -0.4){\tiny{$\alpha$}};
\node at (5.5, -0.4){\tiny{$\alpha$}};
\node at (1.9, 0.7){$I_1$};
\node at (4.4, 0.7){$I_2$};
\node at (5.5, 0.7){$I_3$};
\end{tikzpicture}
\qquad 
\begin{tikzpicture}[xscale=0.9,yscale=0.9]
\draw (0,0) -- (6,0);
\draw (0, -0.2) -- (0, 0.2);
\draw (3.8, -0.2) -- (3.8, 0.2);
\draw (2.5, -0.2) -- (2.5, 0.2);
\draw (6, -0.2) -- (6, 0.2);
\node at (-0.2, -0.8){$0$};
\node at (3.98, -0.8){$\rho \pi$};
\node at (2.6, -0.8){$a\pi$};
\node at (6.1, -0.8){$\pi$};
\node at (3.15, -0.4){\tiny{$\beta$}};
\node at (1.25, -0.4){\tiny{$\beta$}};
\node at (4.9, -0.4){\tiny{$\alpha$}};
\node at (3.15, 0.7){$I_2$};
\node at (1.25, 0.7){$I_1$};
\node at (4.9, 0.7){$I_3$};
\end{tikzpicture}
\quad 
\begin{tikzpicture}[xscale=0.9,yscale=0.9]
\draw (0,0) -- (6,0);
\draw (0, -0.2) -- (0, 0.2);
\draw (3.8, -0.2) -- (3.8, 0.2);
\draw (5, -0.2) -- (5, 0.2);
\draw (6, -0.2) -- (6, 0.2);
\node at (-0.2, -0.8){$0$};
\node at (3.8, -0.8){$\rho \pi$};
\node at (4.9, -0.8){$a\pi$};
\node at (6.1, -0.8){$\pi$};
\node at (1.9, -0.4){\tiny{$\alpha$}};
\node at (4.4, -0.4){\tiny{$\beta$}};
\node at (5.5, -0.4){\tiny{$\beta$}};
\node at (1.9, 0.7){$I_1$};
\node at (4.4, 0.7){$I_2$};
\node at (5.5, 0.7){$I_3$};
\end{tikzpicture}
\qquad 
\begin{tikzpicture}[xscale=0.9,yscale=0.9]
\draw (0,0) -- (6,0);
\draw (0, -0.2) -- (0, 0.2);
\draw (3.8, -0.2) -- (3.8, 0.2);
\draw (2.5, -0.2) -- (2.5, 0.2);
\draw (6, -0.2) -- (6, 0.2);
\node at (-0.2, -0.8){$0$};
\node at (3.98, -0.8){$\rho \pi$};
\node at (2.6, -0.8){$a\pi$};
\node at (6.1, -0.8){$\pi$};
\node at (3.15, -0.4){\tiny{$\alpha$}};
\node at (1.25, -0.4){\tiny{$\alpha$}};
\node at (4.9, -0.4){\tiny{$\beta$}};
\node at (3.15, 0.7){$I_2$};
\node at (1.25, 0.7){$I_1$};
\node at (4.9, 0.7){$I_3$};
\end{tikzpicture}
\caption{A visual description of an inhomogeneous half-beam with constant densities $\alpha < \beta$ and a pier placed at the point $a\pi$, $a < 1$, according to whether, on the one hand, the denser (top) or the lighter (bottom) material is placed in the middle of the beam and, on the other hand, $a > \rho$ (left) or $a < \rho$ (right).}\label{figbeam}
\end{center}
\end{figure}
%hence to investigate more in details (at least numerically) the linear stability for bi-modal solutions of \eqref{beamEV}.  

We first observe that, proceeding as in the proof of Proposition \ref{regolarita}, we are here able to state that the eigenfunctions $e_k$ belong to $C^4(I_l)$, $l=1, 2, 3$, and they glue with $C^3$ regularity in correspondence of $\pm\rho \pi$ and with $C^2$ regularity in correspondence of the piers (unless $\rho=a$, in which case they simply glue with $C^2$ regularity in correspondence of the piers). Consequently, we can seek the eigenfunctions of \eqref{beam0} 
in the form 
\begin{equation}\label{forma}
e_{\lambda}(x)=\left\{
\begin{array}{ll}
A_1 \cos(\lambda x)+B_1 \sin(\lambda x) + C_1 \cosh(\lambda x) + D_1 \sinh(\lambda x) & x \in I_1 \\ 
A_2 \cos(\lambda x)+B_2 \sin(\lambda x) + C_2 \cosh(\lambda x) + D_2 \sinh(\lambda x) & x \in I_2 \\ 
A_3 \cos(\lambda x)+B_3 \sin(\lambda x) + C_3 \cosh(\lambda x) + D_3 \sinh(\lambda x) & x \in I_3, 
\end{array}
\right.
\end{equation}
imposing that they are strong solutions of the differential equation in \eqref{beam0} on $I_l$, $l=1, 2, 3$, and that the gluings at the endpoints of $I_1, I_2, I_3$ fulfill the described regularity. Of course, the expression of $e_\lambda$ in \eqref{forma} is then extended by even or odd symmetry to the whole $I$.
This procedure leads to the following result. 
\begin{proposition}\label{autoesp}
Let $\alpha, \beta$ be fixed, with $\alpha < 1 < \beta$, and let $p$ be as in \eqref{defp2} or in \eqref{defp22}, with $\rho$ defined as in \eqref{ics0}. Moreover, let
\begin{equation*}
%\label{definizioni}
\left\{
\begin{array}{llll}
 \sigma=\alpha^{1/4} & \text{and} & \tau=\beta^{1/4} & \text{ if } p(0)=\beta \vspace{0.1cm}\\
 \sigma=\beta^{1/4} & \text{and} & \tau=\alpha^{1/4} & \text{ if } p(0)=\alpha, 
\end{array}
\right.
\end{equation*}
and set 
$$
\delta=\frac{\tau}{\sigma}.
$$
Then, the eigenvalues of \eqref{beam0} are given by $\lambda=\mu^4$, where $\mu$ is implicitly determined as follows:
\smallbreak
\noindent
- for \underline{even eigenfunctions}, 
\begin{itemize}
\item if $\rho=a$, $\mu$ is a solution of
{\footnotesize
$$
\delta\cos[\mu\pi\sigma(1-\rho)]\sinh[\mu\pi\sigma(1-\rho)] = \sin[\mu\pi\sigma(1-\rho)](\delta \cosh[\mu\pi\sigma(1-\rho)]-  \sinh[\mu\pi\sigma(1-\rho)](\tan(\mu\pi\tau \rho)+\tanh[\mu\pi\tau \rho]));
$$
}\item if $\rho > a$, $\mu$ is a solution of $\textnormal{det}_{e, +}(\mu)=0$, where $\textnormal{det}_{e, +}$ is the determinant associated with the linear system 
{\footnotesize
\begin{equation*}
%\label{sistemae+}
\left\{
\begin{array}{l}
\blacklozenge \; \cos^2(\mu \tau a \pi)\sinh(\mu \tau a \pi) X_1 + [\cosh(\mu \tau a \pi)+\sin(\mu \tau a \pi)\cos(\mu \tau a \pi)\sinh(\mu \tau a \pi)] X_2 + \\ 
\cos(\mu \tau a \pi)\sinh(\mu \tau a \pi)\cosh(\mu \tau a \pi) X_3+ \cos(\mu \tau a \pi) \cosh^2(\mu \tau a \pi)X_4 = 0 \vspace{0.1cm} \\
\blacklozenge \; \cos(\mu \tau a \pi) X_1 +  \sin(\mu \tau a \pi) X_2 + \cosh(\mu \tau a \pi) X_3 + \sinh(\mu \tau a \pi) X_4 = 0 \vspace{0.1cm} \\
\blacklozenge \; (1-\delta^2) [\cos(\mu \tau \rho \pi)\cosh(\mu \sigma (\rho-1) \pi)+\delta \sin(\mu \tau \rho \pi)\sinh(\mu \sigma (\rho-1) \pi)] X_1 + \\
(1-\delta^2) [\sin(\mu \tau \rho \pi)\cosh(\mu \sigma (\rho-1) \pi)-\delta \cos(\mu
 \tau \rho \pi)\sinh(\mu \sigma (\rho-1) \pi)] X_2 +\\ 
(1+\delta^2) [\cosh(\mu \tau \rho \pi)\cosh(\mu \sigma (\rho-1) \pi)-\delta \sinh(\mu \tau \rho \pi)\sinh(\mu \sigma (\rho-1) \pi)] X_3 + \\
(1+\delta^2) [\sinh(\mu \tau \rho \pi)\cosh(\mu \sigma (\rho-1) \pi)-\delta \cosh(\mu \tau \rho \pi)\sinh(\mu \sigma (\rho-1) \pi)] X_4 =0  \vspace{0.1cm}\\ 
\blacklozenge \; (1+\delta^2) [\cos(\mu \tau \rho \pi)\cos(\mu \sigma (\rho-1) \pi)+\delta \sin(\mu \tau \rho \pi)\sin(\mu \sigma (\rho-1) \pi)] X_1 + \\
(1+\delta^2) [\sin(\mu
 \tau \rho \pi)\cos(\mu \sigma (\rho-1) \pi)-\delta \cos(\mu \tau \rho \pi)\sin(\mu \sigma (\rho-1) \pi)] X_2 +\\ 
(1-\delta^2) [\cosh(\mu \tau \rho \pi)\cos(\mu \sigma (\rho-1) \pi)-\delta \sinh(\mu \tau \rho \pi)\sin(\mu \sigma (\rho-1) \pi)] X_3 + \\
(1-\delta^2) [\sinh(\mu \tau \rho \pi)\cos(\mu \sigma (\rho-1) \pi)-\delta \cosh(\mu \tau \rho \pi)\sin(\mu \sigma (\rho-1) \pi)] X_4 =0; 
\end{array}
\right.
\end{equation*}
}\item if $\rho < a$, $\mu$ is a solution of $\textnormal{det}_{e, -}(\mu)=0$, where $\textnormal{det}_{e, -}$ is the determinant associated with the linear system 
{\footnotesize
\begin{equation*}
%\label{sistemae-}
\left\{
\begin{array}{l}
\blacklozenge \; \cos(\mu \sigma a \pi) X_1 + \sin(\mu \sigma a \pi) X_2 + \cosh(\mu \sigma a \pi) X_3 + \sinh(\mu \sigma a \pi) X_4 = 0 \vspace{0.1cm} \\
\blacklozenge \{\cos(\mu \sigma a \pi)\cos[\mu\sigma(a-1)\pi]\sinh[\mu \sigma (a-1) \pi] - \cos(\mu \sigma a\pi)\cosh[\mu\sigma(a-1)\pi]\sin[\mu \sigma (a-1) \pi] + \\ 
\sin(\mu\sigma a\pi)\sin[\mu \sigma(a-1)\pi]\sinh[\mu\sigma(a-1)\pi]\} X_1 + 
\{\sin(\mu \sigma a \pi)\cos[\mu\sigma(a-1)\pi]\sinh[\mu \sigma (a-1) \pi] - \\ \sin(\mu \sigma a\pi)\cosh[\mu\sigma(a-1)\pi]\sin[\mu \sigma (a-1) \pi] - 
\cos(\mu\sigma a\pi)\sin[\mu \sigma(a-1)\pi]\sinh[\mu\sigma(a-1)\pi]\}
X_2 -\\
\sinh(\mu\sigma a \pi) \sin[\mu \sigma(a-1)\pi] \sinh[\mu \sigma(a-1)\pi] X_3 -\cosh(\mu\sigma a\pi)\sin[\mu \sigma(a-1)\pi]\sinh[\mu \sigma(a-1)\pi] X_4=0 \vspace{0.1cm}\\
\blacklozenge \; (1-\delta^2) [\sin(\mu \sigma \rho \pi)\cosh(\mu \tau \rho \pi)+\delta \cos(\mu \sigma \rho \pi)\sinh(\mu \tau \rho \pi)] X_1 - \\
(1-\delta^2) [\cos(\mu \sigma \rho \pi)\cosh(\mu \tau \rho \pi)-\delta \sin(\mu \sigma \rho \pi)\sinh(\mu \tau \rho \pi)] X_2 + \\
(1+\delta^2) [\sinh(\mu \sigma \rho \pi)\cosh(\mu \tau \rho \pi)-\delta \cosh(\mu \sigma \rho \pi)\sinh(\mu \tau \rho \pi)] X_3 + \\
(1+\delta^2) [\cosh(\mu \sigma \rho \pi)\cosh(\mu \tau \rho \pi)-\delta \sinh(\mu \sigma \rho \pi)\sinh(\mu \tau \rho \pi)] X_4 =0  \vspace{0.1cm}\\ 
\blacklozenge \; (1+\delta^2) [\sin(\mu \sigma \rho \pi)\cos(\mu \tau \rho \pi)-\delta \cos(\mu \sigma \rho \pi)\sin(\mu \tau \rho \pi)] X_1 - \\
(1+\delta^2) [\cos(\mu \sigma \rho \pi)\cos(\mu \tau \rho \pi)+\delta \sin(\mu \sigma \rho \pi)\sin(\mu \tau \rho \pi)] X_2 +\\ 
(1-\delta^2) [\sinh(\mu \sigma \rho \pi)\cos(\mu \tau \rho \pi)+\delta \cosh(\mu \sigma \rho \pi)\sin(\mu \tau \rho \pi)] X_3 + \\
(1-\delta^2) [\cosh(\mu \sigma \rho \pi)\cos(\mu \tau \rho \pi)+\delta \sinh(\mu \sigma \rho \pi)\sin(\mu \tau \rho \pi)] X_4 =0;
\end{array}
\right.
\end{equation*}
}\end{itemize}
\smallbreak
\noindent
- for \underline{odd eigenfunctions}, 
\begin{itemize}
\item if $\rho=a$, $\mu$ is a solution of
{\footnotesize
$$
\delta\cosh[\mu\pi\sigma(1-\rho)]\sin[\mu\pi\sigma(1-\rho)] = \sinh[\mu\pi\sigma(1-\rho)](\delta \cos[\mu\pi\sigma(1-\rho)]-  \sin[\mu\pi\sigma(1-\rho)](\cotan(\mu\pi\tau \rho)+\cotanh[\mu\pi\tau \rho]));
$$
}
\item if $\rho > a$, $\mu$ is a solution of $\textnormal{det}_{o, +}(\mu)=0$, where $\textnormal{det}_{o, +}$ is the determinant associated with the linear system 
{\footnotesize
\begin{equation*}
%\label{sistemao+}
\left\{
\begin{array}{l}
\blacklozenge \; [\sinh(\mu \tau a \pi)-\sin(\mu \tau a \pi)\cos(\mu \tau a \pi)\cosh(\mu \tau a \pi)] X_1 -\sin^2(\mu \tau a \pi)\cosh(\mu \tau a \pi) X_2 - \\ 
\sin(\mu \tau a \pi)\sinh^2(\mu \tau a \pi) X_3- \sin(\mu \tau a \pi) \sinh(\mu \tau a \pi)\cosh(\mu \tau a \pi)X_4 = 0 \vspace{0.1cm} \\
\blacklozenge \; \cos(\mu \tau a \pi) X_1 +  \sin(\mu \tau a \pi) X_2 + \cosh(\mu \tau a \pi) X_3 + \sinh(\mu \tau a \pi) X_4 = 0 \vspace{0.1cm} \\
\blacklozenge \; (1-\delta^2) [\cos(\mu \tau \rho \pi)\cosh(\mu \sigma (\rho-1) \pi)+\delta \sin(\mu \tau \rho \pi)\sinh(\mu \sigma (\rho-1) \pi)] X_1 + \\
(1-\delta^2) [\sin(\mu \tau \rho \pi)\cosh(\mu \sigma (\rho-1) \pi)-\delta \cos(\mu \tau \rho \pi)\sinh(\mu \sigma (\rho-1) \pi)] X_2 +\\ 
(1+\delta^2) [\cosh(\mu \tau \rho \pi)\cosh(\mu \sigma (\rho-1) \pi)-\delta \sinh(\mu \tau \rho \pi)\sinh(\mu \sigma (\rho-1) \pi)] X_3 + \\
(1+\delta^2) [\sinh(\mu \tau \rho \pi)\cosh(\mu \sigma (\rho-1) \pi)-\delta \cosh(\mu \tau \rho \pi)\sinh(\mu \sigma (\rho-1) \pi)] X_4 =0  \vspace{0.1cm}\\ 
\blacklozenge \; (1+\delta^2) [\cos(\mu \tau \rho \pi)\cos(\mu \sigma (\rho-1) \pi)+\delta \sin(\mu \tau \rho \pi)\sin(\mu \sigma (\rho-1) \pi)] X_1 + \\
(1+\delta^2) [\sin(\mu \tau \rho \pi)\cos(\mu \sigma (\rho-1) \pi)-\delta \cos(\mu \tau \rho \pi)\sin(\mu \sigma (\rho-1) \pi)] X_2 +\\ 
(1-\delta^2) [\cosh(\mu \tau \rho \pi)\cos(\mu \sigma (\rho-1) \pi)-\delta \sinh(\mu \tau \rho \pi)\sin(\mu \sigma (\rho-1) \pi)] X_3 + \\
(1-\delta^2) [\sinh(\mu \tau \rho \pi)\cos(\mu \sigma (\rho-1) \pi)-\delta \cosh(\mu \tau \rho \pi)\sin(\mu \sigma (\rho-1) \pi)] X_4 =0; 
\end{array}
\right.
\end{equation*}
}\item if $\rho < a$, $\mu$ is a solution of $\textnormal{det}_{o, -}(\mu)=0$, where $\textnormal{det}_{o, -}$ is the determinant associated with the linear system 
{\footnotesize
\begin{equation*}
%\label{sistemao-}
\left\{
\begin{array}{l}
\blacklozenge \; \cos(\mu \sigma a \pi) X_1 + \sin(\mu \sigma a \pi) X_2 + \cosh(\mu
 \sigma a \pi) X_3 + \sinh(\mu \sigma a \pi) X_4 = 0 \vspace{0.1cm} \\
\blacklozenge \{\cos(\mu \sigma a \pi)\cos[\mu\sigma(a-1)\pi]\sinh[\mu \sigma (a-1) \pi] - \cos(\mu \sigma a\pi)\cosh[\mu\sigma(a-1)\pi]\sin[\mu \sigma (a-1) \pi] + \\ 
\sin(\mu\sigma a\pi)\sin[\mu \sigma(a-1)\pi]\sinh[\mu\sigma(a-1)\pi]\} X_1 + 
\{\sin(\mu \sigma a \pi)\cos[\mu\sigma(a-1)\pi]\sinh[\mu \sigma (a-1) \pi] - \\ \sin(\mu \sigma a\pi)\cosh[\mu\sigma(a-1)\pi]\sin[\mu \sigma (a-1) \pi] - 
\cos(\mu\sigma a\pi)\sin[\mu \sigma(a-1)\pi]\sinh[\mu\sigma(a-1)\pi]\}
X_2 -\\
\sinh(\mu\sigma a \pi) \sin[\mu \sigma(a-1)\pi] \sinh[\mu \sigma(a-1)\pi] X_3 -\cosh(\mu\sigma a\pi)\sin[\mu \sigma(a-1)\pi]\sinh[\mu \sigma(a-1)\pi] X_4=0 \vspace{0.1cm}\\
\blacklozenge \; (1-\delta^2) [\sin(\mu \sigma \rho \pi)\sinh(\mu \tau \rho \pi)+\delta \cos(\mu \sigma \rho \pi)\cosh(\mu \tau \rho \pi)] X_1 - \\
(1-\delta^2) [\cos(\mu \sigma \rho \pi)\sinh(\mu \tau \rho \pi)-\delta \sin(\mu \sigma \rho \pi)\cosh(\mu \tau \rho \pi)] X_2 + \\
(1+\delta^2) [\sinh(\mu \sigma \rho \pi)\sinh(\mu \tau \rho \pi)-\delta \cosh(\mu \sigma \rho \pi)\cosh(\mu \tau \rho \pi)] X_3 + \\
(1+\delta^2) [\cosh(\mu \sigma \rho \pi)\sinh(\mu \tau \rho \pi)-\delta \sinh(\mu \sigma \rho \pi)\cosh(\mu \tau \rho \pi)] X_4 =0  \vspace{0.1cm}\\ 
\blacklozenge \; (1+\delta^2) [-\sin(\mu \sigma \rho \pi)\sin(\mu \tau \rho \pi)-\delta \cos(\mu \sigma \rho \pi)\cos(\mu \tau \rho \pi)] X_1 + \\
(1+\delta^2) [\cos(\mu \sigma \rho \pi)\sin(\mu \tau \rho \pi)-\delta \sin(\mu \sigma \rho \pi)\cos(\mu \tau \rho \pi)] X_2 +\\ 
(1-\delta^2) [-\sinh(\mu \sigma \rho \pi)\sin(\mu \tau \rho \pi)+\delta \cosh(\mu \sigma \rho \pi)\cos(\mu \tau \rho \pi)] X_3 + \\
(1-\delta^2) [-\cosh(\mu \sigma \rho \pi)\sin(\mu \tau \rho \pi)+\delta \sinh(\mu \sigma \rho \pi)\cos(\mu \tau \rho \pi)] X_4 =0.
\end{array}
\right.
\end{equation*}
}\end{itemize}
\end{proposition}
The proof is essentially a matter of explicit computations: 
one starts from the $10 \times 10$ linear system satisfied by the coefficients $A_l, B_l, C_l, D_l$ appearing in \eqref{forma}  ($l=1, 2, 3$), obtained taking into account the parity of the eigenfunction, imposing the internal-boundary conditions in \eqref{beam0} and requiring the regularity stated in Proposition \ref{regolarita}. By direct computations, such a system is simplified until it takes the form appearing in Proposition \ref{autoesp}; further simplifications appear heavy to perform, as well as explicitly writing the deriving determinant. We omit further details, as well as the explicit expression of the eigenfunctions, which appears cumbersome; yet, the simplification to a $4\times4$ linear system is already enough to proceed with some experimental analyses without being affected by the possible numerical drawbacks appearing for a $10\times10$ matrix.  
\smallbreak

\subsection{Stability analysis: numerical results}
\label{stab2step}
We proceed as explained in Section \ref{critical energies}. We fix two eigenvalues $\lambda < \nu$ of \eqref{beam0} and we consider the solutions of \eqref{beamEV} of the form \eqref{form3} starting with potential initial data, measuring their instability through the notion of critical energy threshold given in \eq{ecritica}. As already remarked, formula \eq{ecritica} suggests that the most dangerous situation in terms of stability occurs for $\lambda=\lambda_j$ and $\nu=\lambda_{j+1}$, for some positive integer $j$, i.e. $\lambda$ and $\nu$ consecutive eigenvalues. Furthermore, since twelve modes are more than enough to approximate the motion of real bridges and of beams having similar structural responses, we restrict our attention to the first twelve modes (see, \cite[Section 4.3]{GarGazBK} for a detailed explanation). Actually, the following maximum problem naturally arises: we fix the two-step density distribution $p$ 
and we define, for every $a \in (0, 1)$, the corresponding energy threshold of linear instability as
\begin{equation}\label{sogliaen}
\mathcal{E}(a):=\, \min_{1 \leq j \leq 11} E\big(\lambda_j(a),\lambda_{j+1}(a)\big)\,,
\end{equation}
where $E$ is as given in  \eq{ecritica}; on varying $a \in (0, 1)$, we can then determine the best placement of the piers at fixed density, namely the one maximizing $\mathcal{E}(a)$. We then numerically compute the energy threshold of instability on varying $p$ and $a$. 
\medbreak
From the point of view of the applications, it is reasonable to assume that $\alpha \in \Sigma_1:=\{5/6, 2/3, 1/2, 1/3\}$ and $\beta \in \Sigma_2:=\{3/2, 2, 5/2, 3\}$; to give a rough idea, if $p \equiv 1$ is taken as corresponding to a beam made of reinforced concrete, we consider densities ranging from light materials such as some kinds of wood, to heavy materials such as steel (common materials in civil constructions). As for the choice of $a$, we take $a \in A:=\{0.10, 0.20, 0.30, 0.35, 0.40, 0.45, \ldots, 0.65, 0.70, 0.80, 0.90\}$; namely, we generally move $a$ with a step of $0.1$, but when taking into account ``physical values'' of $a$, that is, in a neighborhood of the 
``physical range'' 
\begin{equation}\label{physrangea}
\frac{1}{2} \leq a \leq \frac{2}{3}, 
\end{equation}
corresponding to the most frequent configurations in real structures \cite{podolny}, we refine the step in order to have a more complete overview.
We then study the problem
\begin{equation}\label{amax}
\mathcal{E}^*:=\max_{a\in A} \mathcal{E}(a)\,, 
\end{equation}
in dependence on the chosen densities (thus, $\mathcal{E}^*=\mathcal{E}^*(\alpha, \beta)$). 
In Table \ref{energia}, we take into account the case of \emph{heavier density in the middle of the beam} ($p(0)=\beta$), reporting the critical energy of instability $\mathcal{E}^*$, together with the corresponding ratio $R$ of eigenvalues achieving the minimum in \eq{sogliaen}, the optimal value $a=a_{\text{opt}}$ achieving the maximum in \eq{amax} and the value of $\rho$; in Table \ref{energia2}, we report the same data in the case of \emph{lighter density in the middle the beam} ($p(0)=\alpha$). In bold characters, we indicate the situations in which, in correspondence of the optimal configuration, the pier has to be reinforced, that is, around the pier we find the \emph{heavier material} (i.e., $a<\rho$ in Table \ref{energia} and $a>\rho$ in Table \ref{energia2}). Incidentally, notice that the critical energy thresholds $\mathcal{E}^*$ are always of the order of $10^2$, and for this reason we choose to report them using the normalized scientific notation (i.e., writing the value of $\mathcal{E}^* / 10^2$). 

\begin{table}[ht!]
\begin{center}
\resizebox{\textwidth}{!}{\footnotesize
\begin{tabular}{|c|c|c|c|c||c|c|c|c||c|c|c|c||c|c|c|c||}
\hline
$\!\alpha\downarrow$ $\beta \to\!$  & \multicolumn{4}{c||}{3/2} & \multicolumn{4}{c||}{2} & \multicolumn{4}{c||}{5/2} & \multicolumn{4}{c||}{3}  \\
\hline
  & $\mathcal{E}^* / 10^{2}$ & $R$ & $a_\text{opt}$ & $\rho$ & $\mathcal{E}^* / 10^{2}$ & $R$ & $a_\text{opt}$ & $\rho$   & $\mathcal{E}^* / 10^{2}$ & $R$ & $a_\text{opt}$ & $\rho$ & $\mathcal{E}^* / 10^{2}$ & $R$ & $a_\text{opt}$ & $\rho$ \\
\hline
$5/6$ & 2.15 & $\lambda_2/\lambda_1$ & 0.50  & 0.25 & 2.76 & $\lambda_2/\lambda_1$ & 0.50 & 0.14 & 3.04 & $\lambda_2/\lambda_1$ & 0.50  & 0.10 & 2.93 & $\lambda_3/\lambda_2$ & 0.50 & 0.08 \\
\hline
$2/3$ & 2.30 & $\lambda_2/\lambda_1$ & 0.45  & 0.40 & 2.37 & $\lambda_2/\lambda_1$ & 0.45 & 0.25 &2.57 & $\lambda_2/\lambda_1$ & 0.50  & 0.18 & 3.22 & $\lambda_2/\lambda_1$ & 0.50 & 0.14  \\
\hline
$1/2$  & 2.85&$\lambda_2/\lambda_1$& \textbf{0.45}& 0.50 & 2.67 &  $\lambda_3/\lambda_2$ & 0.40 & 0.33& 2.60&$\lambda_3/\lambda_2$& 0.40& 0.25 & 2.87 &  $\lambda_2/\lambda_1$ & 0.45 & 0.20  \\
\hline
$1/3$  & 3.41 &$\lambda_2/\lambda_1$& \textbf{0.40}& 0.57 & 3.77 &  $\lambda_2/\lambda_1$ & \textbf{0.40} & 0.40 & 3.20 &$\lambda_3/\lambda_2$& 0.35& 0.31 & 3.25 & $\lambda_3/\lambda_2$ & 0.35 & 0.25  \\
\hline
\end{tabular}
}
\smallbreak
\caption{Parameters of linear instability in the case of heavier density in the middle of the beam ($p(0)=\beta$).}\label{energia}
\end{center}
\end{table}

\begin{table}[ht!]
\begin{center}
\resizebox{\textwidth}{!}{\footnotesize
\begin{tabular}{|c|c|c|c|c||c|c|c|c||c|c|c|c||c|c|c|c||}
\hline
$\!\alpha\downarrow$ $\beta \to\!$  & \multicolumn{4}{c||}{3/2} & \multicolumn{4}{c||}{2} & \multicolumn{4}{c||}{5/2} & \multicolumn{4}{c||}{3}  \\
\hline
  & $\mathcal{E}^* / 10^{2}$ & $R$ & $a_\text{opt}$ & $\rho$ & $\mathcal{E}^* / 10^{2}$ & $R$ & $a_\text{opt}$ & $\rho$   & $\mathcal{E}^* / 10^{2}$ & $R$ & $a_\text{opt}$ & $\rho$ & $\mathcal{E}^* / 10^{2}$ & $R$ & $a_\text{opt}$ & $\rho$ \\
\hline
$5/6$ & 2.12 & $\lambda_2/\lambda_1$ & 0.55  & 0.75 & 2.36 & $\lambda_2/\lambda_1$ & 0.55 & 0.86 & 2.60 & $\lambda_3/\lambda_2$ & 0.50  & 0.90 & 2.92 & $\lambda_2/\lambda_1$ & 0.50 & 0.92 \\
\hline
$2/3$ & 2.19 & $\lambda_3/\lambda_2$ & 0.55  & 0.60 & 2.07 & $\lambda_2/\lambda_1$ & 0.60 & 0.75 &2.58 & $\lambda_3/\lambda_2$ & 0.55  & 0.82 & 3.11 & $\lambda_2/\lambda_1$ & 0.55 & 0.86  \\
\hline
$1/2$  & 3.06&$\lambda_2/\lambda_1$& \textbf{0.60} & 0.50 & 2.64 &  $\lambda_2/\lambda_1$ & 0.60 & 0.67& 2.61&$\lambda_3/\lambda_2$& 0.60& 0.75 & 2.93 &  $\lambda_2/\lambda_1$ & 0.60 & 0.80 \\
\hline
$1/3$  & 3.31 &$\lambda_2/\lambda_1$& \textbf{0.60}& 0.43 & 4.40 &  $\lambda_2/\lambda_1$ & \textbf{0.65} & 0.60 & 3.89 &$\lambda_2/\lambda_1$& 0.65& 0.69 & 3.70 &  $\lambda_2/\lambda_1$ & 0.65 & 0.75  \\
\hline
\end{tabular}
}
\smallbreak
\caption{Parameters of linear instability in the case of lighter density in the middle of the beam ($p(0)=\alpha$).}\label{energia2}
\end{center}
\end{table}

Taking into account that, in the homogeneous case $\alpha=\beta=1$, the critical energy threshold \eq{sogliaen} is equal to $2.17 \cdot 10^2$ \cite[Table 3.9]{GarGazBK}, reached by the ratio $\lambda_2/\lambda_1$, based the critical energy thresholds $\mathcal{E}^*$ collected in Tables \ref{energia} and \ref{energia2} we infer the following conclusions:
\begin{enumerate}[a)]
\item a non homogeneous beam is in general more stable than a homogeneous one, i.e., the critical energy thresholds are in general larger in the former case;
% \item in general, the higher is $\beta/\alpha$, the higher seem to be the corresponding energy thresholds of instability;
\item the optimal position of the piers ranges between $a=0.35$ and $a=0.50$ in the case $p(0)=\beta$ and between $a=0.55$ and $a=0.65$ in the case $p(0)=\alpha$;
\item fixing $\beta$, on decreasing of $\alpha$ the optimal position of the piers moves towards the middle of the beam if $p(0)=\beta$ and towards the endpoints if $p(0)=\alpha$, and the corresponding energy thresholds generally become larger;
\item in general, it is better to have $p(x)=\alpha$ for $x \approx a\pi$, i.e., around the pier there should be the lighter material, with some exceptions if $\alpha$ is sufficiently small;
\item if $\alpha$ is sufficiently small (i.e., $\alpha=1/3$ and $\alpha=1/2$), the beam tends to be stabler if $p(0)=\alpha$, rather than if $p(0)=\beta$;
\item the ratio of eigenvalues which is responsible for the critical energy threshold (i.e., the one corresponding to the eigenvalues achieving the minimum in \eq{sogliaen}) is generally $R=\lambda_2/\lambda_1$, up to some cases in which it is $R=\lambda_3/\lambda_2$ (see the comments at the end of the section); %It seems that subsequent couples of eigenvalue curves produce this effect with higher frequency on growing of $\beta$. 
\item both for $p$ as in \eq{defp2} and for $p$ as in \eq{defp22}, the best pairing of materials is $\alpha=1/3$ and $\beta=2$, namely
$$
\max_{(\alpha,\beta) \in \Sigma_1 \times \Sigma_2} \mathcal{E}^*(\alpha,\beta)=\mathcal{E}^*(1/3,2)\,.
$$ 
\end{enumerate}
In Figure \ref{autovalori}, for some choices of $\alpha$ and $\beta$ we depict the eigenvalue curves implicitly defined, in the $(a, \mu)$-plane, by the equalities $\textnormal{det}_{e, \pm}(\mu)=0$ and $\textnormal{det}_{o,\pm}(\mu)=0$ ($\mu(a)=\lambda^{1/4}(a)$, see Proposition \ref{autoesp}), the blue ones corresponding to even eigenfunctions and the orange ones to odd eigenfunctions. The qualitative properties of the eigenvalue curves are indeed important in the stability analysis; it seems that, the higher is $\beta$, the faster such curves change monotonicity (displaying faster oscillations between different critical points) for $a$ small. Comparing with \cite[Figure 9]{gazgar}, where the same picture is shown for $\alpha=1=\beta,$ it seems that, for fixed $\alpha$, the curves are ``compressed towards $a=0$'' and ``stretched towards $a=1$'' on growing of $\beta$. These deformations possibly make couples of consecutive eigenvalue curves other than $\{\mu=\lambda_1^{1/4}(a), \mu=\lambda_2^{1/4}(a)\}$ very close (for $a$ belonging to a certain range); the ratio between the corresponding eigenvalues may then turn to be small enough to become the one responsible for the loss of stability. With reference to the above observation f), the right picture in Figure \ref{autovalori} shows a case in which the second and the third eigenvalue curves have this behavior for $a=a_{\text{opt}}$ (red line), and indeed the ratio of eigenvalues corresponding to the instability threshold is here $\lambda_3/\lambda_2$, see Table \ref{energia}. Incidentally, notice that Figure \ref{autovalori} brings a strong evidence of the fact that odd and even eigenfunctions alternate, as claimed after Proposition \ref{gmax}. 
\begin{figure}
	\centering
{\includegraphics[width=5cm]{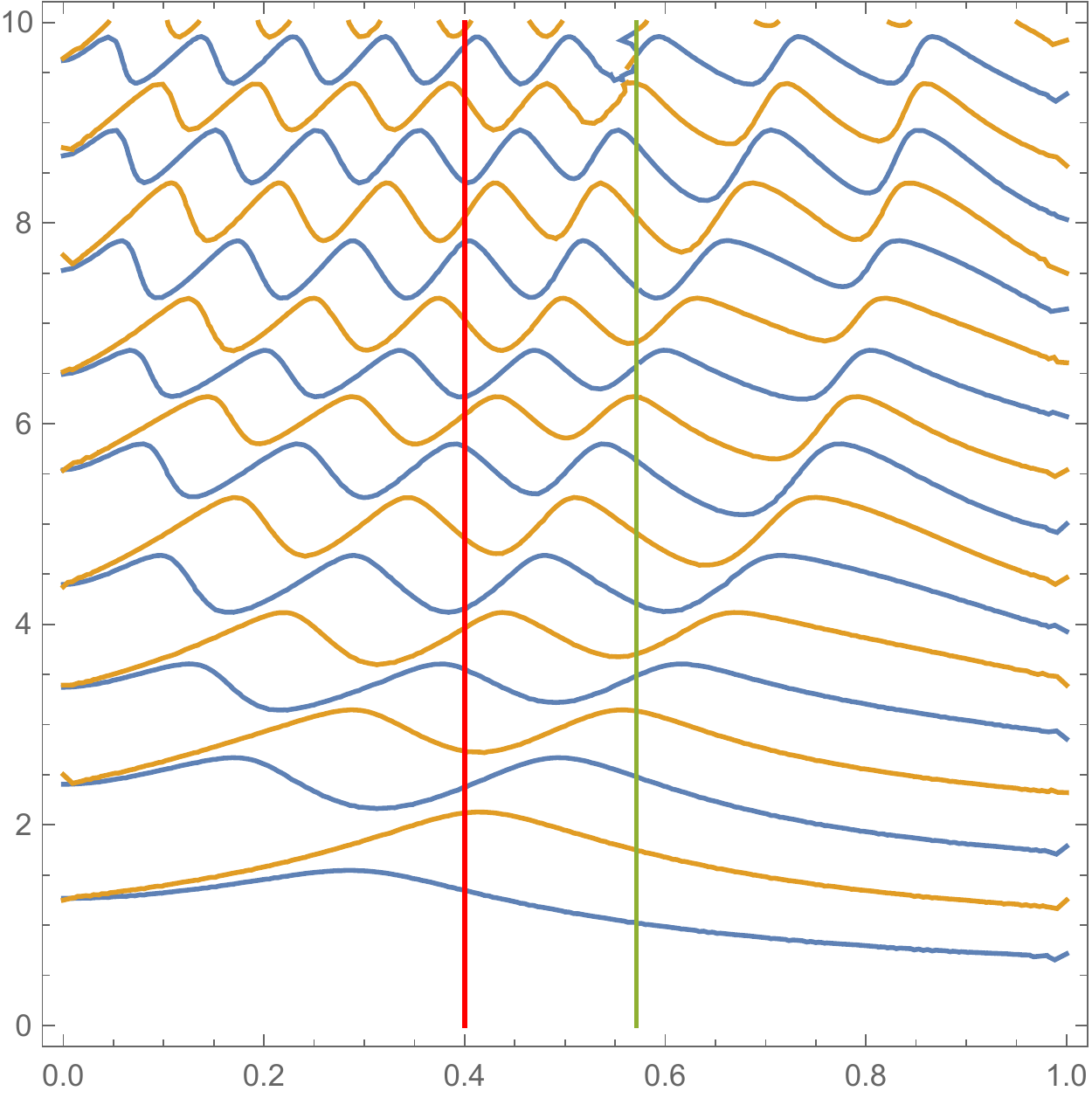}}
\qquad 
{\includegraphics[width=5cm]{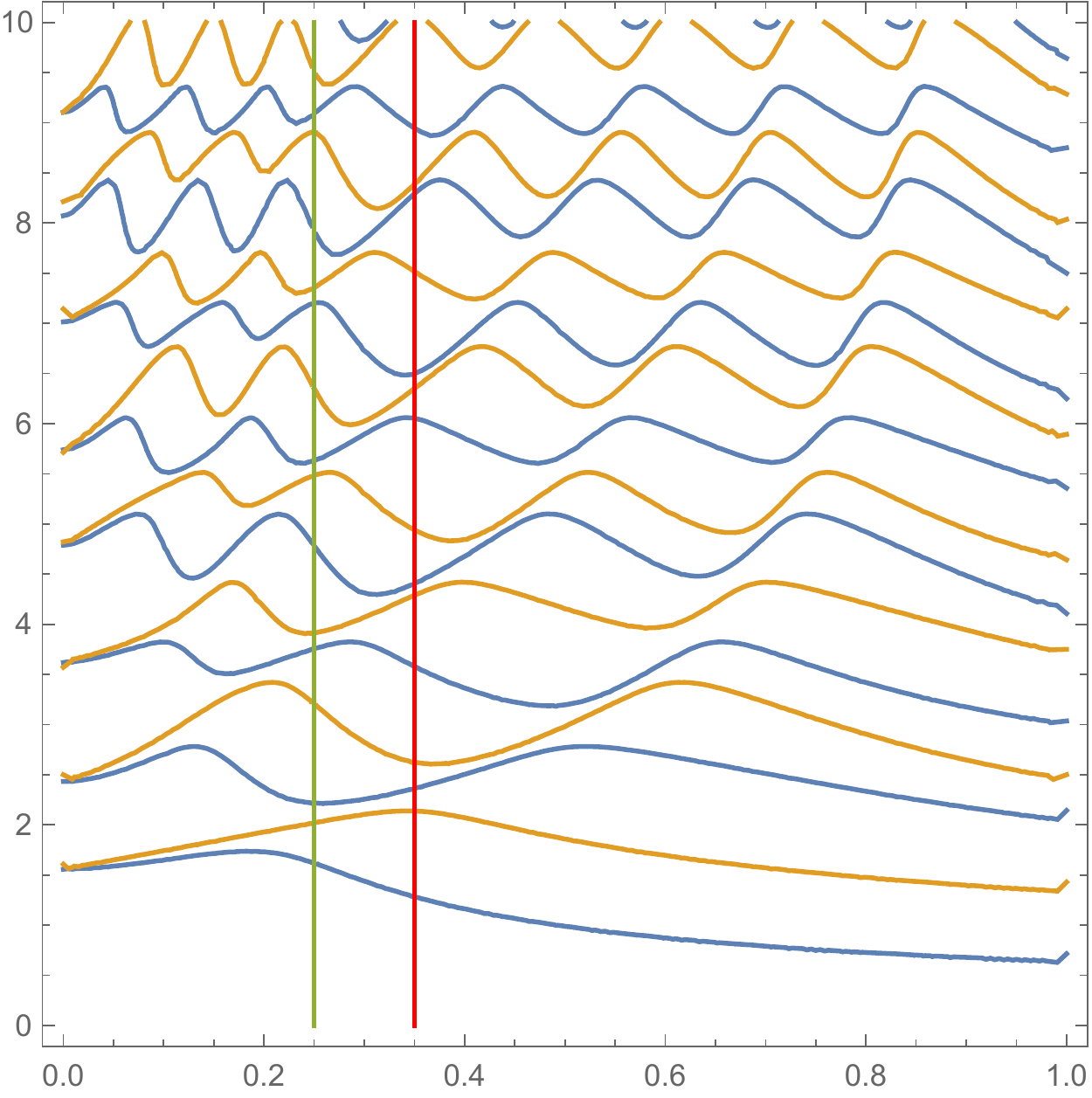}}
\caption{The eigenvalue curves given by Proposition \ref{autoesp} in the $(a, \lambda^{1/4})$-plane, for $\alpha=1/3$ and $\beta=5/2$ (left) and $\alpha=1/3$ and $\beta=3$ (right), in the case $p(0)=\beta$; the green vertical line corresponds to $a=\rho$, the red vertical line corresponds to $a=a_{\text{opt}}$.} %the optimal $a$ in terms of linear stability for the associated two-step density.
\label{autovalori}
\end{figure}

\section{Numerical results for more general densities and conclusions}\label{num} 
In this section we repeat the stability analysis of Section \ref{stab2step} taking into account the possibility of having bang-bang densities with an arbitrary number of jumps; this is done by computing numerically the critical energy thresholds and optimizing with respect to the number of jumps, the position of the piers and the densities of the materials.
\subsection{Numerical computation of the eigenvalues}\label{num_eigen}
We first recall that the eigenvalues and the corresponding eigenfunctions of \eqref{beam0} with $p\equiv 1$ are explicitly known from \cite[Theorem 6]{gazgar}. We normalize such eigenfunctions in $L^2$ and, according to their parity, we denote them by $\eta^e_j(x)$ (even) and $\eta^o_j(x)$ (odd), with corresponding eigenvalues $\Lambda_j^e$ and $\Lambda_j^o$, respectively.\par 
Next, since $\{\eta^e_j, \eta^o_j\}_j$ is a complete system (both in $L^2$ and in $V$), for fixed $a\in (0,1)$ and $p\in P_{\alpha, \beta}$ we expand the eigenfunctions of \eqref{beam0} with respect to such a basis; due to the parity of $p$, also the eigenfunctions of \eqref{beam0} are either even or odd, hence we can expand them as follows:
\begin{equation}
e_{\lambda^e}(x,p)=\sum_{i=1}^{\infty} a_i \eta^e_i(x) \quad \text{(even)} \qquad \text{or} \qquad e_{\lambda^o}(x,p)=\sum_{i=1}^{\infty} b_i\eta^o_i(x)\, \quad \text{(odd)}.
\label{fou}
\end{equation}
We truncate the expressions in \eqref{fou} at a certain order $N$ and we insert 
them into \eqref{beamweak}; testing with $v=\eta_j^e(x)$ or $v=\eta_j^o(x)$, we get the following two linear systems in the unknown $a_i$ and $b_i$:
\begin{equation*}
\left\{
\begin{array}{l}
	a_j\Lambda^e_j=\lambda_j^e \sum\limits_{i=1}^{N} a_i \, \big[\int_Ip(x)\eta^e_i(x)\eta^e_j(x)\,dx\big]\\
	b_j\Lambda^o_j=\lambda_j^o\sum\limits_{i=1}^{N} b_i \, \big[\int_Ip(x)\eta^o_i(x)\eta^o_j(x)\,dx\big],
\end{array}
\quad j=1,\dots,N.
\right.
\end{equation*}
The eigenvalues $\lambda_j^e$ and $\lambda_j^o$ of \eqref{beam0} are then determined numerically by imposing that the determinants associated with the above systems vanish. 
In our numerical scheme, we choose $N=14$; we make this choice in order to be able to properly approximate the eigenvalues of \eqref{beam0} through this procedure, for a general $p$. Indeed, comparing with Proposition \ref{autoesp}, truncating the series at $N=14$ allows us to re-obtain the first twelve eigenvalues of \eqref{beam0} (six even and six odd) for the two-step case (i.e., $p$ as in \eq{defp2} or in \eq{defp22}) with a good level of accuracy.

\subsection{Stability analysis.}
For $a\in (0,1)$ and $0<\alpha <1<\beta$ fixed, we start by optimizing with respect to the density; namely, we consider the problem
\begin{equation}
\label{afix}
\mathcal{E}^*=\mathcal{E}^*(a,\alpha,\beta):=\max_{p\in \overline P}\mathcal{E}(p),
\end{equation}
where $\mathcal{E}$ is, as usual, the instability energy threshold:
\begin{equation}
\label{sogliaenp}
\mathcal{E}(p):=\, \min_{1 \leq j \leq 11} E\big(\lambda_j(p),\lambda_{j+1}(p)\big), 
\end{equation}
$E$ is as given in  \eq{ecritica} and $\overline P$ is a suitable subset of $P_{\alpha, \beta}$, defined according to what proved in Theorem \ref{thm-rapp}. More precisely, we set:
$$
\overline P:=\{p_{1}^{(i)},p_{2}^{(i)},p_{3}^{(i)} \}_{0\leq i\leq 10},
$$ 
where we take $p_1^{(0)} \equiv 1$ and $p_{2}^{(0)}, p_{3}^{(0)}$ equal, respectively, to the weights in \eqref{defp2} and \eqref{defp22}. Then, for $i\geq 1$ and $\ell=1,2,3$, the weights $p^{(i)}=p_{\ell}^{(i)}$ are determined iteratively as follows:
\begin{enumerate}
\item for $i\geq 0$ we numerically compute the eigenvalues of \eqref{beam0} with $p=p^{(i)}$ and we set
$$
\mathcal{E}^{(i)}:=\, \min_{1 \leq j \leq 11} E\big(\lambda_j(p^{(i)}),\lambda_{j+1}(p^{(i)})\big)
$$
and
$$
 g(x,p^{(i)})=	\dfrac{\lambda_{k+1}(p^{(i)})}{\lambda_{k}(p^{(i)})}\,[u_k^2(x,p^{(i)})-u_{k+1}^2(x,p^{(i)})],
$$
where $k=k(i)$ is the index achieving the above minimum, i.e.
$E\big(\lambda_k(p^{(i)}),\lambda_{k+1}(p^{(i)})\big)=\mathcal{E}^{(i)}$;
\item we find $t^{(i)}\in \text{Im}\big(g(x,p^{(i)})\big)$ such that $$I^{(i)}=|g(x,p^{(i)})\geq t^{(i)}|=|I|\dfrac{1-\alpha}{\beta-\alpha}\,;$$
\item we define
$p^{(i+1)}(x):=\beta \chi_{I^{(i)}}(x)+\alpha \chi_{I\setminus I^{(i)}}(x) \in P_{\alpha, \beta}$. 
\end{enumerate}
Intuitively, we proceed by subsequent approximations by trying, at each step, to approach a density maximizing the  ``dangerous'' ratio of eigenvalues, the one corresponding to the instability threshold \eq{sogliaenp}. We stop the algorithm at $i=10$ since, after 10 iterations, the weights seem to become recurrent, independently of the choice of $p^0=p_{\ell}^{(0)}$ in the above classes. 
\par
By computing the eigenvalues of \eqref{beam0} as explained in Section \ref{num_eigen}, we solve problem \eq{afix} numerically and we collect the obtained results in Tables \ref{tab1}-\ref{tab3} for different choices of $\alpha, \beta$ and $a$. More precisely, 
for $a$ belonging to the set $A=\{0.1, 0.2, 0.3, 0.35, 0.4, 0.45, \ldots, 0.65, 0.7, 0.8, 0.9\}$ defined in Section \ref{stab2step} and for selected couples $(\alpha, \beta) \in \Sigma_1 \times \Sigma_2$ (see again Section \ref{stab2step}), in Tables \ref{tab1} and \ref{tab2} we report the value of $\mathcal{E}^*$ defined in \eqref{afix} and the graph of the corresponding maximizer $p^*$ (clearly, $p^*=p^*(x,a,\alpha, \beta)$), together with the number $N^*$ of discontinuities of $p^*$ for $x > 0$ (by symmetry, this number is the same for $x<0$). %As for the choice of $a$, as in Section \ref{stab2step} we take $a \in A=\{0.1, 0.2, 0.3, 0.35, 0.4, 0.45, \ldots, 0.65, 0.7, 0.8, 0.9\}$. On the other hand, 
We show in details the results for the choices $(\alpha, \beta)=(1/2,3/2), (\alpha, \beta)=(1/2, 2)$ (Table \ref{tab1}) and $(\alpha, \beta)=(1/3,3/2), (\alpha, \beta)=(1/3, 3)$ (Table \ref{tab2}), in order not to overload the contents; for other couples $(\alpha, \beta) \in \Sigma_1 \times \Sigma_2$, the results appear similar (see also Table \ref{tab3} and Figure \ref{plotene}).
We also notice that in Tables \ref{tab1} and \ref{tab2} the critical energy thresholds (recall definition \eqref{sogliaenp}) are always attained in correspondence of the ratio of the first two eigenvalues ($\lambda_2/\lambda_1$), except for few cases  in which they are attained by $\lambda_3/\lambda_2$, pointed out in the tables through the symbol~\dag. 
\par
By taking the maximum of $\mathcal{E}^*$ with respect to $a$, we then obtain the solution of the problem
$$
\max_{a\in A}\mathcal{E}^*(a,\alpha,\beta)=\mathcal{E}^*(a_{opt},\alpha,\beta),
$$
where $a_{opt}=a_{opt}(\alpha,\beta)$ denotes the corresponding maximizer. In Table \ref{tab3}, we give an overview of the optimal position $a_{opt}$ of the piers for all the densities $(\alpha, \beta) \in \Sigma_1 \times \Sigma_2$, reporting the associated energy threshold of instability $\mathcal{E}^*$ (together with the corresponding ratio of eigenvalues $R$) and the number $N^*$ of discontinuities of the optimal weight, which allows to reconstruct the qualitative shape of the optimal density $p^*$.
\smallbreak
From Tables \ref{tab1}-\ref{tab3} we draw the following conclusions:
\begin{enumerate}[a)] 
\item comparing with the homogeneous case $p \equiv 1$ \cite[Table 3.9]{GarGazBK}, for every $a$ the non homogeneity gives rise, in general, to higher energy thresholds of instability; 
\item for fixed $a$ in the ``physical range'' \eqref{physrangea},
with few exceptions (for $\alpha$ close to $1$), the optimal density $p^*$ has 4 jumps for $x > 0$, hence it is not of the kind examined in Section \ref{2step};
\item on the contrary, if the pier is placed next to the center of the beam ($a < 0.30$) or next to the endpoints  ($a > 0.70$) and $\beta$ is sufficiently large, then the optimal density is two-step; in the former case, the beam is reinforced in the middle, while in the latter it is reinforced next to the endpoints; 
\item  in most cases, the heavier material has to be located around the piers, with very few exceptions, always for $a$ outside the range \eq{physrangea};
\item the discontinuities of $p^*$ are located a little before and a little after the piers, and a little before reaching the endpoints (namely, the reinforces around the pier and around the endpoints take place along small lengths); 
\item in most cases, we have
$$
a_\text{opt}=0.50\,,
$$
namely the optimal position of the piers is mostly the one for which \emph{the length of the central span is twice the length of the lateral spans} (a fact also observed, e.g., in \cite{Gar, GarGazBK}). More in general, the best performances are always obtained for $a \in \{0.50, 0.55\}$. 
\end{enumerate}
Finally, a ``visual'' summary of all the results is provided in Figure \ref{plotene}, where we plot the values of $\mathcal{E}^*=\mathcal{E}^*(a,\alpha,\beta)$ versus $a$, for $(\alpha, \beta) \in \Sigma_1 \times \Sigma_2$. The picture confirms that the best location of piers is at $a=0.50$; furthermore, we infer that  %i.e the length of the central span must be twice the length of the lateral spans; furthermore we infer that
$$
\max_{(a,\alpha,\beta)\in A\times \Sigma_1 \times \Sigma_2} \mathcal{E}^*(a,\alpha,\beta)=\mathcal{E}^*(0.50,1/3,3)\,.
$$

\clearpage
\begin{table}[h!]\centering
	\resizebox{\textwidth}{!}{\begin{tabular}{|c||c|c|c||c|c|c|}
			\hline
			$a$& \multicolumn{3}{c||}{	$\alpha=1/2$, $\beta =3/2$} & \multicolumn{3}{c|}{$\alpha=1/2$, $\beta =2$} \\	\hline
			&$\mathcal{E}^*/10^2$&$N^*$&$p^*(x)$&$\mathcal{E}^*/10^2$&$N^*$&$p^*(x)$\rule[2mm]{0mm}{4mm} \\
			\hline
			\hline
			0.10&$\dfrac{4.27}{10^{2}}$&2& \parbox[c]{65mm}{\includegraphics[scale=0.018]{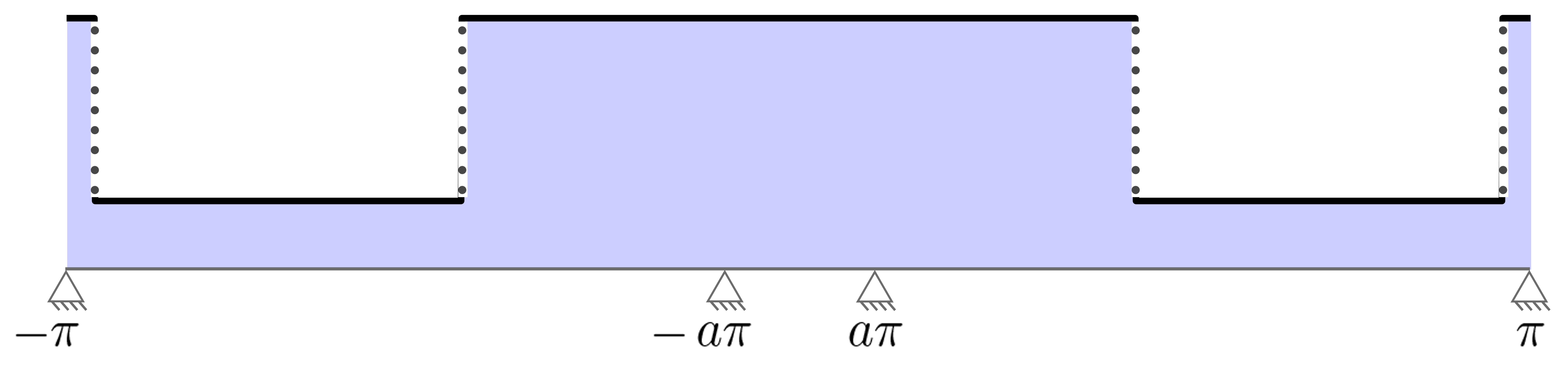}}&$\dfrac{6.50}{10^{2}}$&1& \parbox[c]{65mm}{\includegraphics[scale=0.018]{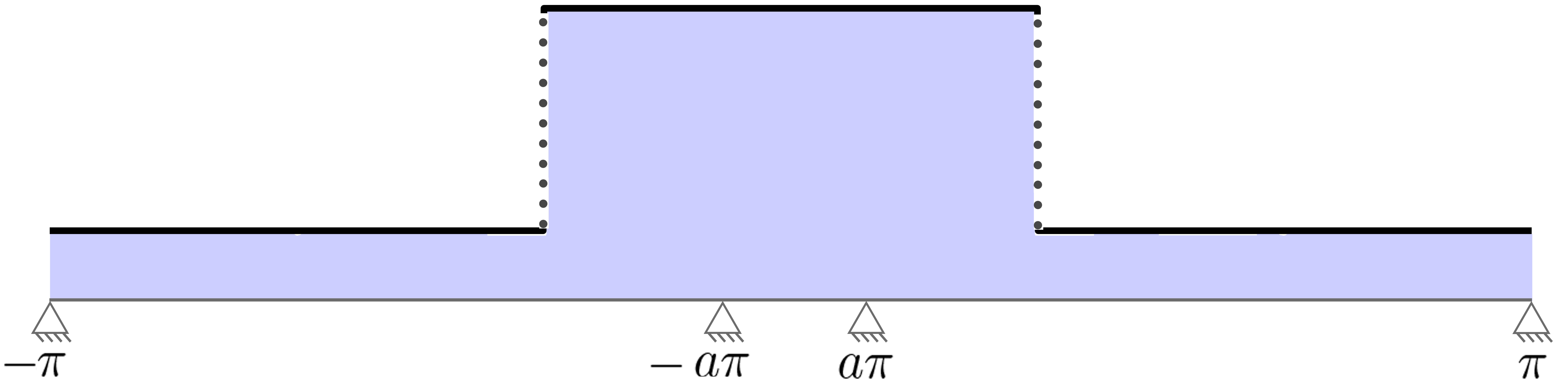}}
			\rule[2mm]{0mm}{1cm} \\
			0.20&$\dfrac{1.48}{10}$&2& \parbox[c]{65mm}{\includegraphics[scale=0.018]{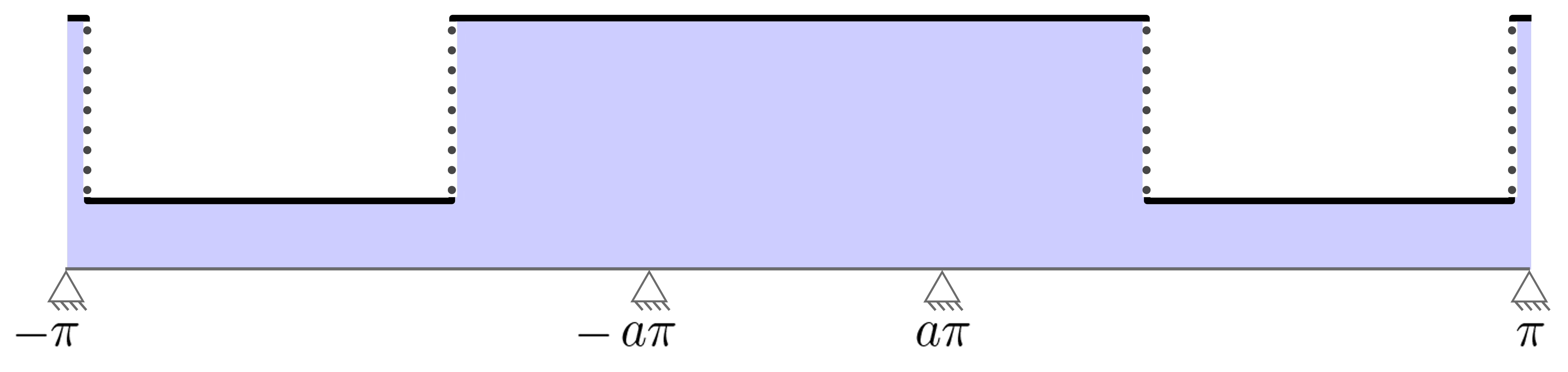}}&$\dfrac{2.25}{10}$&1& \parbox[c]{65mm}{\includegraphics[scale=0.018]{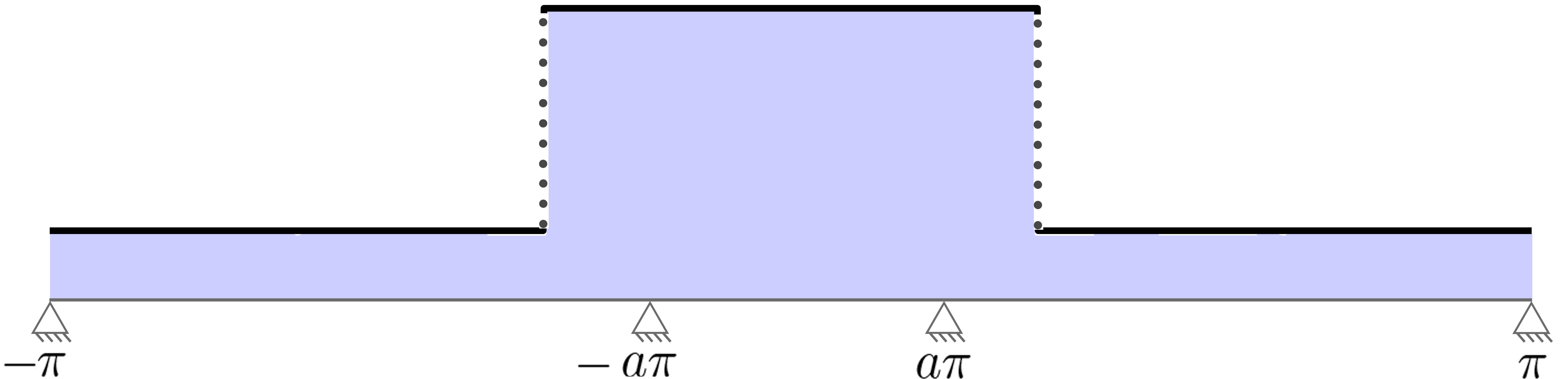}}
			\rule[2mm]{0mm}{0.8cm} \\
			0.30&$\dfrac{9.05}{10}$&2& \parbox[c]{65mm}{\includegraphics[scale=0.018]{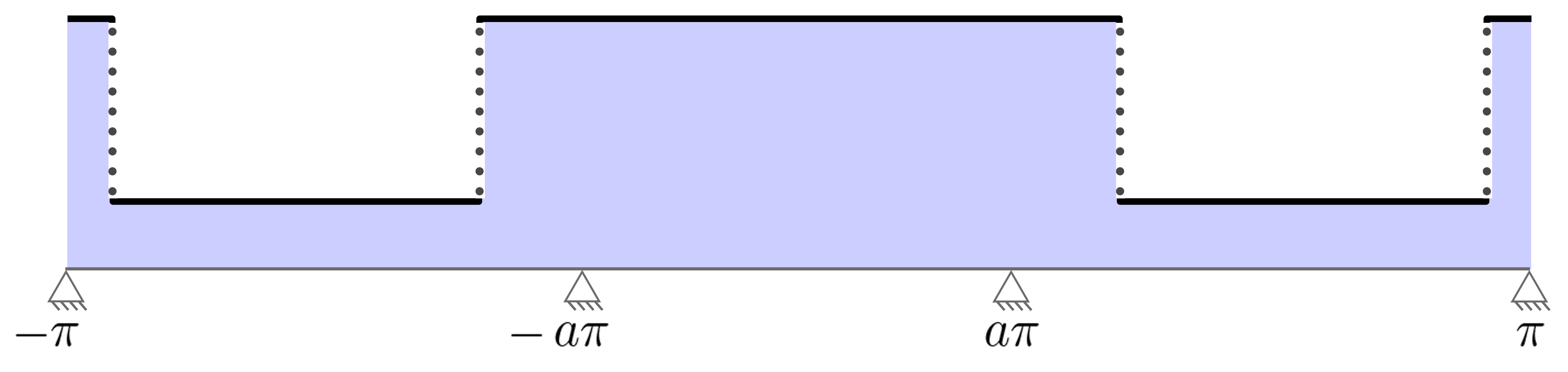}}&1.16&1& \parbox[c]{65mm}{\includegraphics[scale=0.018]{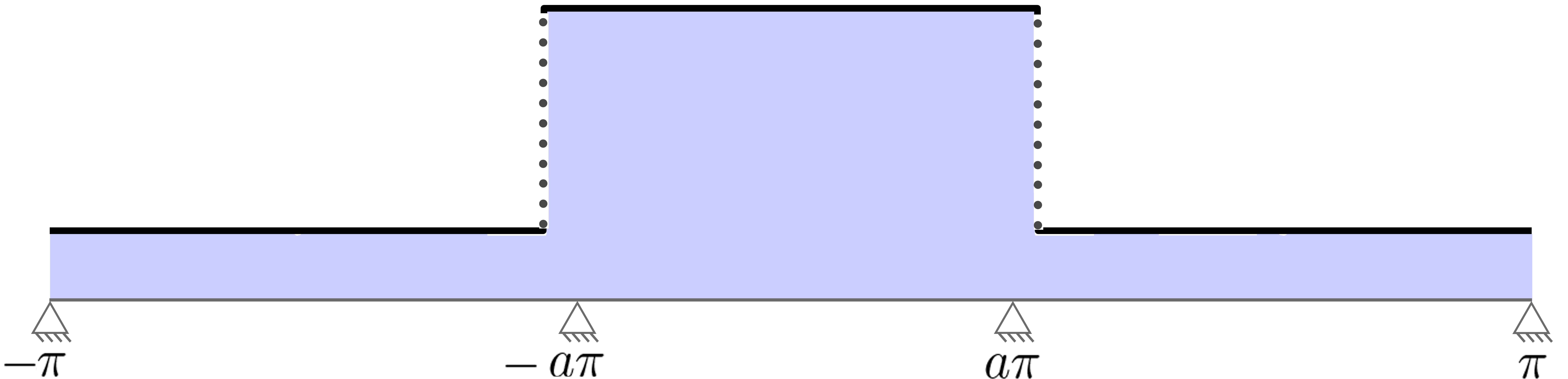}}
			\rule[2mm]{0mm}{0.8cm} \\
			0.35&1.47$^\dag$ &2& \parbox[c]{65mm}{\includegraphics[scale=0.018]{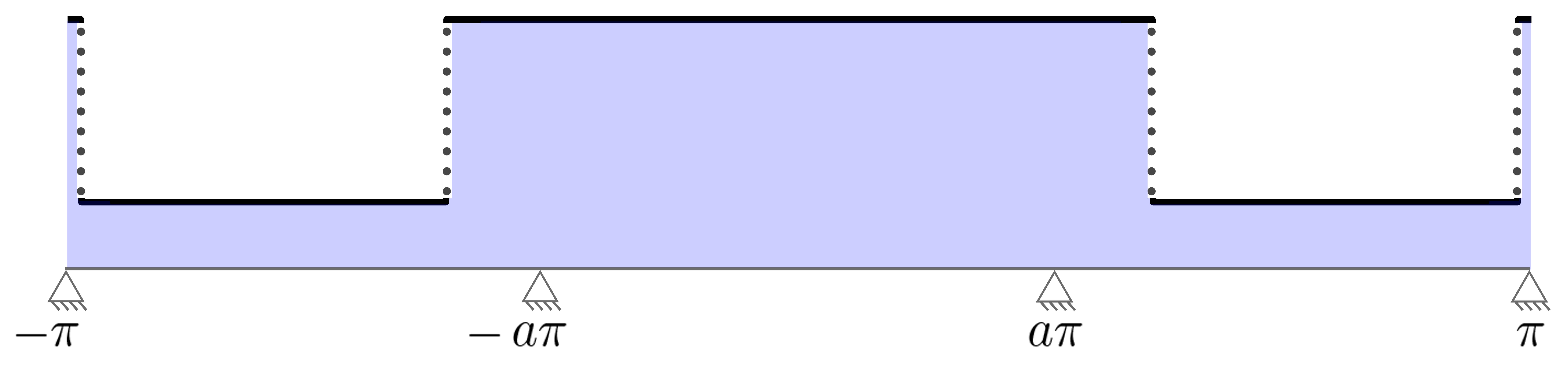}}&1.30$^\dag$&3& \parbox[c]{65mm}{\includegraphics[scale=0.018]{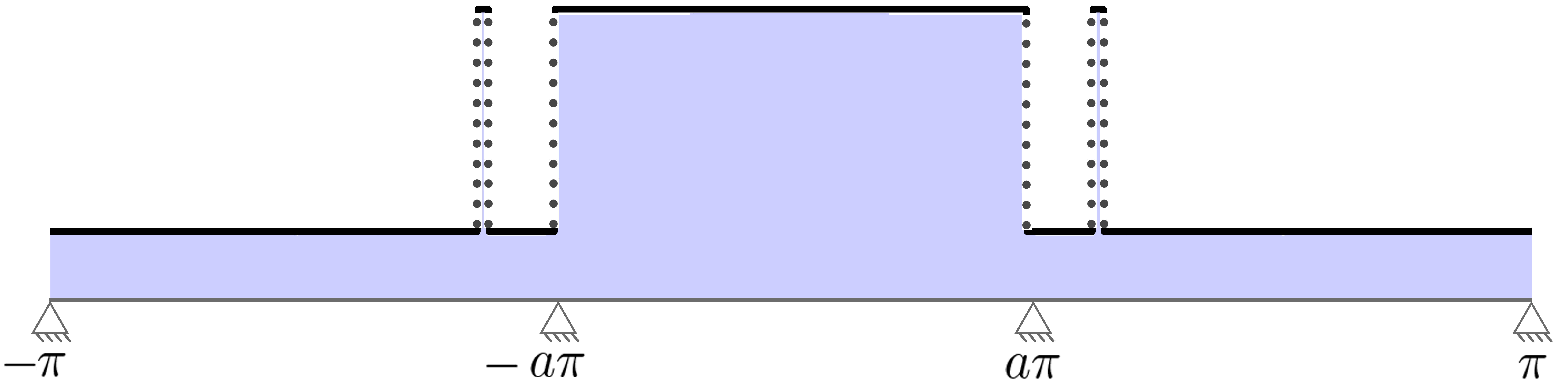}}
			\rule[2mm]{0mm}{0.8cm} \\
			0.40&2.37$^\dag$&2& \parbox[c]{65mm}{\includegraphics[scale=0.018]{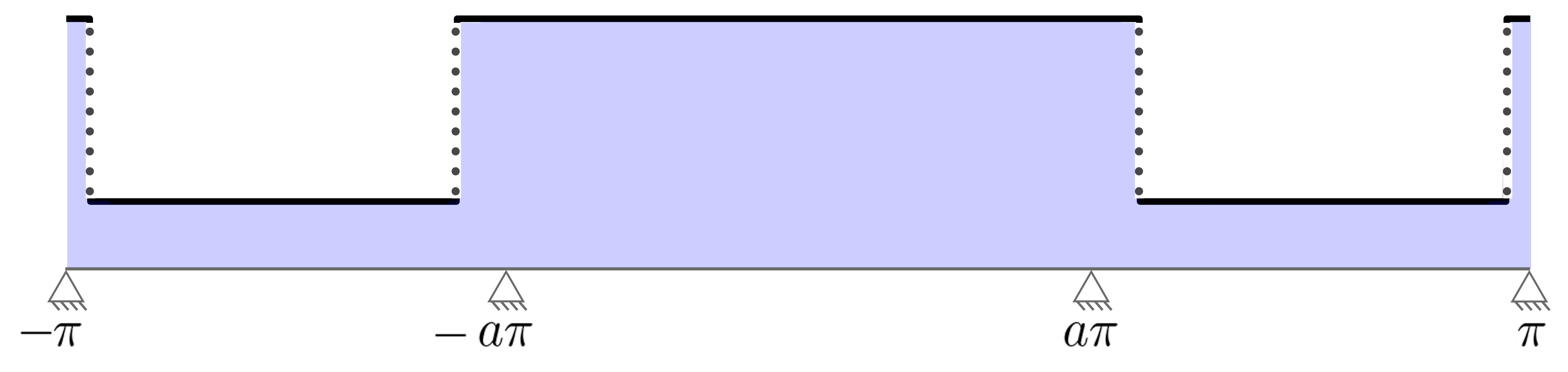}}&2.67$^\dag$&1& \parbox[c]{65mm}{\includegraphics[scale=0.018]{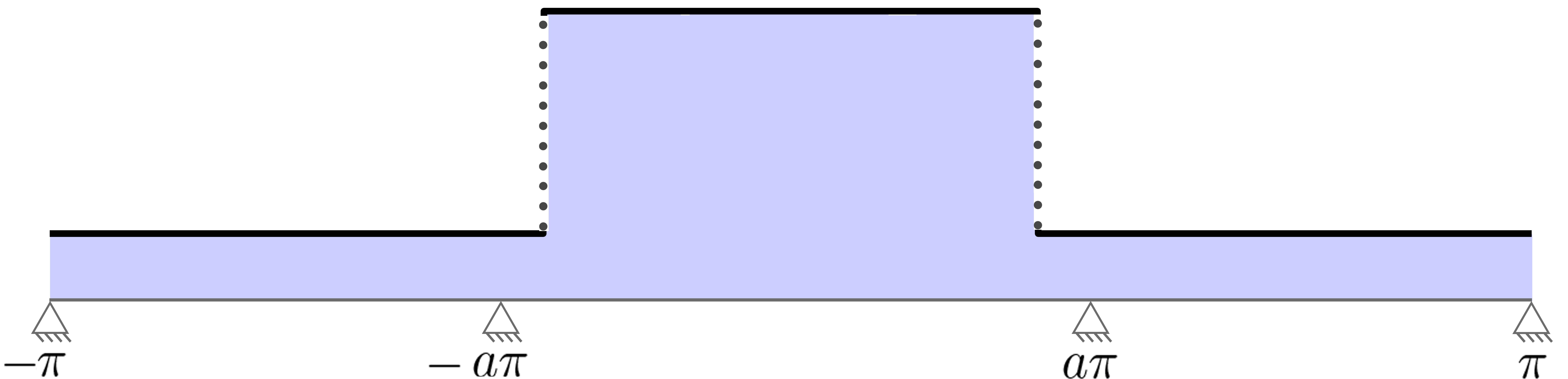}}
			\rule[2mm]{0mm}{0.8cm} \\
			0.45&3.84$^\dag$&4& \parbox[c]{65mm}{\includegraphics[scale=0.018]{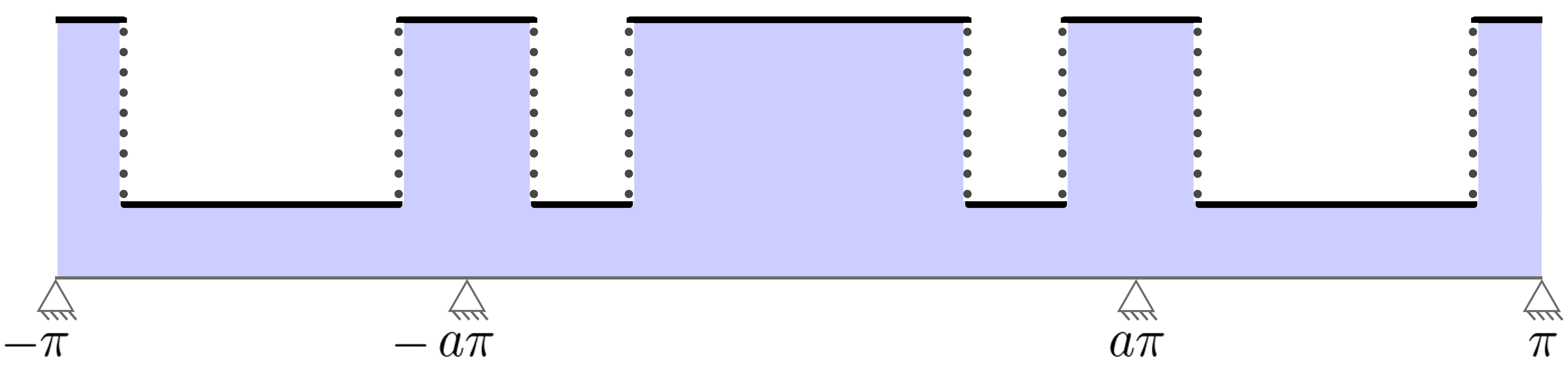}}&4.33$^\dag$&4& \parbox[c]{65mm}{\includegraphics[scale=0.018]{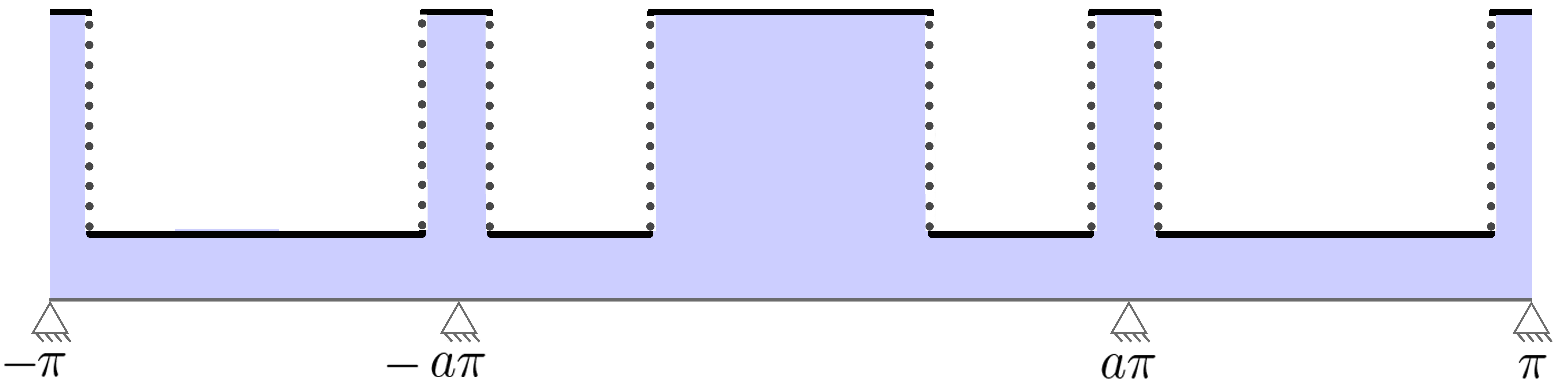}}
			\rule[2mm]{0mm}{0.8cm} \\
			0.50&4.54&4& \parbox[c]{65mm}{\includegraphics[scale=0.018]{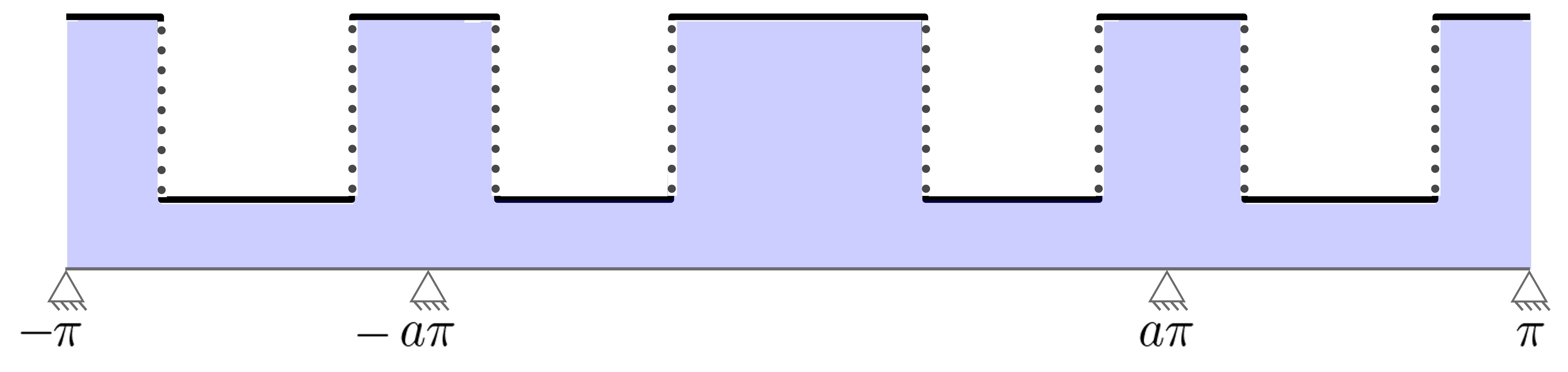}}&5.66&4& \parbox[c]{65mm}{\includegraphics[scale=0.018]{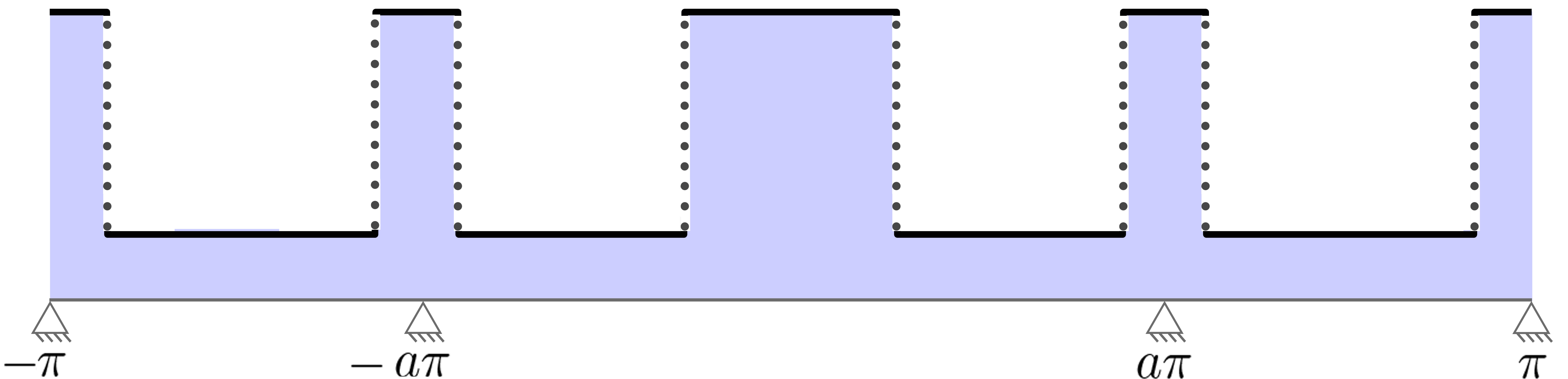}}
			\rule[2mm]{0mm}{0.8cm} \\
			0.55&3.98&4& \parbox[c]{65mm}{\includegraphics[scale=0.018]{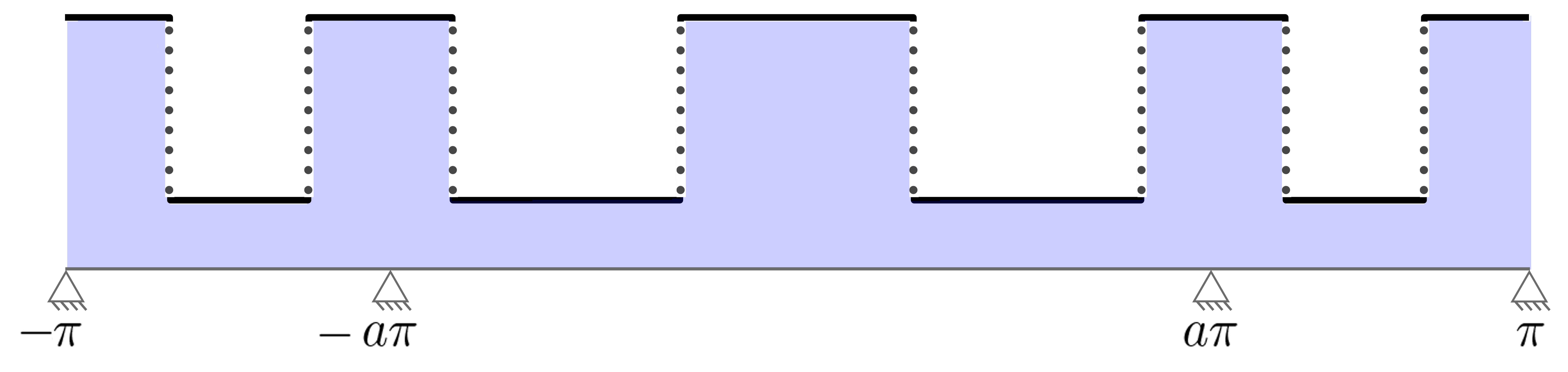}}&4.65&4& \parbox[c]{65mm}{\includegraphics[scale=0.018]{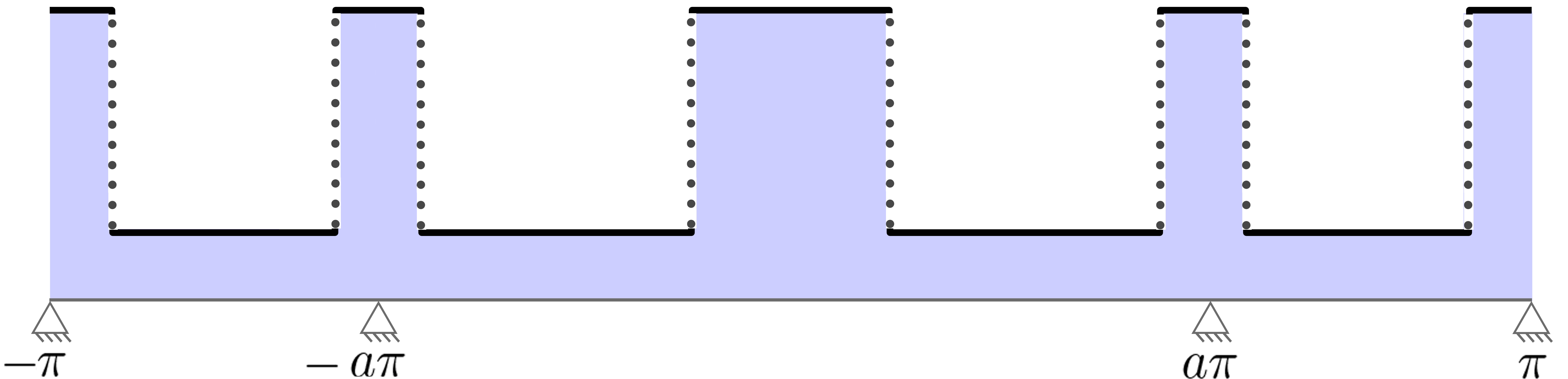}}
			\rule[2mm]{0mm}{0.8cm} \\
			0.60&3.06&1& \parbox[c]{65mm}{\includegraphics[scale=0.018]{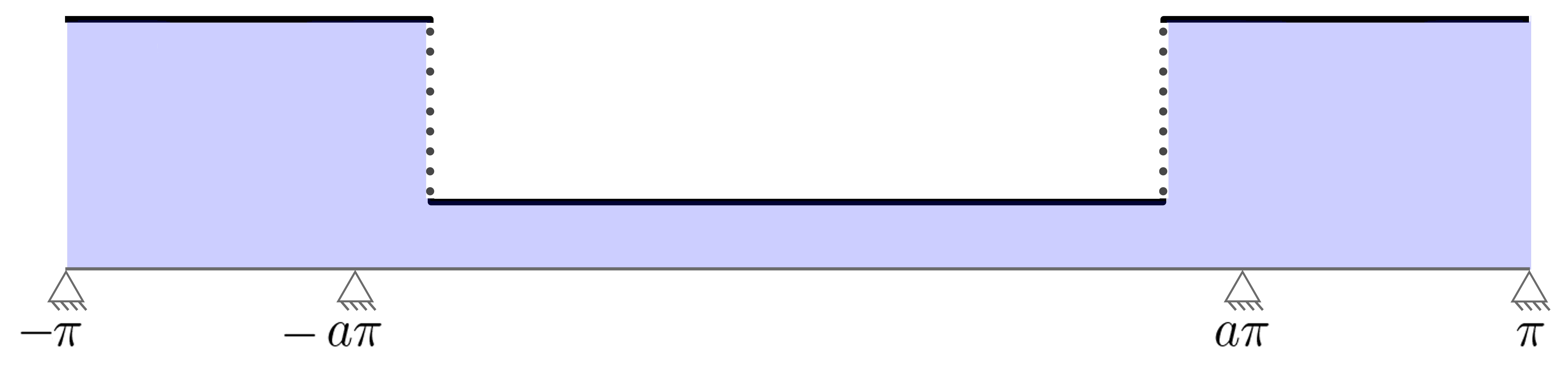}}&3.12&4& \parbox[c]{65mm}{\includegraphics[scale=0.018]{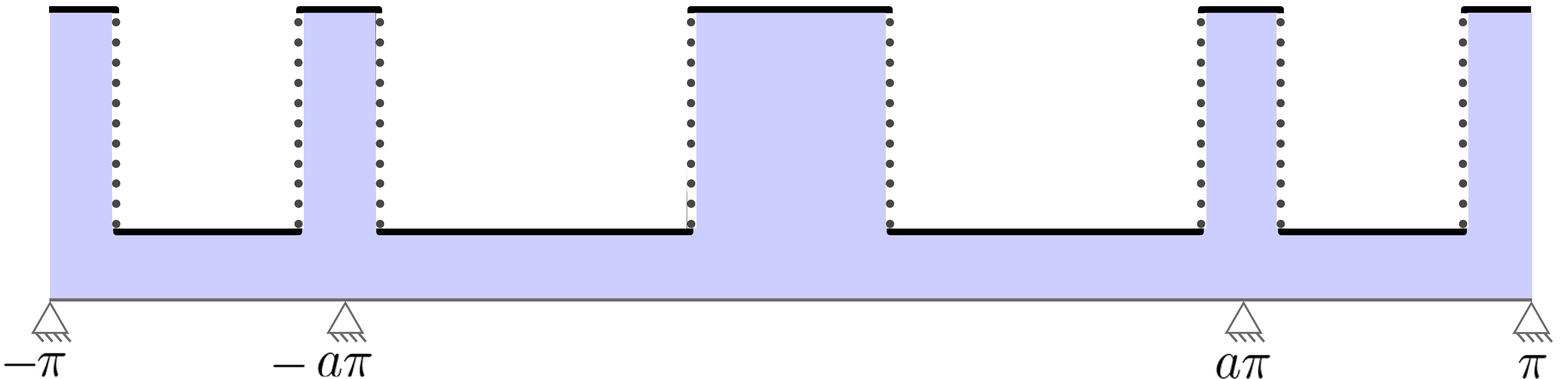}}
			\rule[2mm]{0mm}{0.8cm} \\
			0.65&2.12&1& \parbox[c]{65mm}{\includegraphics[scale=0.018]{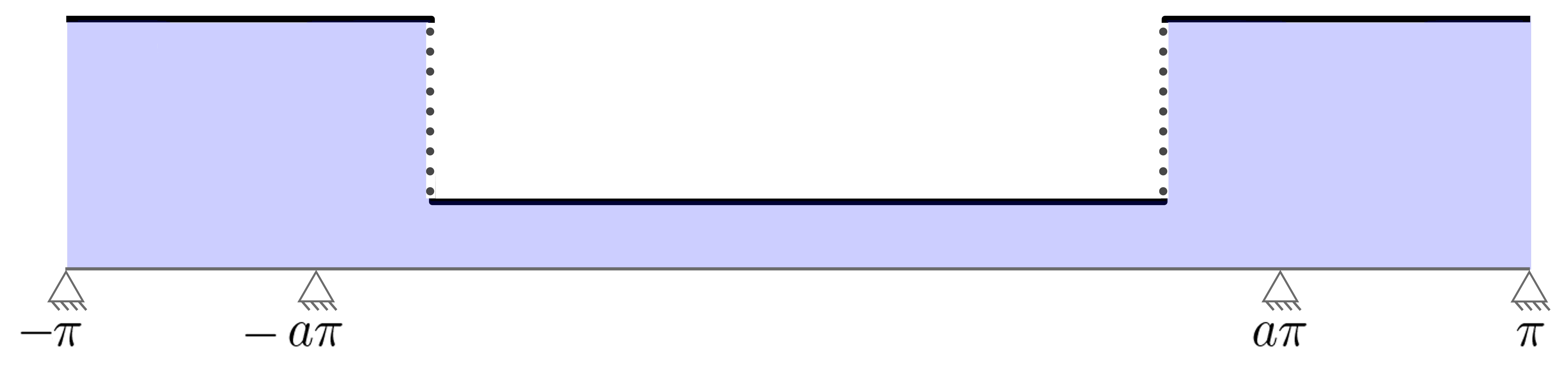}}&2.38&1& \parbox[c]{65mm}{\includegraphics[scale=0.018]{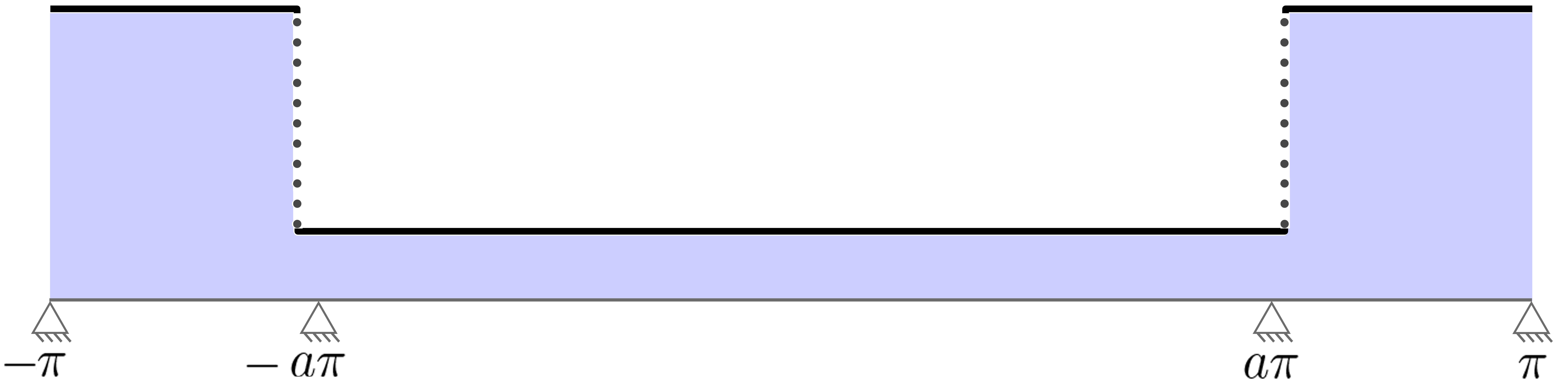}}
			\rule[2mm]{0mm}{0.8cm} \\
			0.70&1.27&4& \parbox[c]{65mm}{\includegraphics[scale=0.018]{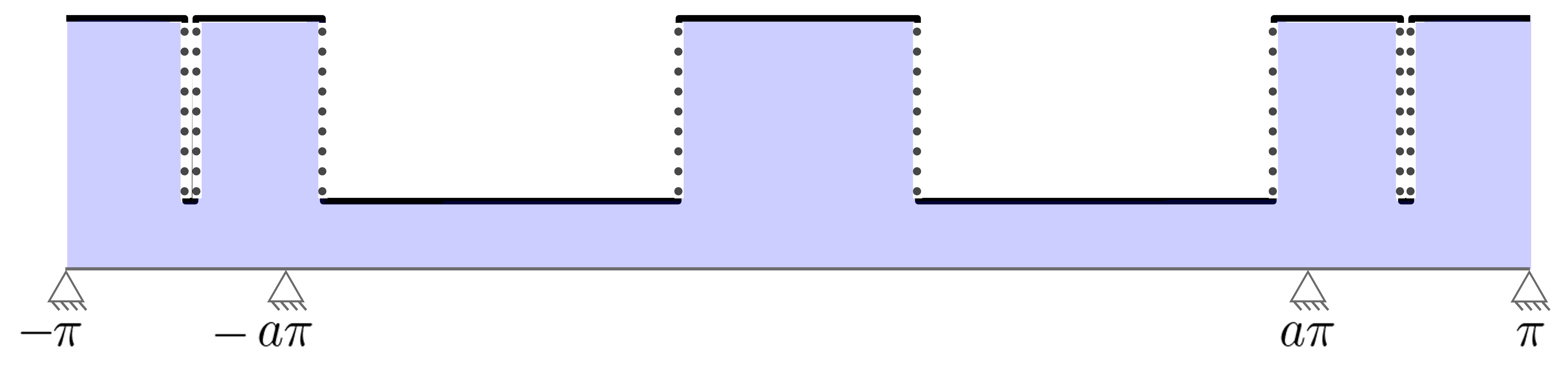}}&1.67&1& \parbox[c]{65mm}{\includegraphics[scale=0.018]{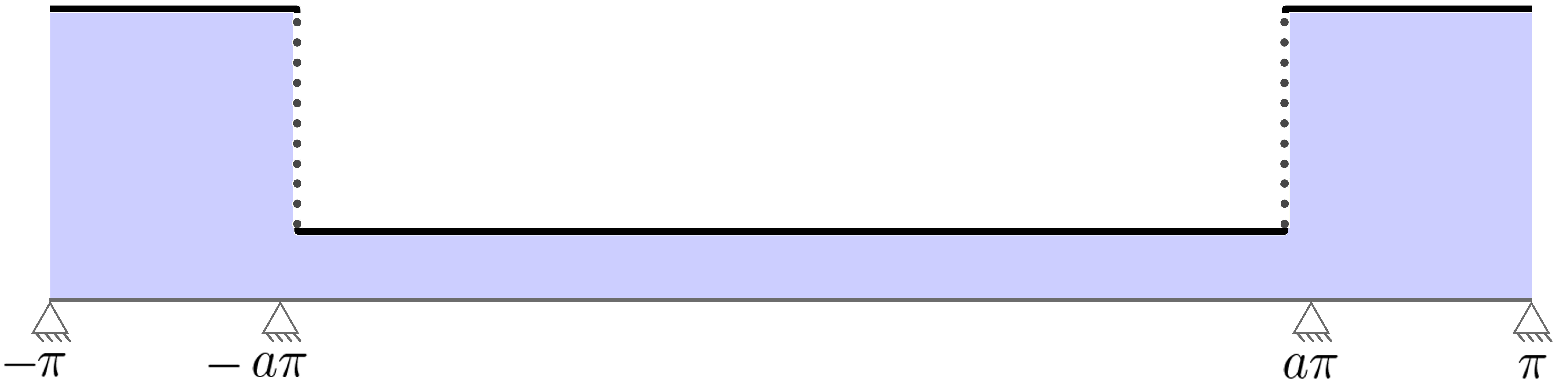}}
			\rule[2mm]{0mm}{0.8cm} \\
			0.80&$\dfrac{5.35}{10}$&2& \parbox[c]{65mm}{\includegraphics[scale=0.018]{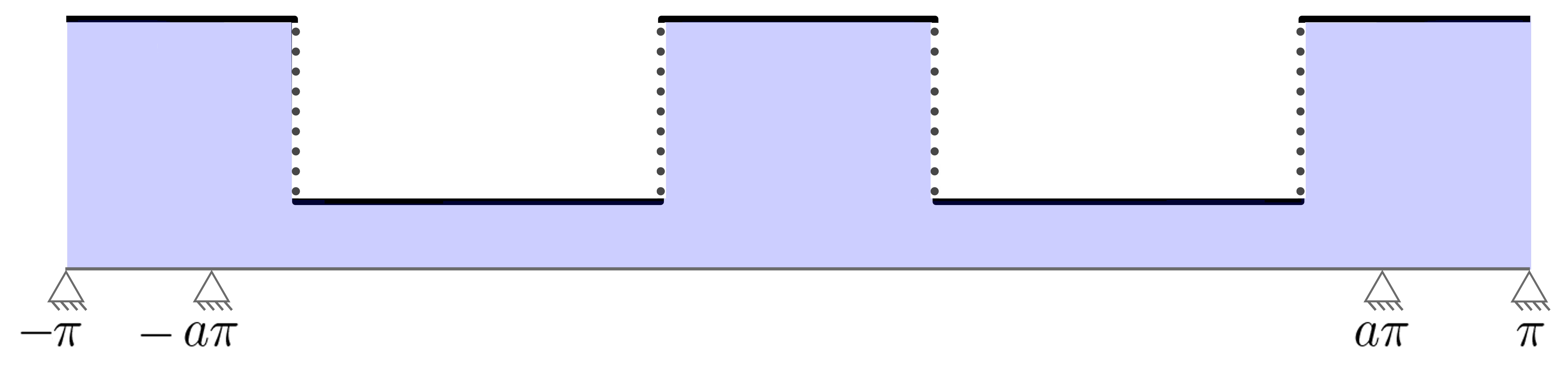}}&$\dfrac{7.05}{10}$&1& \parbox[c]{65mm}{\includegraphics[scale=0.018]{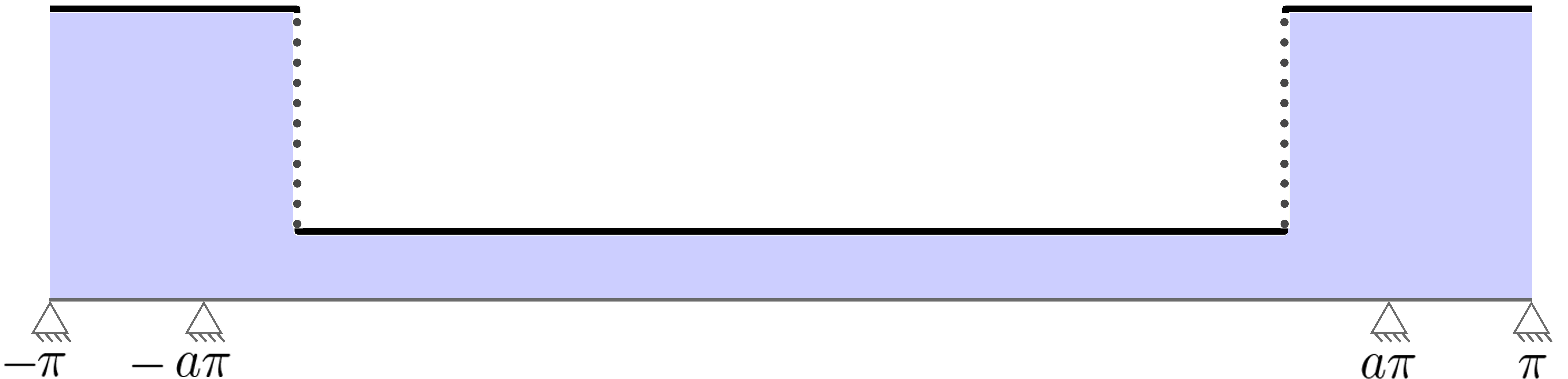}}
			\rule[2mm]{0mm}{0.8cm} \\
			0.90&$\dfrac{2.42}{10}$&2& \parbox[c]{65mm}{\includegraphics[scale=0.018]{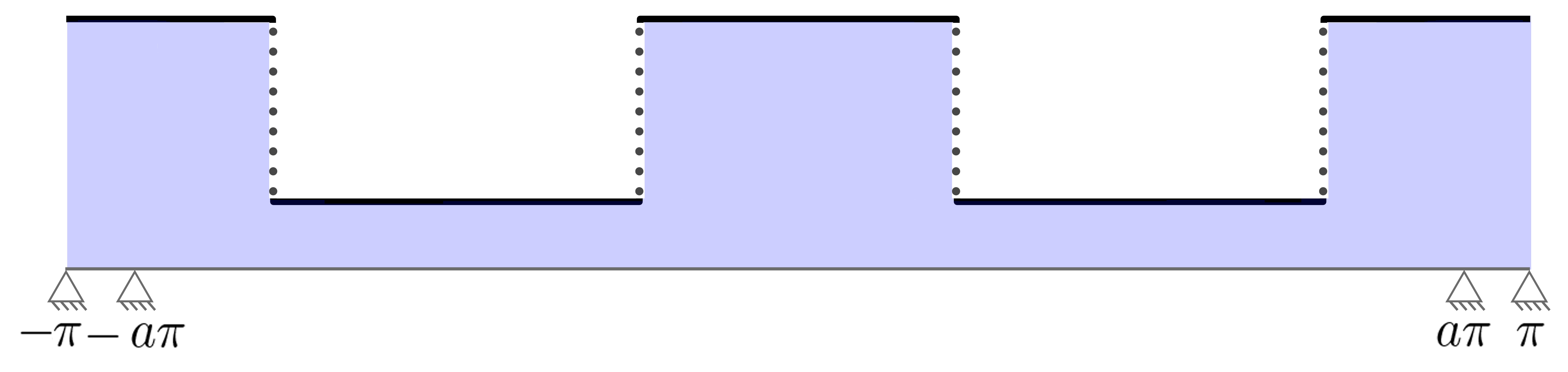}}&
			$\dfrac{2.87}{10}$&1& \parbox[c]{65mm}{\includegraphics[scale=0.018]{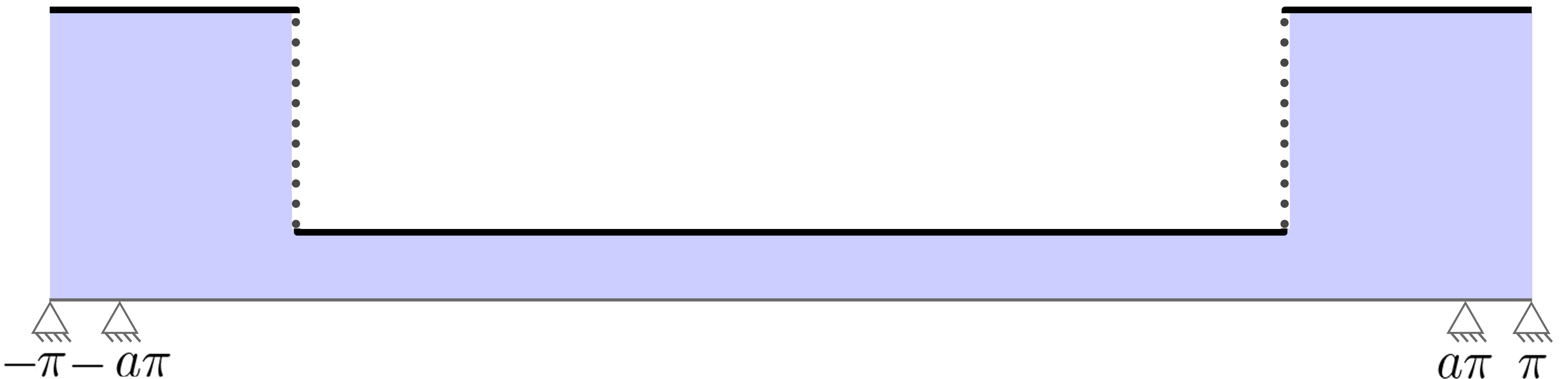}}
			\rule[2mm]{0mm}{0.8cm} \\
			\hline
	\end{tabular}} 
		
	\vspace{3mm}
	\caption{For different values of $\alpha$, $\beta$ and $a$, the numerical values of the critical energy threshold in \eqref{afix}, the graph of the corresponding piecewise constant maximizer $p^*(x)$ and the number $N^*$ of its discontinuities for $x > 0$.}
	\label{tab1}
\end{table}

\clearpage
\begin{table}[htbp!]\centering
	\resizebox{\textwidth}{!}{\begin{tabular}{|c||c|c|c||c|c|c|}
			\hline
			$a$& \multicolumn{3}{c||}{	$\alpha=1/3$, $\beta =3/2$} & \multicolumn{3}{c|}{$\alpha=1/3$, $\beta =3$} \\	\hline
			&$\mathcal{E}^*/10^2$&$N^*$&$p^*(x)$&$\mathcal{E}^*/10^2$&$N^*$&$p^*(x)$\rule[2mm]{0mm}{4mm} \\
			\hline
			\hline
			0.10&$\dfrac{6.11}{10^{2}}$&2& \parbox[c]{65mm}{\includegraphics[scale=0.018]{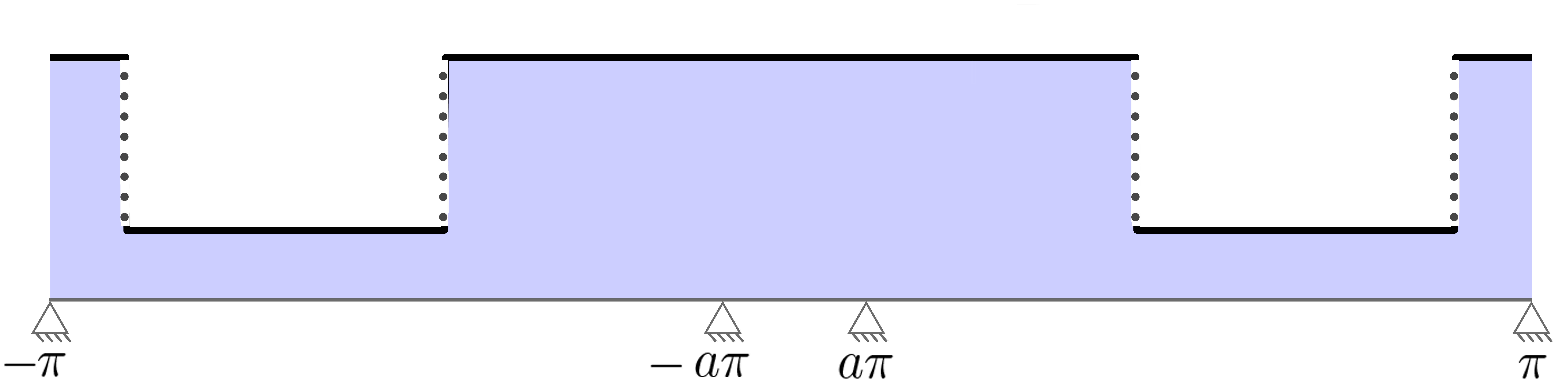}}&$\dfrac{1.66}{10}$&1& \parbox[c]{65mm}{\includegraphics[scale=0.018]{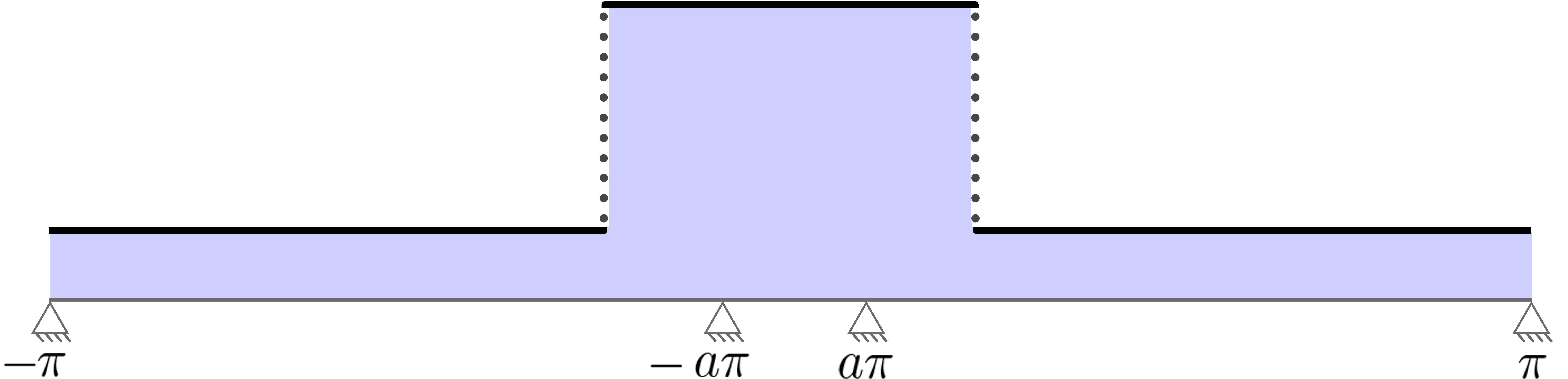}}
			\rule[2mm]{0mm}{1cm} \\
			0.20&$\dfrac{2.38}{10}$&2& \parbox[c]{65mm}{\includegraphics[scale=0.018]{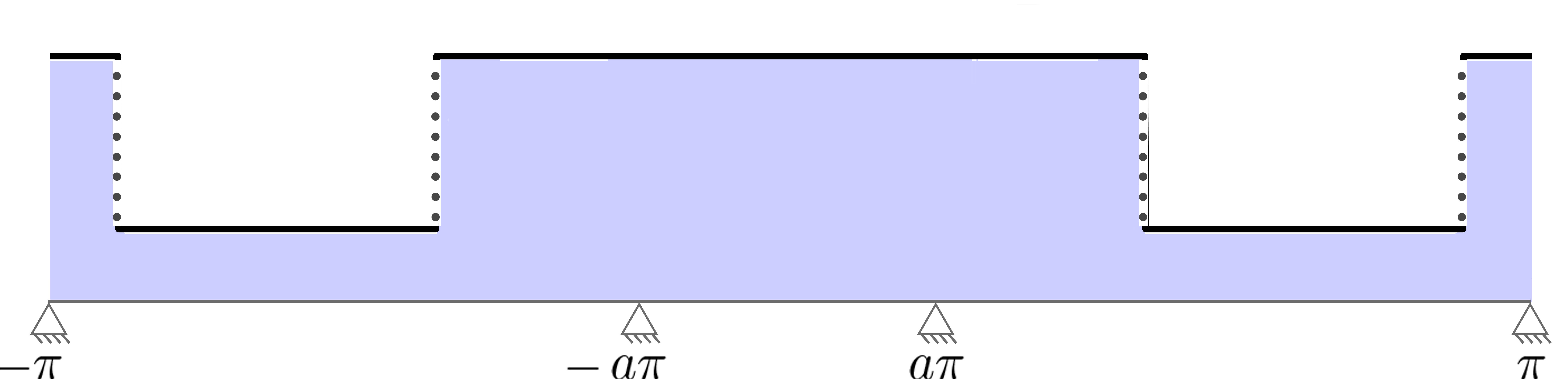}}&$\dfrac{6.76}{10}$&1& \parbox[c]{65mm}{\includegraphics[scale=0.018]{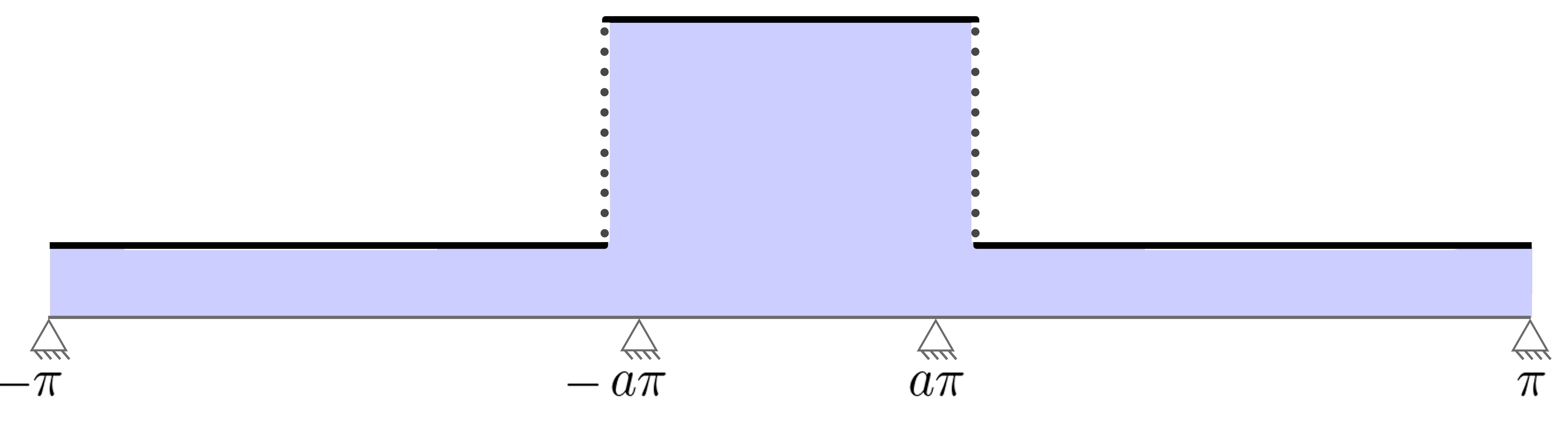}}
			\rule[2mm]{0mm}{0.8cm} \\
			0.30&1.59&2& \parbox[c]{65mm}{\includegraphics[scale=0.018]{p_03.jpg}}&1.65$^\dag$&1& \parbox[c]{65mm}{\includegraphics[scale=0.018]{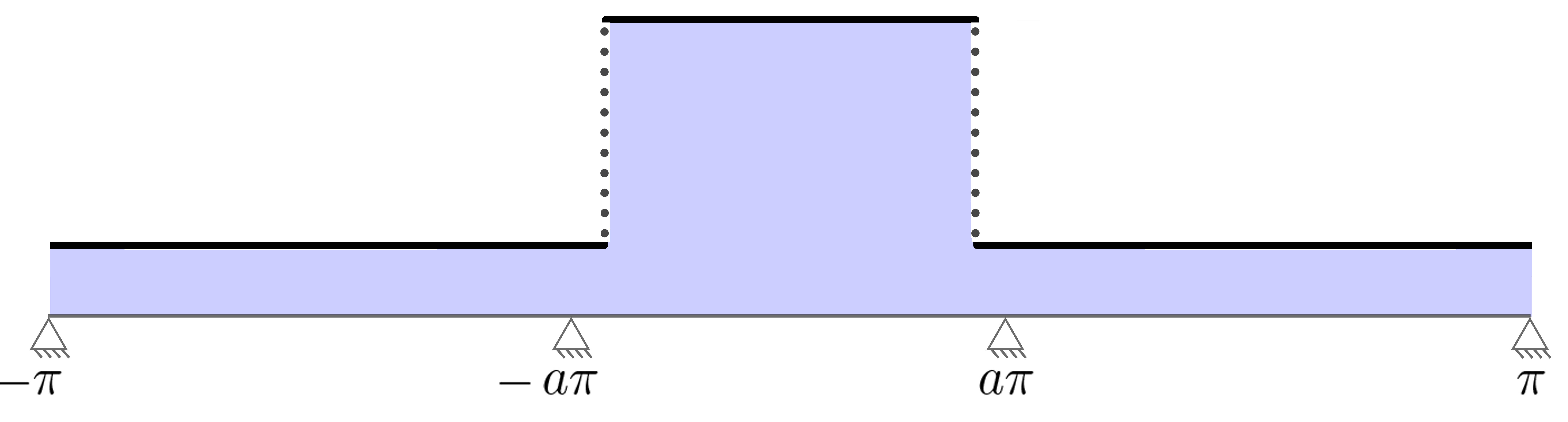}}
			\rule[2mm]{0mm}{0.8cm} \\
			0.35&2.49$^\dag$&2& \parbox[c]{65mm}{\includegraphics[scale=0.018]{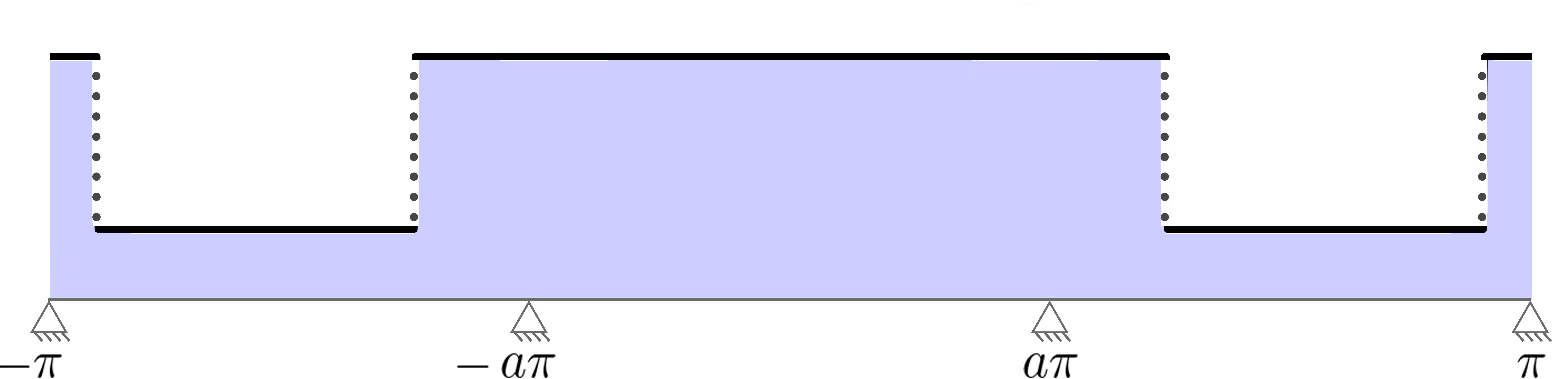}}&3.25$^\dag$&1& \parbox[c]{65mm}{\includegraphics[scale=0.018]{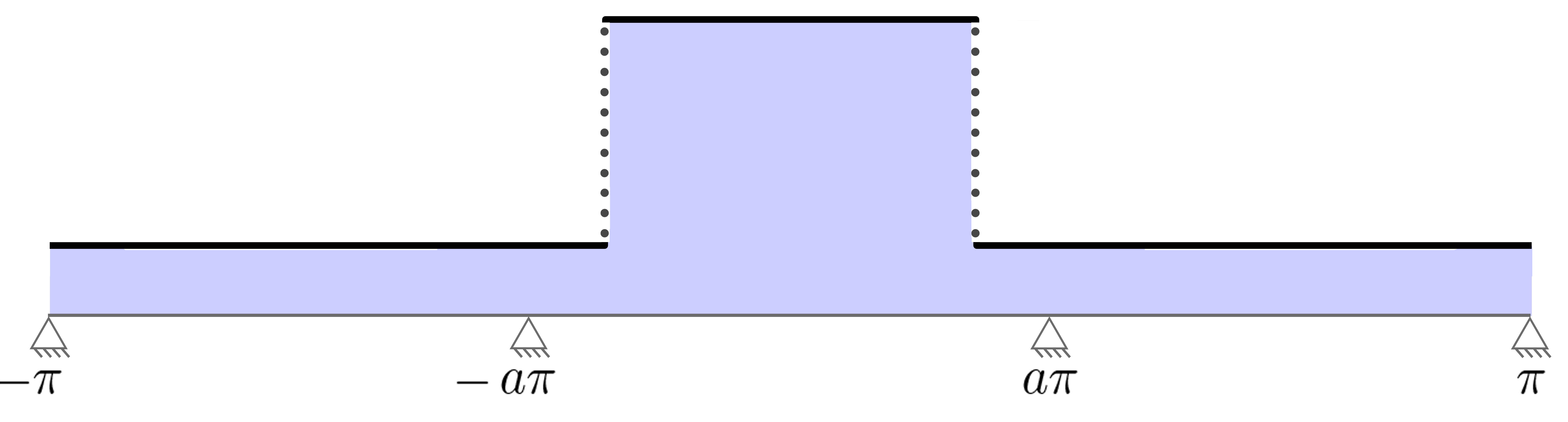}}
			\rule[2mm]{0mm}{0.8cm} \\
			0.40&4.66&2& \parbox[c]{65mm}{\includegraphics[scale=0.018]{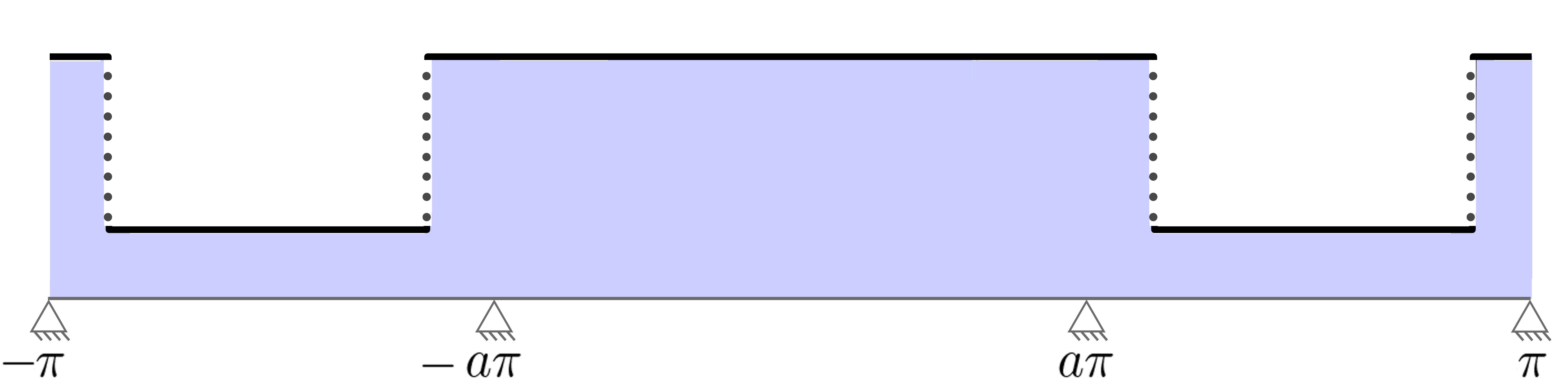}}&5.53$^\dag$&4& \parbox[c]{65mm}{\includegraphics[scale=0.018]{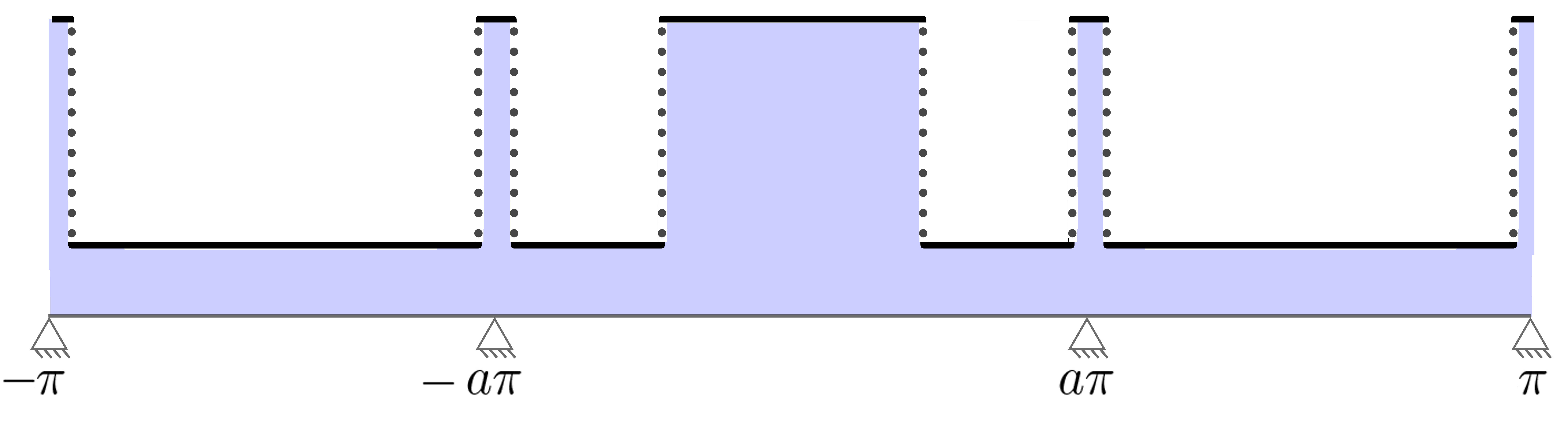}}
			\rule[2mm]{0mm}{0.8cm} \\
			0.45&4.79$^\dag$&4& \parbox[c]{65mm}{\includegraphics[scale=0.018]{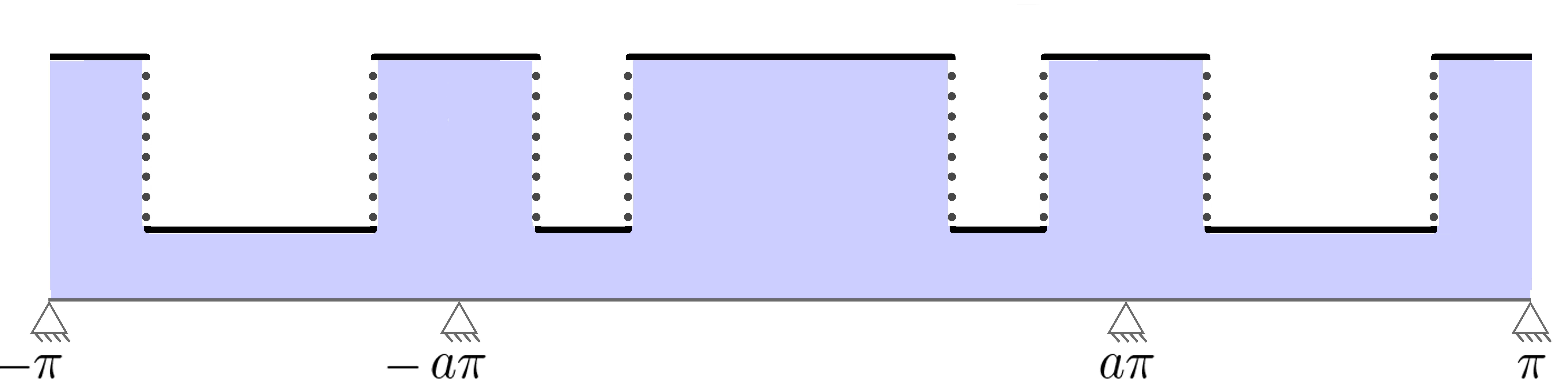}}&8.38$^\dag$&4& \parbox[c]{65mm}{\includegraphics[scale=0.018]{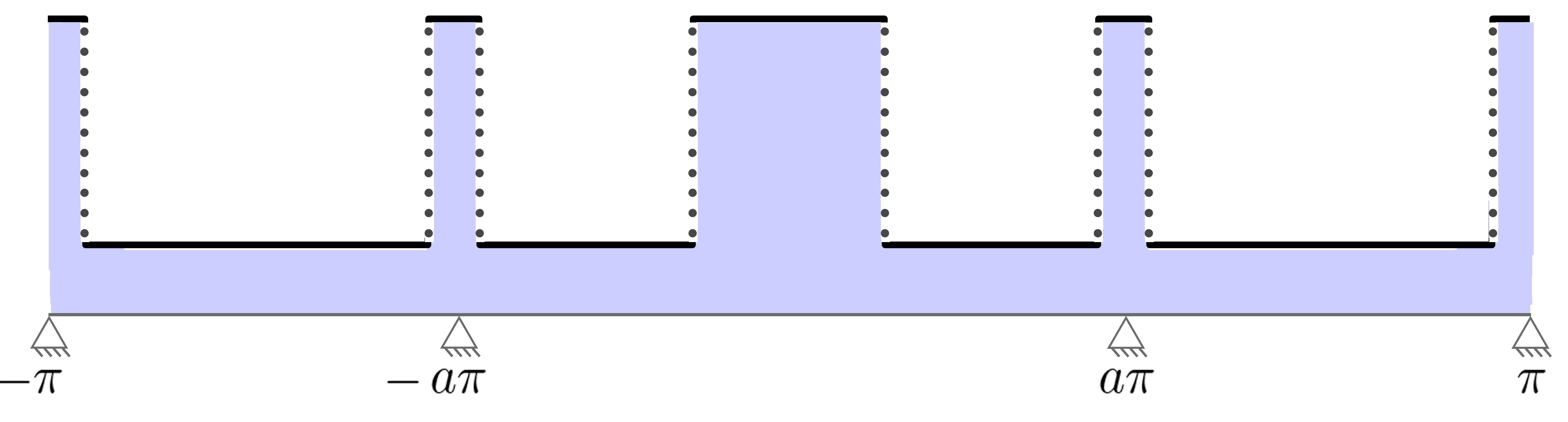}}
			\rule[2mm]{0mm}{0.8cm} \\
			0.50&4.73$^\dag$&4& \parbox[c]{65mm}{\includegraphics[scale=0.018]{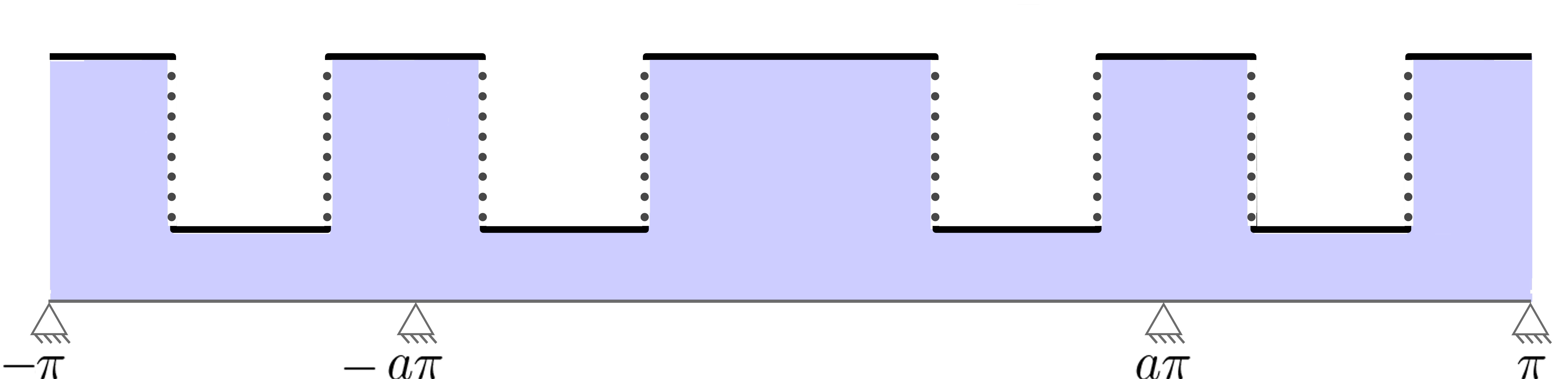}}&$1.22\cdot 10$&4& \parbox[c]{65mm}{\includegraphics[scale=0.018]{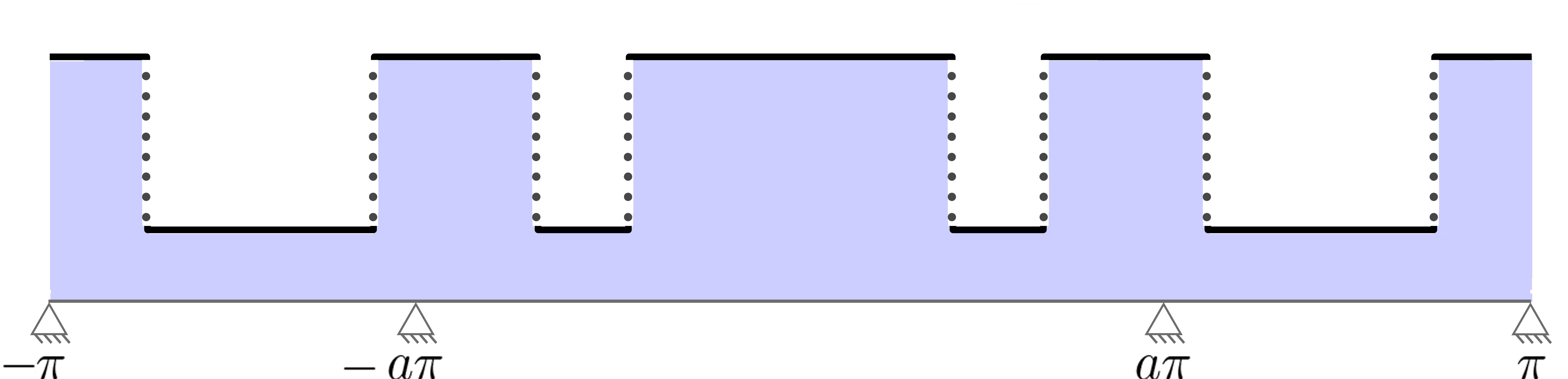}}
			\rule[2mm]{0mm}{0.8cm} \\
			0.55&5.06&4& \parbox[c]{65mm}{\includegraphics[scale=0.018]{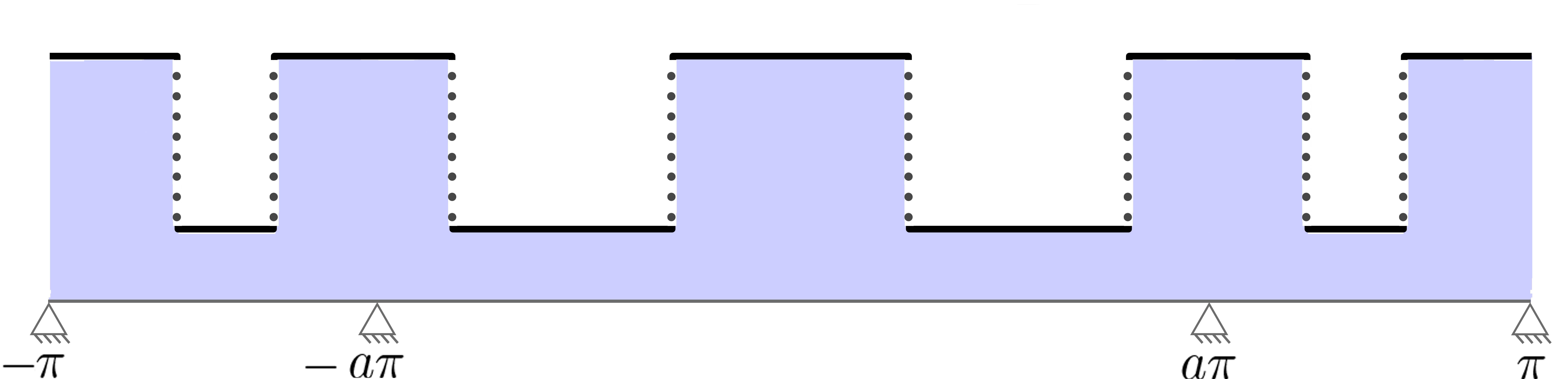}}&9.59&4& \parbox[c]{65mm}{\includegraphics[scale=0.018]{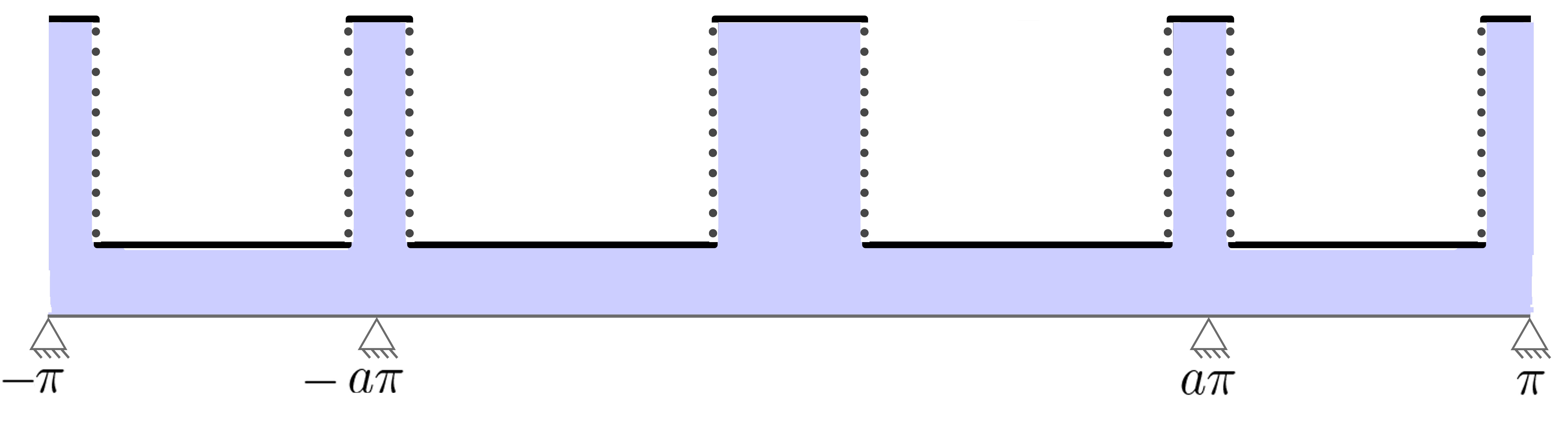}}
			\rule[2mm]{0mm}{0.8cm} \\
			0.60&4.42&4& \parbox[c]{65mm}{\includegraphics[scale=0.018]{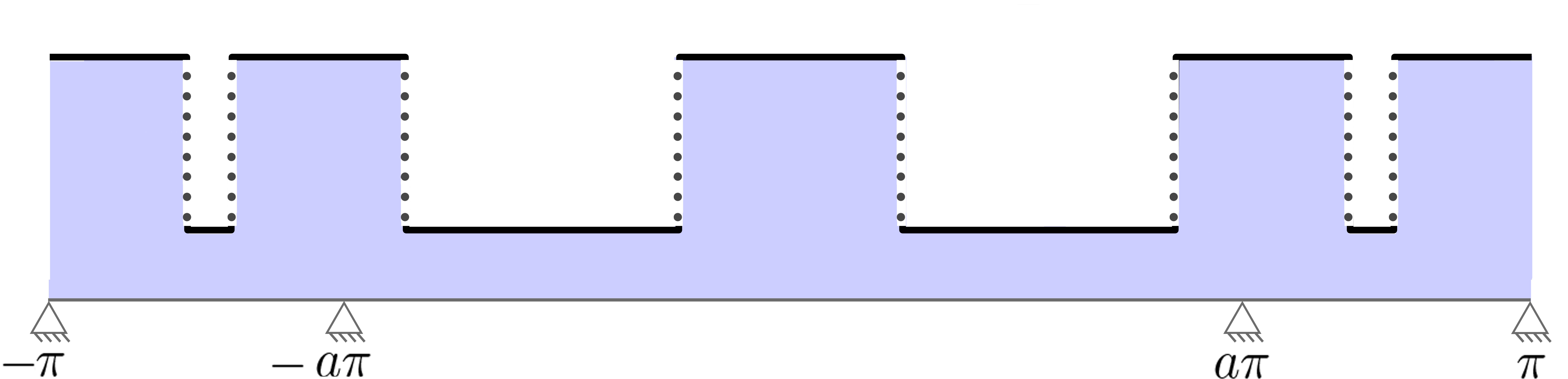}}&6.31&4& \parbox[c]{65mm}{\includegraphics[scale=0.018]{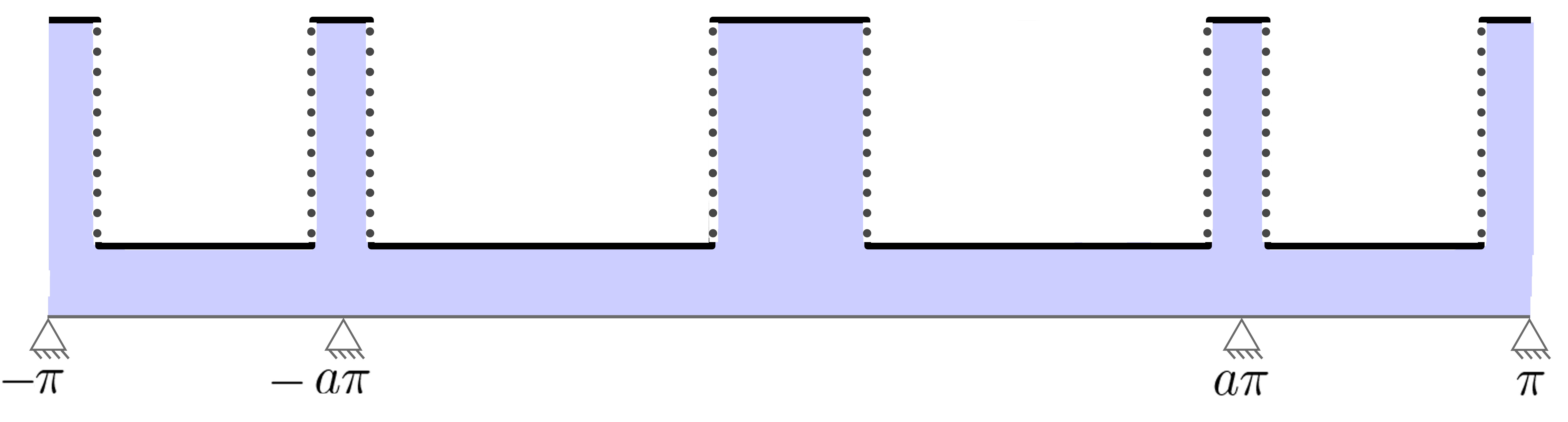}}
			\rule[2mm]{0mm}{0.8cm} \\
			0.65&3.41&2& \parbox[c]{65mm}{\includegraphics[scale=0.018]{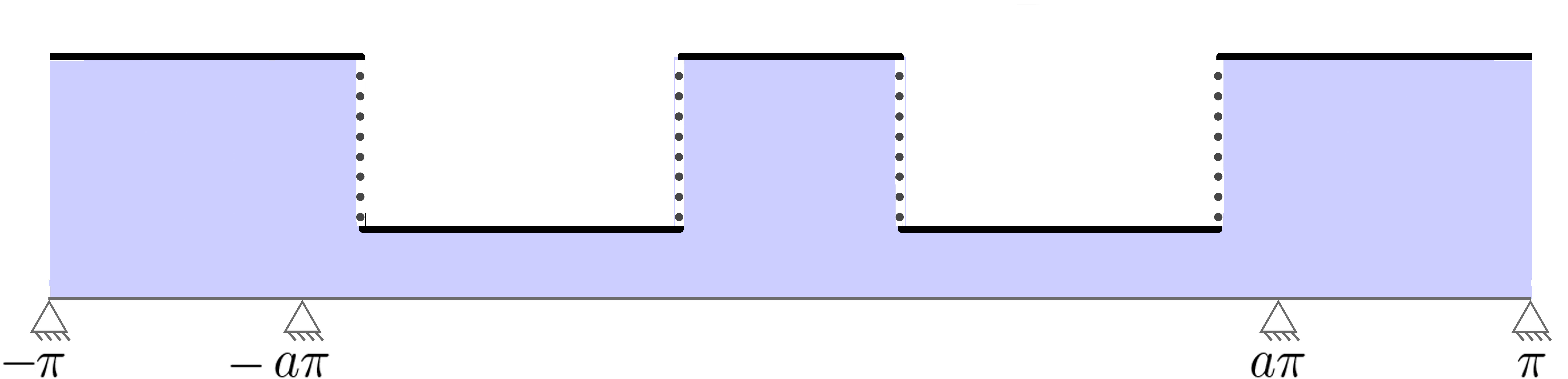}}&4.03&4& \parbox[c]{65mm}{\includegraphics[scale=0.018]{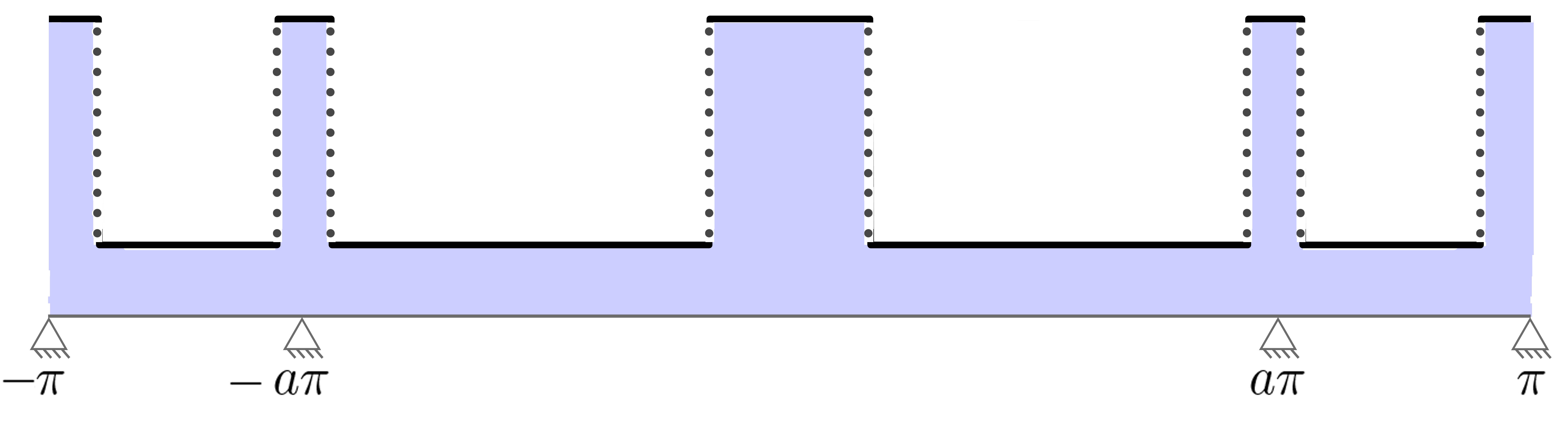}}
			\rule[2mm]{0mm}{0.8cm} \\
			0.70&2.16&2& \parbox[c]{65mm}{\includegraphics[scale=0.018]{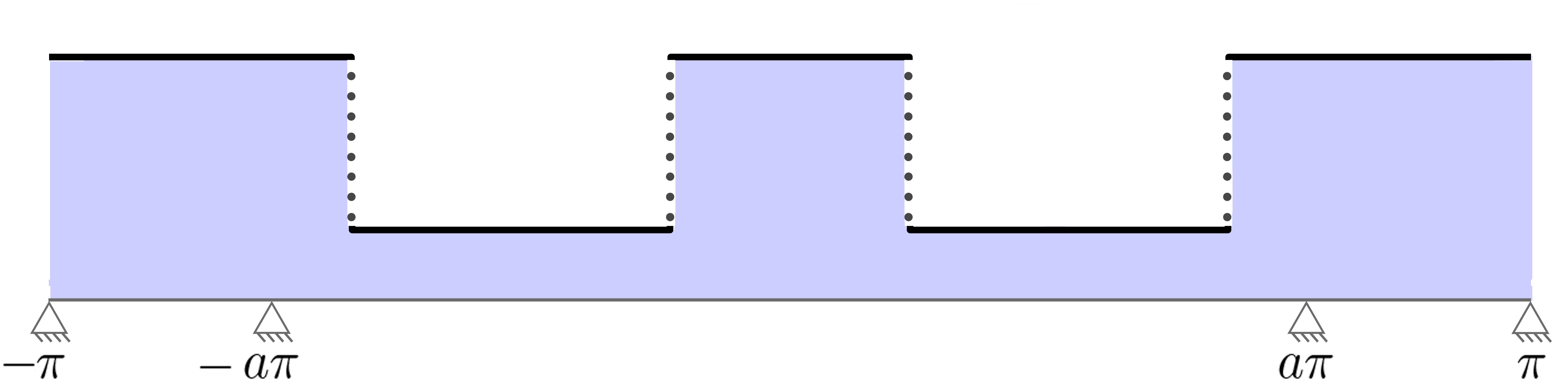}}&3.34&1& \parbox[c]{65mm}{\includegraphics[scale=0.018]{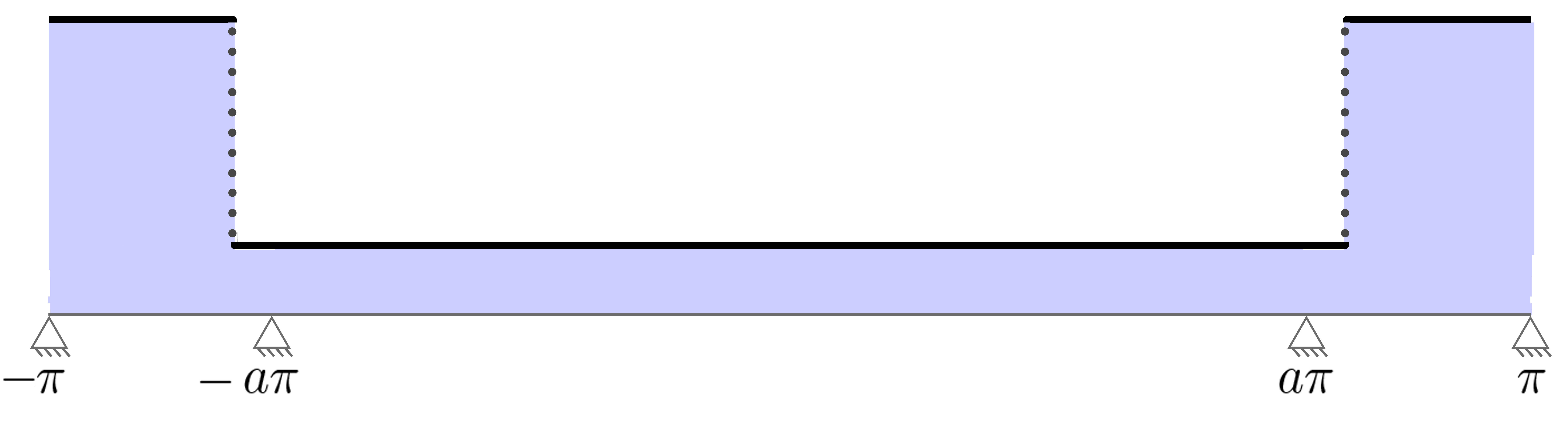}}
			\rule[2mm]{0mm}{0.8cm} \\
			0.80&$\dfrac{7.99}{10}$&2& \parbox[c]{65mm}{\includegraphics[scale=0.018]{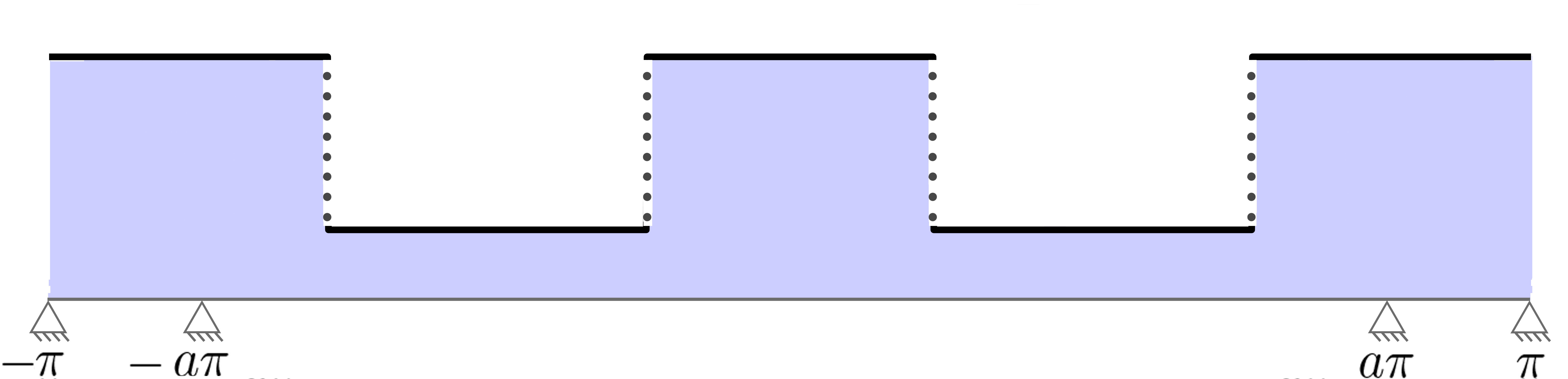}}&1.68&1& \parbox[c]{65mm}{\includegraphics[scale=0.018]{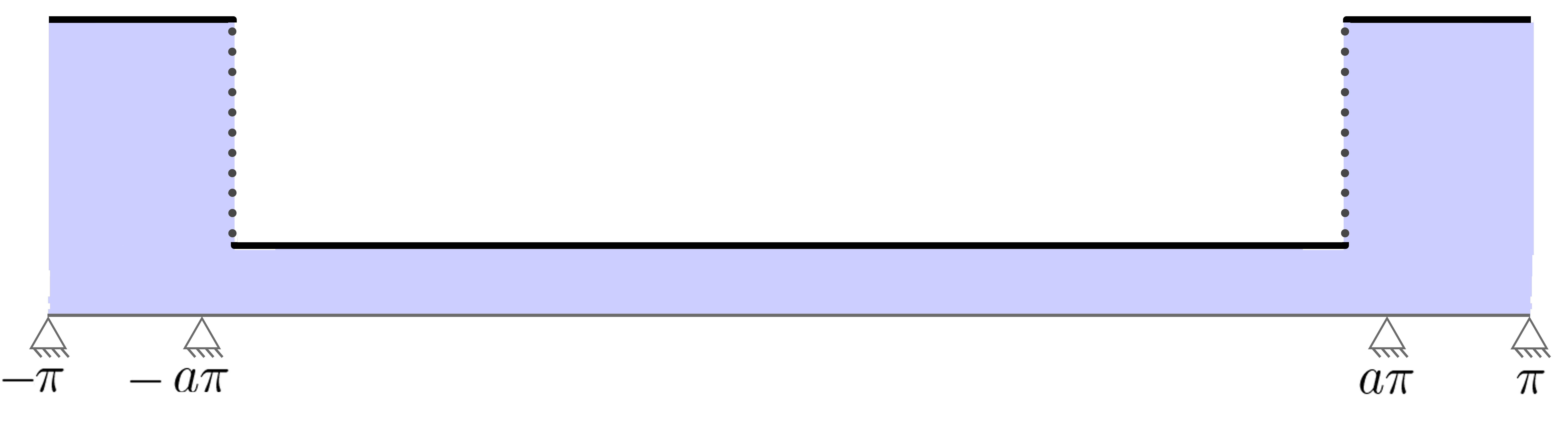}}
			\rule[2mm]{0mm}{0.8cm} \\
			0.90&$\dfrac{3.26}{10}$&2& \parbox[c]{65mm}{\includegraphics[scale=0.018]{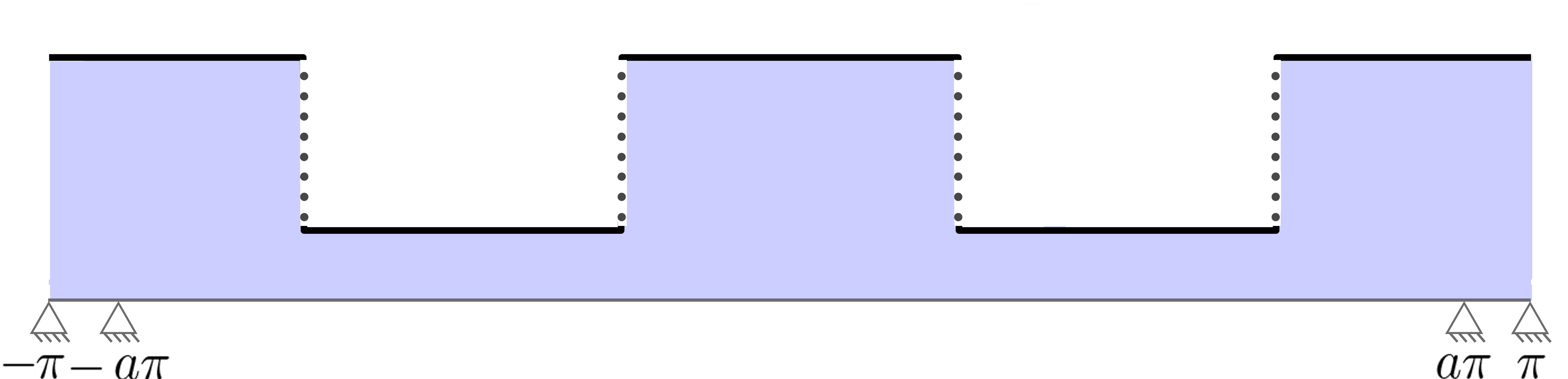}}&
			$\dfrac{7.37}{10}$&1& \parbox[c]{65mm}{\includegraphics[scale=0.018]{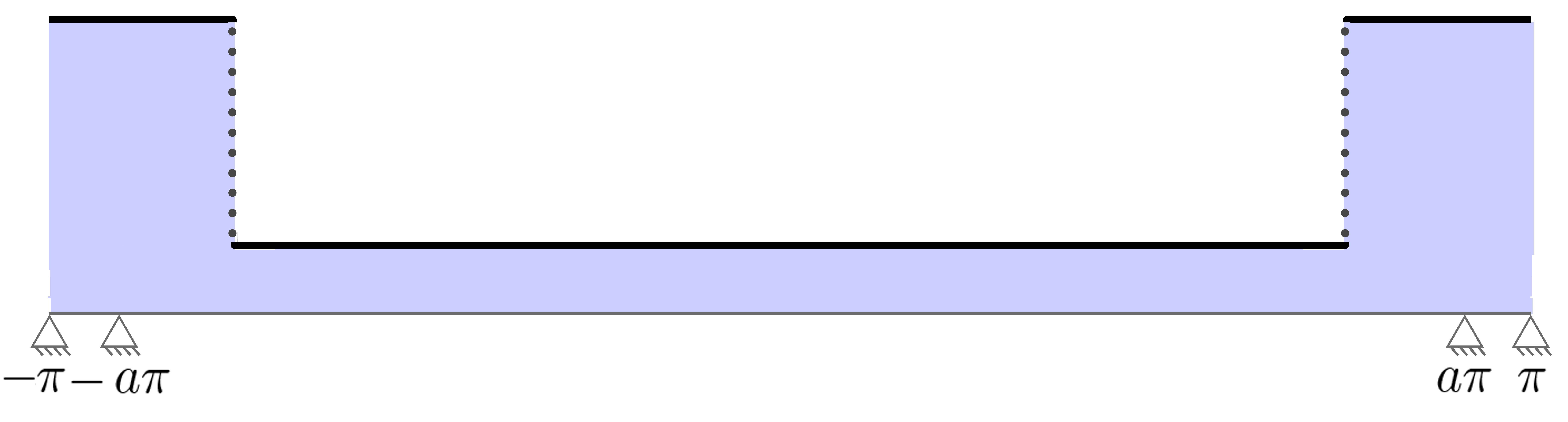}}
			\rule[2mm]{0mm}{0.8cm} \\
			\hline
	\end{tabular}} 

	\vspace{3mm}
	\caption{For different values of $\alpha$, $\beta$ and $a$, the numerical values of the critical energy threshold in \eqref{afix}, the graph of the corresponding piecewise constant maximizer $p^*(x)$ and the number $N^*$ of its discontinuities for $x > 0$. }
	\label{tab2}
\end{table}

\clearpage

\begin{table}[ht!]
	\begin{center}
		\resizebox{\textwidth}{!}{
			\begin{tabular}{|c|c|c|c|c||c|c|c|c||c|c|c|c||c|c|c|c||}
				\hline
				$\!\alpha\downarrow$ $\beta \to\!$  & \multicolumn{4}{c||}{3/2} & \multicolumn{4}{c||}{2} & \multicolumn{4}{c||}{5/2}& \multicolumn{4}{c||}{3}  \\
				\hline
				& $\mathcal{E}^*/10^2$ & $R$ & $N^*$ & $a_\text{opt}$ & $\mathcal{E}^*/10^2$ & $R$ & $N^*$ & $a_\text{opt}$ & $\mathcal{E}^*/10^2$ & $R$ & $N^*$ & $a_\text{opt}$& $\mathcal{E}^*/10^2$ & $R$ & $N^*$ & $a_\text{opt}$  \\
				\hline
				$5/6$ & 2.70 & $\lambda_2/\lambda_1$& 4 & \textbf{0.50} & 2.76 & $\lambda_2/\lambda_1$& 1 & 0.50 & 3.04 & $\lambda_2/\lambda_1$&1 & 0.50& 2.93 & $\lambda_3/\lambda_2$&1 & 0.50   \\
				\hline
				$2/3$ & 3.64 & $\lambda_2/\lambda_1$&4 & \textbf{0.50}  & 3.84 & $\lambda_2/\lambda_1$&4  & \textbf{0.50} &3.83 & $\lambda_2/\lambda_1$& 4& \textbf{0.50}&3.73 & $\lambda_2/\lambda_1$& 4& \textbf{0.50} \\
				\hline
				$1/2$  &4.54& $\lambda_2/\lambda_1$&4& \textbf{0.50} & 5.66& $\lambda_2/\lambda_1$&4 &\textbf{0.50} & 6.07& $\lambda_2/\lambda_1$&4& \textbf{0.50}& 6.13& $\lambda_2/\lambda_1$&4& \textbf{0.50}\\
				\hline
				$1/3$  & 5.06 & $\lambda_2/\lambda_1$&4& \textbf{0.55} & 8.39 & $\lambda_2/\lambda_1$&4 & \textbf{0.50} & 1.08$\cdot 10$ & $\lambda_2/\lambda_1$&4& \textbf{0.50}& 1.22$\cdot 10$ & $\lambda_2/\lambda_1$&4& \textbf{0.50}  \\
				\hline
			\end{tabular}
		}
		\smallbreak
		\caption{The optimal position $a_{\text{opt}}$ of the intermediate piers with respect to different $\alpha$ and $\beta$. $R$ is the ratio of eigenvalues corresponding to the critical energy threshold $\mathcal{E}^*$ and $N^*$ is the number of discontinuity points of the optimal weight $p^*$ for $x > 0$. In bold the situations with reinforced piers.}\label{tab3}
	\end{center}
\end{table}

\begin{figure}[h!]
	\centering
	{\includegraphics[width=12cm]{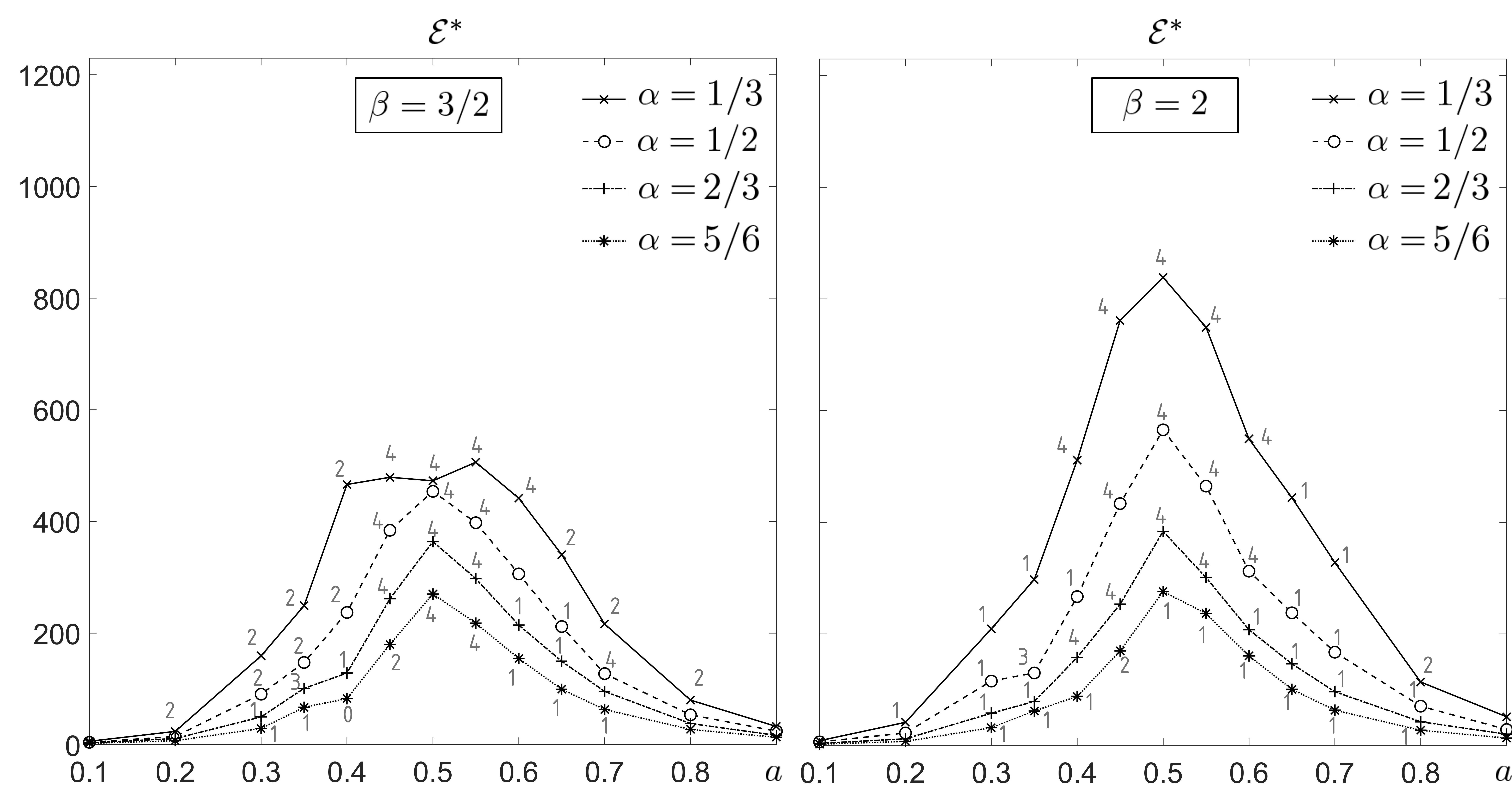}\\\includegraphics[width=12cm]{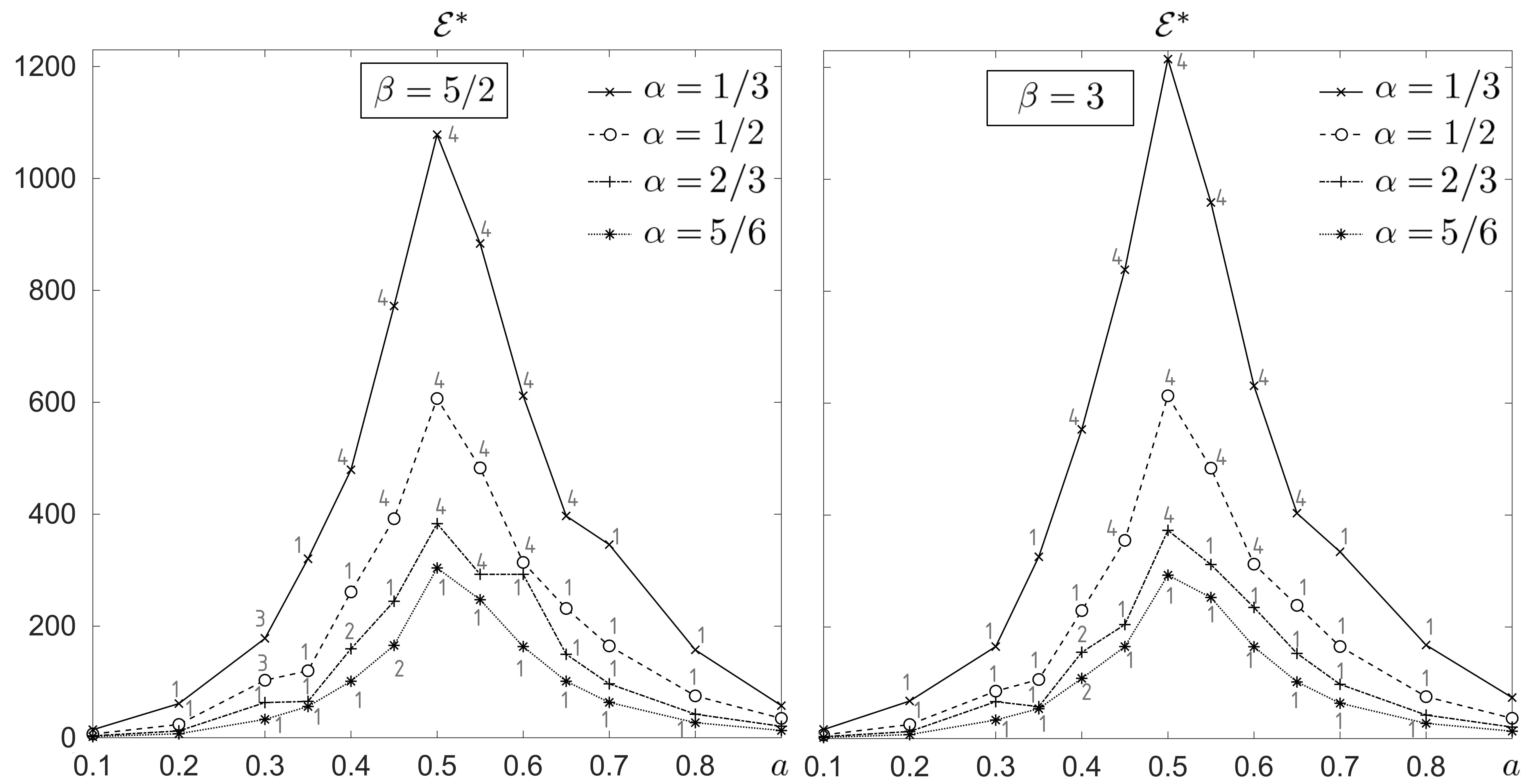}}
	\caption{ $\mathcal{E}^*$ versus $a$ with respect to different $\alpha$ and $\beta$. The numbers along the curves stand for the numbers $N^*$ of discontinuities of the corresponding $p^*(x)$ for $x > 0$. }
	\label{plotene}
\end{figure}

\subsection{Summary and conclusions.}\label{sumcon}

The stability analysis provided in \cite[Chapter 3]{GarGazBK} in the case of a \emph{homogeneous} beam ($p \equiv 1$) suggested an optimal position of the piers, given by $a=0.5$. We have here performed a similar stability analysis for a non homogeneous beam made of two different materials; as a first outcome (Claim a) at the end of Section \ref{2step} and Claim a) at the end of Section 5.2), it turns out that
\begin{center}
\textbf{introducing different densities along the beam reinforces the structure},
\end{center}
namely provides a stabler beam. We are then motivated to search an optimal combination ``position of the piers-density'' which increases as much as possible the energy thresholds of instability, maintaining however the ratio between the densities of the materials coherent with the one of two realistic building materials. Moreover, from the point of view of the applications, the number of jumps between the two materials along the beam has to be controlled, in order to make the building of the beam feasible. We have then proceeded as follows: 
\begin{itemize}
\item[1)] in Section \ref{2step}, we have constrained the total number of jumps to two; in this case, the eigenvalues are obtainable explicitly. We have found the best $a$ for each fixed density in a certain range and then we have maximized the obtained threshold as a function of the considered densities; 
\item[2)] in Section \ref{num}, for fixed $a$, we have looked for densities maximizing the ratio of ``dangerous'' couples of eigenvalues, accordingly obtaining optimal densities which, in general, have - at least in the physical range \eqref{physrangea} - more than two total jumps. For each $a$, we have then checked whether such densities give rise to higher critical energy thresholds with respect to those obtained for two-step densities; finally, we have maximized the obtained instability thresholds with respect to $a$, finding the optimal combination  ``position of the piers-density'' among those considered.
\end{itemize}
We have observed the following:
\begin{itemize}
\item[1)] in the first case, the optimal position of the piers always ranges between $0.35$ and $0.65$ and the best combination ``position of the piers-density'' is given by $a=0.65$ and $\alpha=1/3$, $\beta=2$, with lighter density in the middle of the beam and \emph{both the piers and the endpoints of the beam reinforced}; %\blu dire quali materiali sarebbero? ``titanium VS bricks''? \nero
\item[2)] in the second case, the highest thresholds are obtained in a similar range for $a$ ($0.40$-$0.60$), with densities having more than two total jumps, for which the heavier material is in general present \emph{around the piers} and \emph{at the endpoints} of the beam. The optimal combination ``position of the piers-density'' is here given by $a=0.50$ and $\alpha=1/3$, $\beta=3$. This choice reaches the highest observed energy threshold in our experiments, and is thus optimal among the considered ones. Just to give a rough idea, normalizing the density of reinforced concrete to $1$, the two corresponding materials could be, for instance, wood and steel. 
\end{itemize}
We then conclude that 
\begin{center}
\textbf{the physical range \eqref{physrangea} contains the optimal choices of the position of the piers, \\ both for two-step densities and for general ones;}
\end{center}
however, 
\begin{center}
\textbf{in the physical range \eqref{physrangea} it is generally better to choose densities with more than two total jumps, with $\alpha$ and $\beta$ sufficiently different one from the other}.
\end{center}
Moreover, in correspondence of the optimal choice of $a$, 
\begin{center}
\textbf{the denser material has always to be present around the piers \\ and next to the endpoints of the beam.}
\end{center}

These results may be the starting point for further investigations involving physical models for real structures like bridges (see, e.g. \cite{elvisefilippo} and \cite{bookgaz}), with the aim of testing the effectiveness of the above conclusions in preventing possible failures.

\section{Proofs}\label{dimostrazioni}
%\label{proof} 

In the following, we denote by $\mathbb{N}_+=\{1, 2, 3, \ldots\}$ the set of positive integers.

 \subsection{Proof of Theorem \ref{Galerkin}}
 %\label{proofG}

 Let $\{e_j\}_{j\in \N_+}$ be the set of eigenfunctions of \eqref{beam0}, with corresponding eigenvalues $\lambda_j$. We observe that $\{e_j\}_j$ is a complete system, both in $V(I)$ and in $L^2_p(I)$; we normalize each $e_j$ in such a way that $ \|e_j\|_{L_p^2(I)}^2=1$ for all $j\in \N_+$. Then, for a given $n\in \N_+$, we define
\begin{equation*}
%\label{un}
u^n(x,t)=\sum_{j=1}^{n} c_j(t) e_j(x)\,,
\end{equation*}
where the coefficients $c_j \in C^2([0, +\infty))$ satisfy the system of ODE's
\begin{equation} \label{coeff system}
\ddot{c}_j(t) +\lambda_j c_j(t) + \bigg(\sum_{k=1}^{n} c_k^2(t) \bigg)\,c_j(t)=0\,
\end{equation}
for $t>0$ and $j=1,...,n$. Moreover, writing the expansions of the initial data as
$$
g(x)=\sum_{j=1}^\infty g^j e_j(x)\quad\text{in }V(I),\qquad h(x)=\sum_{j=1}^\infty h^je_j(x)\quad\text{in }L^2_p(I),
$$
we set %assume that, for every $1\leq j\leq n$, the $c_j$'s satisfy the initial conditions
\begin{equation} \label{initial condition}
c_j(0)= g^j\,, \quad \dot{c}_j(0)= h^j\,, \qquad \text{for every } 1\leq j\leq n.
\end{equation}
The existence of a unique local solution to \eq{coeff system}-\eq{initial condition} in some maximal interval of continuation
$[0, \tau_n)$, $\tau_n>0$, follows from standard theory of ODE's. Then, for each $j=1, \ldots, n$ we multiply equation \eq{coeff system} by $\dot{c}_j(t)$ and we sum the so obtained equations, getting 
$$
\frac{d}{d\,t} \bigg[\frac{1}{2}\sum_{j=1}^{n} (\dot{c}_j(t))^2+\frac{1}{2} \sum_{j=1}^{n} \lambda_j\, c^2_j(t)+\frac{1}{4}\big(\sum_{j=1}^{n}\, c^2_j(t) \big)^2\bigg ]=0\,.
$$
Integrating this equality on $(0,t)$ and recalling the definition of $u^n$, we conclude that
\begin{equation}\label{uniform bound}
2\|\dot{u}^n(t)\|_{L_p^2(I)}^2+2\|u_{xx}^n(t)\|_2^2+\|u^n(t)\|_{L_p^2(I)}^4=2\|h^n\|_{L_p^2(I)}^2+2\|g_{xx}^n\|_2^2+\|g^n\|_{L_p^2(I)}^4 \leq C 
\end{equation}
for any $t\in[0,\tau_n)$ and every $n\geq 1$, where the constant $C$ is independent of $n$ and $t$.  Hence, $u^n$ is globally defined in $\R_+$ for every $n \geq 1$ and the sequence $\{u^n\}_n$ is uniformly bounded in the space $C^0([0,T];V(I))\cap C^1([0,T];L_p^2(I))$, for all finite $T>0$. We now show that $\{u^n\}_n$ admits a strongly convergent subsequence in the same spaces.\par
The estimate \eq{uniform bound} shows that $\{u^n\}_n$ is bounded and equicontinuous in $C^0([0,T];L^2_p(I))$. By the Ascoli-Arzel\`a Theorem
we then conclude that, up to a subsequence, $u^n\rightarrow u$ strongly in $C^0([0,T];L^2_p(I))$.
\par
Next, for every $n>m\ge1$, we set $u^{n,m}:=u^n-u^m$, $g^{n,m}:=g^n-g^m$ and $h^{n,m}:=h^n-h^m$. Repeating the computations which yield \eq{uniform bound} for $u^m$ and subtracting the obtained inequality from \eqref{uniform bound}, for all $t\in[0,T]$, we get
 \begin{equation*}
\begin{split}
&2\|\dot{u}^{n,m}(t)\|_{L^2_p(I)}^2+2\|u_{xx}^{n,m}(t)\|_2^2=-\left(\|u^n(t)\|_{L_p^2(I)}^2+\|u^m(t)\|_{L_p^2(I)}^2 \right)\left( \|u^n(t)\|_{L_p^2(I)}^2-\|u^m(t)\|_{L_p^2(I)}^2\right)\\
&+2\|h^{n,m}\|_{L_p^2(I)}^2+2\|g^{n,m}_{xx}\|_2^2+\left(\|g^n\|_{L_p^2(I)}^2+\|g^m\|_{L_p^2(I)}^2 \right)\left( \|g^n\|_{L_p^2(I)}^2-\|g^m\|_{L_p^2(I)}^2\right)\to0\qquad\mbox{as }n,m\to\infty,
\end{split}
\end{equation*}
so that $\{u^n\}_n$ is a Cauchy sequence in the space $C^0([0,T];V(I))\cap C^1([0,T];L_p^2(I))$. In turn this yields
$$ u^n\rightarrow u\quad  \text{ in } C^0([0,T];V(I))\cap C^1([0,T];L_p^2(I)) \quad \text{as }n\rightarrow +\infty\,.$$
\par 
We now show that $u$ satisfies equation \eq{beam_weak}.
We rewrite \eq{coeff system} as
 \begin{equation}\label{weakode}
 \int_0^T \bigg\{-\dot{c}_j(s)\chi'(s)+\lambda_j c_j(s) \chi(s)+ \bigg(\sum_{k=1}^{n} c_k^2(t) \bigg)\,c_j(s) \chi(s)\, \bigg\}\, ds=0,
   \end{equation}
for all $\chi \in \mathcal{D}(0, T)$. Moreover, for fixed $\varphi\in V(I)$ we denote by $\varphi^n(x)=\sum_{j=0}^{n} d_j e_j(x)$ the
orthogonal projection of $\varphi$ onto $X_n:={\rm span}\{e_j(x)\}_{j=1}^n$. Multiplying \eq{weakode} by $d_j$ and summing with respect to $j=1,...n$, we obtain
 \begin{equation*}
\int_0^T \bigg\{-(u^n_t(s),\varphi^n(x)\,\chi'(s))_p+(u^n_{xx}(s),\varphi^n_{xx}(x)\, \chi(s))+ \|u^n\|_{L_p^2(I)}^2\,(u^n(s),\varphi^n(x)\, \chi(s))_p\, \bigg\}\, ds=0
 \end{equation*}
   for all $\varphi \in  V(I)$ and $\chi \in \mathcal{D}(0, T)$.
By letting $n\rightarrow +\infty$ in the above equation, we conclude that $u$ satisfies \eq{beam_weak}. The verification of the initial conditions follows by noticing that $u^n(x,0) \rightarrow u(x,0)$ in $V(I)$ and $u_t^n(x,0) \rightarrow u_t(x,0)$ in $L^2_p(I)$. The proof of the existence part is complete, once we
observe that all the above results hold for any $T>0$.\par
For the uniqueness proof, we first show that \eq{energy} holds for any solution to \eq{beam_weak} (and not only for the limit solution obtained above). Indeed, let $u\in C^0([0,T]; V(I))\cap C^1(([0,T]; L_p^2(I))$ satisfy \eq{beam_weak}; we may write its expansion as $u(x,t)=\sum_{j=1}^{\infty} c_j(t) e_j(x)$, which converges in $V(I)$ and in $L^2_p(I)$. By testing \eq{beam_weak} with $\phi(x)=e_j(x)$, we readily get that each coefficient $c_j$ satisfies the equation
$$
 \ddot{c}_j(s)+ \lambda_j c_j(s) + \|u(s)\|_{L_p^2(I)}^2\, c_j(s)=0\,.
 $$
Multiplying this equality by $\dot{c}_j(t)$ and summing the obtained equations with respect to $j=1, \ldots, n$, we get
$$\frac{d}{d\,t} \left(\|\dot{u}^n(t)\|_{L_p^2(I)}^2+ \|u_{xx}^n(t)\|_2^2+2 \int_0^t   \|u(s)\|_{L_p^2(I)}^2\,(u^n(s),u_t^n(s))_p\, ds \right)=0$$
and, passing to the limit, we obtain
$$\|\dot{u}(t)\|_{L_p^2(I)}^2+ \|u_{xx}(t)\|_2^2+2 \int_0^t   \|u(s)\|_{L_p^2(I)}^2\,(u(s),u_t(s))_p\, ds=\|h\|_{L_p^2(I)}^2+ \|g_{xx}\|_2^2\,.$$
Finally, \eq{energy} follows from the fact that
$$ 
\int_0^t   \|u(s)\|_{L_p^2(I)}^2\,(u(s),u_t(s))_p\,ds=\frac{1}{4}(\|u(t)\|_{L_p^2(I)}^4-\|u(0)\|_{L_p^2(I)}^4),
$$ 
which can be proved by noticing that
$ \|u^n(s)\|_{L_p^2(I)}^2\,(u^n(s),u^n_t(s))_p= \frac{1}{4}\frac{d}{d\,t} \|u^n(s)\|_{L_p^2(I)}^4$
and letting $n\rightarrow +\infty$.
\par
We are now ready to complete the proof of the uniqueness. We adapt to our framework the idea of \cite[Theorem 11]{holubova2}. Let $v_1$ and $v_2$ be two solutions of \eq{beam_weak}; we set $u=v_1-v_2$ and
$$\overline u(x,t)=\int_0^tu(x,\tau)\, d\tau,\qquad w(x,t,\sigma)=-\int_{t}^{\sigma}u(x,\tau)\, d\tau,$$
for all $t, \sigma \in [0, T]$. Then $w(x,t,\sigma)=\overline u(x,t)-\overline u(x,\sigma)$ and $w_t(x,t,\sigma)=u(x,t)$.
From Remark \ref{regw}, $u$ satisfies
\begin{equation*}
\langle p(x) u_{tt}, \varphi(x) \rangle_V+ (u_{xx},\varphi_{xx}(x)) + (\Vert u \Vert^2_{L^2_p(I)} u, \varphi(x))_p = 0 
\quad \forall \varphi \in V(I), \, \forall t > 0.
%+\left(\Big( \|v_1(s)\|_{L_p^2(I)}^2\,v_1(x,s)- \|v_2(s)\|_{L_p^2(I)}^2\,v_2(x,s)\Big),\varphi(x)\right)_p
\end{equation*}
%for all $\varphi \in V(I)$ and for almost every $s \in [0, T]$. 
Since $w(x,t,T) \in C^1([0, T]; V(I))$, we may write the above identity with $\varphi=w$ as test function; then, integrating with respect to time, we obtain 
\begin{equation*}
\begin{split}
& \int_0^T (u_t(x,s),u(x,s))_p\, \, ds-  \int_0^T (w_{txx}(x,s,T),w_{xx}(x,s,T))\, ds \\
&=\int_0^T \left( \left( \|v_1(s)\|_{L_p^2(I)}^2\,v_1(x,s)- \|v_2(s)\|_{L_p^2(I)}^2\,v_2(x,s) \right),w(x,s,T)\right)_p\, ds\, 
\end{split}
\end{equation*}
(notice that, in order to write the first summand in this way, we have performed an integration by parts relying on the fact that $u_t(x, 0)=0$ and $w(x, T, T)=0$).
By noticing that
$ (u_t(s),u(s))_p=\frac{1}{2} \frac{d}{d\,s }  \|u(s)\|_{L_p^2(I)}^2$, $ (w_{txx}(s,T),w_{xx}(s,T))=\frac{1}{2} \frac{d}{d\,s } \|w_{xx}(s,T)\|_{L^2(I)}^2$ and $w_{xx}(x,0,T)=\overline u_{xx}(x,T)$, we thus deduce that
\begin{equation*}
\begin{split}
&\frac{1}{2}\, \|u(x,T)\|_{L_p^2(I)}^2+\frac{1}{2}\, \|\overline u_{xx}(x,T)\|_{L^2(I)}^2 \\
&\leq \sup_{s\in[0,T]} \|v_1(s)\|_{L_p^2(I)}^2  \int_0^T \left(u(x,s),w(x,s,T)\right)_p\, ds\\
&+ \sup_{s\in[0,T]} \|v_2(s)\|_{L^{\infty}(I)} \, \left[\|v_1(s)\|_{L^{\infty}(I)}+\|v_2(s)\|_{L^{\infty}(I)} \right]  \int_0^T \left( u(x,s),w(x,s,T)\right)_p\, ds\\
& \leq C  \int_0^T \left( \|u (x,s)\|_{L_p^2(I)}\,  \|\overline u(x,s)\|_{L_p^2(I)}+ \|u(x,s)\|_{L_p^2(I)}\, \|\overline u(x,T)\|_{L_p^2(I)}\right) \, ds\,,
\end{split}
\end{equation*}
for some $C>0$, where we have exploited \eq{beam_weak}, H\"older and Minkowski inequalities and the fact that $w(x,s,T)=\overline u(x,s)-\overline u(x,T)$. Finally, from Young inequality and the continuous embedding $V(I) \subset L^2_p(I)$ we conclude that
$$ \|u(x,T)\|_{L_p^2(I)}^2+ \|\overline u_{xx}(x,T)\|_{L^2(I)}^2  \leq K  \int_0^T \|u(x,s)\|_{L_p^2(I)}^2+ \|\overline u_{xx}(x,s)\|_{L^2(I)}^2 \, ds$$
for some $K>0$. Therefore, by Gronwall lemma we conclude that $u\equiv 0$, which is the thesis.

 \subsection{Proof of Proposition \ref{general facts}}
 %\label{simple} %{Proofs of Section \ref{weighteigenpb}}

 For $h\in\mathbb{N}_+$, we denote respectively by $\lambda_h(p)$ and $e_h(x)$ the $h$-th eigenvalue and the $h$-th eigenfunction of \eqref{beam0}.
 We recall the following variational representation of eigenvalues for every $h\in \N_+$, see e.g.  \cite{courant,henrot}:
 \begin{equation}
 \label{caract1}
 \lambda_h(p)=\inf_{\substack{W_h\subset V\\\dim W_h=h}}\hspace{2mm}\sup_{\substack{v\in W_h\setminus\{0\}}}\frac{\|v\|^2_{V}}{\|\sqrt{p}v\|^2_2}.
 \end{equation} 
 When $h=1$, \eqref{caract1} includes the well known characterization of the first eigenvalue
 \begin{equation*}
 %\label{first}
 \lambda_1(p)=\inf_{\substack{v\in V\setminus\{0\}}}\frac{\|v\|^2_{V}}{\|\sqrt{p}v\|^2_2}.
 \end{equation*} 
 First, we prove the following lemma.
 \begin{lemma}\label{nehari}
 	Let $e$ be a weak solution of \eqref{beam0} which satisfies
 	\begin{equation}\label{ci}
 	\begin{cases}
 	e'(-\pi)=0\\
 	e'''(-\pi)=A>0
 	\end{cases}
	\qquad \text{or} \qquad 
	\begin{cases}
 	e'(-\pi)=A>0\\
 	e'''(-\pi)=0\,.
 	\end{cases}
 	\end{equation}
 	Then 
 	$
 	e(x)>0$ for all $x\in(-\pi,-a\pi]$.
 \end{lemma}
 \begin{proof}
 	We expand on the idea of \cite[Lemma 2.2]{nehari}. Since, by Proposition \ref{regolarita}, $e\in H^4(I) \cap C^3(\overline I_-)$ and solves \eqref{beam0} almost everywhere, for $x\in \overline I_-$ we may write
\begin{equation*}
e'''(x)-e'''(-\pi)=\lambda \int_{-\pi}^{x}p(t)e(t)\,dt,
\end{equation*}
so that %inversione di variabili negli integrali doppi
\begin{equation}\label{eq1}
e''(x)=e'''(-\pi)(x+\pi)+\lambda \int_{-\pi}^{x}(x-t)p(t)e(t)\,dt.
\end{equation}
Assume that the first case in \eqref{ci} occurs. By contradiction, suppose that there exists $x_0\in(-\pi,-a\pi]$ such that $e(x_0)=0$; since $e'''(-\pi)>0$, we infer $e(x)\rightarrow 0^+$ as $x\rightarrow -\pi^+$ and by the continuity of $e$ we then get $e(x)>0$ for $x\in (-\pi,x_0)$. By \eqref{eq1}, this implies $e''(x)>0$ for $x\in (-\pi,x_0]$ and also $e'(x)>0$ for $x\in (-\pi,x_0]$, being $e'(-\pi)=0$. Therefore we obtain
 	$$
 	0=e(-\pi)<e(x)<e(x_0)\qquad \forall x\in (-\pi,x_0),
 	$$
 	getting the contradiction $e(x_0)>0$ and, in turn, the thesis. The same contradiction occurs if the second case in \eqref{ci} occurs.
 \end{proof}
   {\bf Proof of Proposition  \ref{general facts} completed.} The fact that the eigenvalues of \eqref{beam0} form a divergent sequence $\lambda_j(p)$ ($j\in \N_+$) and that the corresponding eigenfunctions $e_j$ form a complete system in $V$ follows from the standard spectral theory of self-adjoint operators. 
It remains to show the simplicity of the eigenvalues. By contradiction, let us suppose that $e_1$ and $e_2$ are two linearly independent eigenfunctions corresponding to the same eigenvalue $\lambda$. Then, $w(x)=c_1 e_1(x)+c_2e_2(x)$ ($c_1, c_2\in\mathbb{R}$) is an eigenfunction associated with the eigenvalue $\lambda$, solving \eqref{beam0}. Now we fix $c_1, c_2\in\mathbb{R}$ so that $w'(-\pi)=0$ and $w'''(-\pi)=A>0$. From Lemma \ref{nehari}, we obtain $w(x)>0$ for all $x\in (-\pi,-a\pi]$, in particular $w(-a\pi)>0$; this is a contradiction since $w$ is a solution of \eqref{beam0} and hence $w(-a\pi)=0$.
 \begin{remark}
\textnormal{As a direct consequence of Lemma \ref{nehari}, we infer that all the eigenfunctions satisfy $e_j'(\pm \pi)\neq 0$ and $e_j'''(\pm \pi)\neq 0$. Indeed, if it were $e_j'(-\pi)= 0$ (resp., $e_j'''(-\pi)=0$), then up to a change of sign the first (resp., the second) condition in \eqref{ci} would be satisfied. This would necessarily imply $e_j'''(-\pi)=0$ (resp., $e_j'(-\pi)=0$), otherwise by Lemma \ref{nehari} it would be $e_j(-a\pi)\neq 0$, which is not the case. But the uniqueness would then yield $e_j\equiv 0$, a contradiction.} % Hence, $e_j'( -\pi)\neq 0$; if it were $e_j'''(-\pi)=0$, up to a change of sign the second condition in \eqref{ci} would occur, but then we would get $e_j(-a\pi) \neq 0$, which is not the case. Thus, it is $e_j'''(-\pi) \neq 0$, as well.}} 
 \end{remark}

 \subsection{Proof of Proposition \ref{regolarita}}
 %\label{reg}
We follow the lines of the proof of \cite[Lemma 2.2]{holubova}. Since $e\in V(I)\subset H^2(I)$, we have that $e\in C^1(\overline I)$, due to the compact embedding $H^2(I)\subset\subset C^1(\overline I)$. Testing \eqref{beamweak} with a fixed function $v_2\in C^\infty_c(I_0)$, we obtain
	 \begin{equation}
	 \int_{I_0} e_0''v_2''\,dx=\lambda\,\int_{I_0} p\,e_0v_2\,dx\qquad \forall v_2\in C^\infty_c(I_0),
	 \label{beamweakI0}
	 \end{equation}
	 where $e_0$ is defined in \eqref{u} (we omit to write the $x$-dependences). We rewrite \eqref{beamweakI0} as
	 \begin{equation}\label{equ1}
	 	\int_{I_0} T(x)v_2''(x)\,dx=0\qquad \forall v_2\in C^\infty_c(I_0),
	 \end{equation}
	 with $T(x)=e_0''(x)-\lambda\int_{\xi}^{x}\int_{\xi}^{t}p(\tau)e_0(\tau)\,d\tau\,dt$. 
	If we consider the equality \eqref{equ1} in distributional sense, we get 
	 	 \begin{equation*}\label{equ2}
	  \langle T'', v_2 \rangle=0\qquad \forall v_2\in  C^\infty_c(I_0),
	 \end{equation*}
where $T''$ is the second derivative of $T(x)$ in sense of the distributions. Hence, 
\begin{equation}\label{distrib}
	T(x)=c_0x+c_1\qquad \text{for a.e. } x\in I_0\quad (c_0,c_1\in \mathbb{R}).
\end{equation}
Next we introduce the function $F: I_0\times \mathbb{R}\rightarrow \mathbb{R}$ through
\begin{equation*}
	F(x,s):=s-\lambda\int_{\xi}^{x}\int_{\xi}^{t}p(\tau)e_0(\tau)\,d\tau\,dt-c_0x-c_1;
\end{equation*}
we observe that, for every $x\in\overline I_0$, there exists a unique $\overline s$ such that $F(x,\overline s)=0$. Being $e_0\in C^1(\overline I_0)$ and $p\in L^\infty(I_0)$ we have that $F\in C^1(\overline I_0\times \mathbb{R})$; hence, by the Implicit Function Theorem we deduce the existence of a function $\overline s=\overline s(x)\in C^1(\overline I_0)$  such that $F(x,\overline s(x))=0$. From the definition of $T(x)$ we observe that \eqref{distrib} can be written as $F(x,e_0''(x))=0$ for a.e. $x\in I_0$, implying that $\overline s(x)=e_0''(x)$ for a.e. $x\in I_0$. By the continuity of $\overline s$, arguing as in \cite[Lemma 2.2]{holubova} we have $s(x)=e_0''(x)$ for every $x \in I_0$ and, being  $\overline s\in C^1(\overline I_0)$, we conclude that $e_0\in C^3(\overline I_0)$. From $F(x, e_0''(x))=0$ we then have that 
\begin{equation}\label{e0terzo}
e_0'''(x)=\lambda\int_{\xi}^{x}p(t)e_0(t)\,dt + c_0. %\in C^0(\overline I_0).
\end{equation}
Since $pe_0\in L^\infty(I_0)\subset L^1(I_0)$, from the Fundamental Theorem of Integral Calculus we deduce that
\begin{equation*}
	e_0''''(x)=\lambda p(x)e_0(x)\qquad \text{for a.e. }x\in I_0.
\end{equation*}
This argument yields the same conclusion on $e_-(x)$ and $e_+(x)$, if we consider respectively $v_1\in  C^\infty_c(I_-)$ and $v_3\in C^\infty_c(I_+)$ as test functions in \eqref{beamweak}; thus, \eqref{beam0} is satisfied for almost every $x\in I$. 
\par
We now conclude the proof by showing that $e \in C^2(\overline{I})$. 
From \eqref{beamweak}-\eqref{u} we get
\begin{equation*}
\int_{I_-} e_-''v''\,dx+ \int_{I_0} e_0''v''\,dx+\int_{I_+} e_+''v''\,dx=\lambda\,\int_{I} p\,ev\,dx\qquad \forall v\in  V(I).	
\end{equation*}
Hence, integrating twice by parts (recall that $e_0'''$ is absolutely continuous by \eqref{e0terzo}), we obtain
\begin{equation*}
\begin{split}
&[e_-''(-a\pi)-e_0''(-a\pi)]v'(-a\pi)+[e_0''(a\pi)-e_+''(a\pi)]v'(a\pi)\\&+\int_{I_-} \big[e_-''''-\lambda\,pe_-\big] \,v\,dx+\int_{I_0} \big[e_0''''-\lambda\,pe_0\big] \,v\,dx+\int_{I_+} \big[e_+''''-\lambda\,pe_+\big] \,v\,dx=0\qquad \forall v\in  V(I),
\end{split}
\end{equation*}
giving
$$
[e_-''(-a\pi)-e_0''(-a\pi)]v'(-a\pi)+[e_0''(a\pi)-e_+''(a\pi)]v'(a\pi)=0\qquad \forall v\in  V(I).
$$
This implies $e_-''(-a\pi)=e_0''(-a\pi)$, $e_0''(a\pi)=e_+''(a\pi)$ and, in turn, $e\in C^2(\overline I)$.

 \subsection{Proof of Proposition \ref{existence}}
 
The proof of Proposition \ref{existence} follows by combining the three lemmas that we state here below.
Preliminarily, for $h\in \N_+$ we introduce the orthogonal projection of $u \in V(I)$ onto the space generated by the first $(h-1)$ eigenfunctions $e_1,\dots, e_{h-1}$ of problem \eqref{beam0}, with respect to the $p$-weighted scalar product: 
$$
P_{h-1}(p)u:=\sum_{i=1}^{h-1}(p\, u,e_i)_{L^2}\, e_i\,.
$$ 
If $h=1$, we assume $P_0(p)u =0$. We then recall the Auchmuty's principle (see \cite{auchmuty} for the proof in a general setting):
 \begin{lemma}\label{lemmaauchmuty}
 	Let $p\in P_{\alpha, \beta}$, $h\in \N_+$ and $\lambda_h(p)$ the $h-$th eigenvalue of \eqref{beam0}, then 
 	$$
 	-\dfrac{1}{2\lambda_h(p)} =\inf_{u \in V(I)}\mathcal{A}_h(p,u)\quad \text{where}\quad \mathcal{A}_h(p,u):=\dfrac{1}{2}\|u\|^2_{V}-\|\sqrt{p}\,\big[u-P_{h-1}(p)u\big]\|_{2}\,.
 	$$
 	Furthermore, the minimum is achieved at a $h$-th eigenfunction normalized according to
 	$$\|e_h\|^2_{V}=\|\sqrt{p}\,e_h\|_{2}=\dfrac{1}{\lambda_h(p)}$$
 \end{lemma}
% \begin{proof}
% \rosso [PROPONGO DI TOGLIERLA] The proof follows arguing as in \cite[Lemma 3.3]{cuccu22} by simply replacing $H^2\cap H^1_0$ with $V$. In alternative, in \cite{auchmuty} one can find the original proof in a general setting. (CONTROLLARE, ma credo valga)\nero
% \end{proof}
The second lemma states a compactness property for the considered set of densities $P_{\alpha, \beta}$.
 \begin{lemma}\label{compactness}
 The set $P_{\alpha,\beta}$ is compact for the weak* topology of $L^\infty(I)$.
 \end{lemma}
 \begin{proof}
First, we prove that $P_{\alpha,\beta}$ is a strongly closed set in $L^2(I)$. To this end, let $\{ p_m \}_m \subset P_{\alpha,\beta}$ be a converging sequence in $L^2(I)$, and denote by $q$ its $L^2$-limit; then $p_m\rightarrow q$ also in $L^1(I)$ for $m\rightarrow +\infty$ and, up to a subsequence (still denoted by $p_m$), we have that $p_m \rightarrow q$ almost everywhere in $I$. Hence, $\alpha\leq q(x) \leq\beta$ for almost every $x \in I$;  moreover, $
 	\int_I p_m\, v\,dx\rightarrow \int_I q \,v\,dx$ for every $v\in L^2(I)
 	$, implying, for the choice $v\equiv 1$, that $2\pi=|I|=\int_I q\,dx$. Thus, $q\in P_{\alpha,\beta}$ and $P_{\alpha,\beta}$ is strongly closed in $L^2(I)$.  \par  
 	%We now show that from any sequence in $P_{\alpha, \beta}$ it is possible to extract a subsequence which converges, in the weak* topology of $L^\infty(I)(I)(I)$, to an element of  $P_{\alpha,\beta}$.
Let now $\{ p_m \}_m \subset P_{\alpha,\beta}$; since $\|p_m\|_\infty\leq \beta$ by the definition of $P_{\alpha, \beta}$, there exist $\overline{p} \in L^\infty(I)$ and a subsequence $\{p_{m_k}\}_k$ for which
 $$
 p_{m_k} \overset{\ast}{\rightharpoonup}  \overline{p} \quad \mbox{in} \ L^\infty(I)\quad \mbox{as} \ k \rightarrow \infty\,.
 $$
 Moreover, we have $\|p_{m_k}\|_2^2\leq \beta^2|I|$ so that, passing to a further subsequence, we infer that 
 $p_{m_{k_{j}}}\rightharpoonup \overline q$ in $L^2(I)$ as $j\rightarrow \infty$, for a suitable $\overline q \in L^2(I)$.
 Therefore, for every $v\in L^2(I)\subset L^1(I)$ it holds
 	$$
 	\int_I p_{m_{k_{j}}}\, v\,dx\rightarrow\int_I\overline q\, v\,dx \quad \mbox{as} \ j \rightarrow \infty % \quad \text{with }\overline q\in P_{\alpha,\beta}\,
 	$$
 	and, since $p_{m_k} \overset{\ast}{\rightharpoonup}  \overline{p}$ in $L^\infty(I)$ yields $\int_I p_{m_{k_{j}}}\, v\,dx\rightarrow\int_I\overline p\, v\,dx$ $\forall v\in L^1(I)$, we conclude that $\overline p=\overline q$ almost everywhere in $I$. Finally, it is easy to check that $P_{\alpha,\beta}$ is a convex set; since strongly closed convex spaces are weakly closed, we readily infer that $\overline q\in P_{\alpha,\beta}$ and hence $\overline p\in P_{\alpha,\beta}$. The proof is complete.
 \end{proof}
Finally, we prove the continuity of the eigenvalues with respect to the weight $p$.
 \begin{lemma}\label{continuity}
 	Let $h\in \N_+$ and let $\lambda_h(p)$ be the $h$-th eigenvalue of \eqref{beam0}. The map $\lambda_h: P_{\alpha, \beta} \to \mathbb{R}$ defined by $p\mapsto\lambda_h(p)$ is continuous with respect to the weak* convergence.
 \end{lemma}
 \begin{proof}
 	Let $\{ p_m \}_m \subset P_{\alpha,\beta}$ be a sequence which converges, in the weak* topology of $L^\infty(I)$, to a suitable $\overline{p} \in L^\infty(I)$, i.e.,
 	$$
 	p_m \overset{\ast}{\rightharpoonup}  \overline{p} \quad \mbox{for} \ m \rightarrow \infty;
 	$$
by Lemma \ref{compactness}, we know that $\overline{p}\in P_{\alpha,\beta}$.\par 
We associate with $p_m$ the $h$-th eigenvalue $\lambda_h(p_m)$ of \eqref{beam0} and the corresponding eigenfunction $e_h(p_m)$, normalized in such a way that $\Vert e_h(p_m) \Vert_{L^2_{p_m}(I)}=1$. By \eqref{beam0}, this implies that $\lambda_h(p_m)=\|e_h(p_m)\|^2_{V(I)}$; % with respect to the weighted scalar product, i.e. $\int_\Omega p_m\,e_h(p_m)\,e_r(p_m)\,dx\,dy=\delta_{hr}$, where $\delta_{hr}$ is the Kronecker delta for all $h,r \in \mathbb{N}_+$ and.\par 
moreover, by \eqref{eq:famiglia} and \eqref{caract1}, we have
\begin{equation}\label{eb}
 	\lambda_h(p)\leq \dfrac{\lambda_h(1)}{\alpha}\qquad \forall p\in P_{\alpha,\beta},
\end{equation}
where $\lambda_h(1)$ is the $h$-th eigenvalue of \eqref{beam0} with $p\equiv 1$.
From \eqref{eb} it then follows that 
 $\lambda_h(p_m)=\|e_h(p_m)\|^2_{V(I)} \leq \lambda_h(1)/\alpha$, so that $\{e_h(p_m)\}_m$ is bounded in $V(I)$.  Therefore, we can extract a subsequence, which we still label by $m$, such that, at the same time, 
 	\begin{equation*}\label{wconvergence}
 	\begin{array}{lll}
 	\lambda_h(p_m)\rightarrow  \overline \lambda_h & \mbox{in} \;  \mathbb{R} & \quad \mbox{as} \ m \rightarrow \infty \\
 	e_h(p_m) \rightharpoonup  \overline{e}_h & \mbox{in} \; V(I)  &\quad \mbox{as} \ m \rightarrow \infty.
 	\end{array}
 	\end{equation*}
 	Moreover, due to the compact embedding $V(I) \hookrightarrow L^2(I)$, we have that $
 	e_h(p_m)$ strongly converges to $\overline{e}_h$ in $L^2(I)$ as $m \rightarrow \infty$.
Taking into account that, for every $v \in V(I)$, it holds
 	\begin{equation*}
 	\bigg|\int_I (p_m\,e_h(p_m)-\overline p\,\overline e_h)\,v\,dx\bigg|\leq \|p_m v\|_2\|e_h(p_m)-\overline e_h\|_2+\bigg|\int_I p_m\,\overline e_hv\,dx-\int_I \overline p\,\overline e_hv\,dx\bigg|\rightarrow 0
 	\end{equation*}
 (since $\overline e_h v\in V(I)\subset L^1$),
 this implies that
 	\begin{equation}\label{pconv}
 	\int_{I} p_m\,e_h(p_m)\,v\,dx\rightarrow\int_I \overline p\,\overline e_h\,v\,dx \quad \mbox{as} \ m \rightarrow \infty, \quad \text{for every } v \in V(I).
 	\end{equation}
Consequently, for every $v \in V(I)$ we obtain
 	\begin{align}\label{allimite}
 0=	\big(e_h(p_m),v\big)_{V}-\lambda_h(p_m)\big( p_m\,e_h(p_m),v\big)_{L^2}\rightarrow \big(\overline e_h,v\big)_{V}-\overline \lambda_h\big( \overline p\,\overline e_h,v\big)_{L^2} \qquad \mbox{as} \ m \rightarrow \infty\,,
 	\end{align}
 	inferring that $\overline\lambda_h$ is an eigenvalue of \eqref{beam0} for $p=\overline p$ and $\overline e_h$ is a corresponding eigenfunction. \par
So far, we know that $\{\overline \lambda_h\}_h$ is a subset of the set of the eigenvalues of \eqref{beam0} with $p=\overline p$. Moreover, arguing as for \eqref{pconv} we obtain
\begin{equation}\label{pconv2}
\delta_{hr} = \int_I p_m e_h(p_m) e_r(p_m) \, dx \to \int_I \overline p \overline e_h \overline e_r \, dx \quad \text{for every } h, r \in \N_+,
\end{equation}
which implies that the $\overline \lambda_h$'s are all distinct.
Therefore, by Proposition \ref{general facts}, $\{\overline \lambda_h\}_h$ is a \emph{strictly increasing} sequence, which then tends, for $h \to +\infty$, 
to some $\overline \lambda \in \mathbb{R} \cup \{+\infty\}$, with $\overline \lambda \neq \overline \lambda_h$ for every $h$. We now show that  $\overline \lambda = +\infty$. 
By \eqref{pconv2}, it holds $\Vert \overline e_h \Vert_{L^2_{\overline p}}=1$ for every $h$ and so $\Vert \overline e_h \Vert^2_{V(I)}=\overline \lambda_h$. Assume by contradiction that $\overline \lambda \in \mathbb{R}$; then, the sequence $\{\overline e_h\}_h$ is bounded in $V(I)$, and hence it converges weakly to some $\bar{e} \in V(I)$, up to subsequences. Consequently, $\overline e_h \to \bar{e}$ strongly in $L^2_{\overline p}(I)$, so that $\Vert \bar{e} \Vert_{L^2_{\overline p}}=1$. Passing to the limit in the right-hand side of \eqref{allimite}, one then obtains that $\overline e$ is an eigenfunction of \eqref{beam0} for $p=\overline p$, with corresponding eigenvalue $\overline \lambda$. Since eigenfunctions associated with different eigenvalues are orthogonal, we have
$$
0= %\int_I \overline e_h'' \overline e'' \, dx =
\overline \lambda_h \int_I p \overline e_h \overline e \, dx \to \overline \lambda \int_I p \overline e^2 \, dx = \overline \lambda  \quad \text{as } h \to \infty,
$$
which is impossible since $\overline \lambda_h$ is an increasing sequence of positive numbers. Hence, $\overline \lambda_h \to +\infty$ for $h \to +\infty$. 
\par
%Arguing in the same way, we also obtain $\int_I p_m\,e_h(p_m)\,e_r(p_m)\,dx\rightarrow\int_I \overline p\,\overline e_h\,\overline e_r\,dx=\delta_{hr}$ for all $h,r \in\mathbb{N}_+$. This implies that
 %	$\overline \lambda_h$ is a diverging sequence for $h\rightarrow\infty$.  hence  This would imply that $\overline \lambda$ is as well an eigenvalue of \eqref{beam0} and $\overline e$ is an associated eigenfunction,  
Finally, we want to prove that $\overline \lambda_h = \lambda_h(\overline p)$ for every $h \in \N_+$; to this end, we show that problem \eqref{beam0} with $p=\overline p$ has no eigenvalues other than $\{\overline \lambda_h\}_h$. If by contradiction this were not true, there would exist another eigenfunction $\overline e$, associated with another eigenvalue $\overline \lambda$, such that $\big(\overline p\, \overline e,\overline e_h \big)_{L^2(I)}=0$ for all $h\in\mathbb{N}_+$.
 	We normalize $\overline e$ so that $\|\sqrt{\overline p}\,\overline e\|_{2}=1/\overline \lambda$; applying Lemma \ref{lemmaauchmuty} we would have 
 	\begin{equation}\label{assurdo}
 	-\dfrac{1}{2\lambda_h(p_m)}\leq \mathcal{A}_h(p_m,\overline e)=\dfrac{1}{2}\|\overline e\|^2_{V}-\|\sqrt{p_m}\,\big[\overline e-P_{h-1}(p_m)\overline e\big]\|_{2}\rightarrow \dfrac{1}{2}\|\overline e\|^2_{V}-\|\sqrt{\overline p}\,\overline e\|_{2}=-\dfrac{1}{2\overline \lambda},
 	\end{equation}
 	where the convergence comes from the fact that
 	$$
 	P_{h-1}(p_m)\overline e=\sum_{i=1}^{h-1}\big( p_m\overline e, e_i(p_m) \big)_{L^2}\,e_i(p_m)\rightarrow \sum_{i=1}^{h-1}\big(\overline p\, \overline e, \overline e_i \big)_{L^2}\,\overline e_i=0. %\quad {\rm in} \hspace{1mm} L^2.
 	$$
 	Therefore, by \eqref{assurdo}, letting $m\rightarrow \infty$, we would obtain
 	$$
 	\overline \lambda\geq \lambda_h(p_m)\rightarrow \overline \lambda_h\qquad \forall h\in\mathbb{N}_+,
 	$$
 	giving a contradiction since $\overline \lambda_h \to +\infty$ for $h\rightarrow\infty$. 
Hence, the set of the eigenvalues of \eqref{beam0} with $p=\overline p$ coincides with $\{\overline \lambda_h\}_h$, and its order is induced by the inequality  
$\lambda_1(p_m) < \lambda_2(p_m) < \lambda_3(p_m) < \ldots$ (valid for every $m$), which passes to the limit. 
We conclude that $\overline \lambda_h= \lambda_h(\overline p)$, implying the thesis. %of $p\mapsto\lambda_h(p)$ for every fixed $h\in\mathbb{N}_+$.

%  to prove that $\overline \lambda_h=\lambda_h(\overline p)$ for every $h\in \mathbb{N}_+$, 
 \end{proof} 
 {\bf Proof of Proposition \ref{existence} completed.}
 Let us consider the function $F: (0,+\infty)\times (0,+\infty)\to \mathbb{R}$ defined by $F(t,s):=\dfrac{t}{s}$, continuous on its domain.  By Lemma \ref{continuity}, the maps $p\mapsto \nu(p)$ and $p\mapsto \lambda(p)$ are continuous on $P_{\alpha,\beta}$ with respect to the weak* convergence; %for all %$j,k\in\mathbb{N}_+$; since $j>k\geq1$, then 
moreover, $\nu(p)>\lambda(p)>0$, so that also $F(\nu(p),\lambda(p))$ is continuous on $P_{\alpha,\beta}$ with respect to the same topology. Finally, the existence of a maximum (or a minimum) of $F(\nu(p),\lambda(p))=\frac{\nu(p)}{\lambda(p)}$ on $P_{\alpha,\beta}$ follows from the compactness of the set $P_{\alpha,\beta}$ with respect to the weak* topology of $L^{\infty}$, proved in Lemma \ref{compactness}.  % for all $j>k\geq1$, integers.\par 
 
 \subsection{Proof of Theorem \ref{thm-rapp}}
 We follows the lines of the proof of \cite[Theorem 3]{banks-gentry}. We say that a function $\delta p$ is an admissible variation of $p\in P_{\alpha,\beta}$ if $p+\delta p \in P_{\alpha,\beta}$. Then, by computing the first variation of the functional $\bigg(\dfrac{\nu}{\lambda}\bigg)(p)$ with respect to an admissible variation of $p\in P_{\alpha,\beta}$, see e.g. \cite[Theorems 1,2]{banks-gentry}, we readily get that
 \begin{equation*}
 %\label{var}
 \delta\bigg(\dfrac{\nu}{\lambda}\bigg)(p)=\int_I g(p,x)\, \delta p(x)\,dx,
 \end{equation*}
 where $g(p,x)$ is as defined in \eqref{gp}. 
 \medskip
 
Next, for $p\in P_{\alpha,\beta}$ and $t\in \R$, we consider the sets
\begin{equation*}
\begin{split}
S_p(t):=\{x\in \overline I: g(p,x)\geq t\}\quad \text{and} \quad T_p(t):=\{x\in \overline I: g(p,x)> t\}.
\end{split}
\end{equation*}
and the function
$$
I_p(t):=\beta |S_p(t)|+\alpha |S_p^c(t)|=(\beta-\alpha)|S_p(t)|+\alpha|I|,
$$
 where $S_p^c(t)=\overline I\setminus S_p(t)$. Clearly, $g(p,\cdot)\in C^0(\overline I)$; if $t'$ is a number less than the minimum of $g(p,\cdot)$ and $t''$ is a number greater than its maximum we get, respectively, $S_p(t')=\overline I$ and $S_p(t'')=\emptyset$. Therefore,
 \begin{equation*}
 	I_p(t')=\beta |I|\geq I_p(t)\geq \alpha |I|=I_p(t'')\qquad \forall t\in [t',t''].
 \end{equation*}
 We observe that $S_p(t)$ is a decreasing set function, implying that also the function $I_p(t)$ is decreasing; thus, $I_p(t)$ may have only jump discontinuities and there exists $t_0=t_0(p)$ such that either $I_p(t_0)=|I|$ or $I_p(t_0^-)> |I|> I_p(t_0^+)$.
 In this latter case we have $|A_p(t_0):=\{x\in \overline I: g(p,x)= t_0\}|>0$, $T_p(t_0) \subset S_p(t_0)$ and
 $$
 I_p(t_0^-)=\beta |S_p(t_0)|+\alpha |S_p^c(t_0)|\quad \text{and} \quad  I_p(t_0^+)=\beta |T_p(t_0)|+\alpha |T_p^c(t_0)|\,.
 $$
For $\theta \in \overline I$, we then define
 $$
S_\theta:=T_p(t_0)\cup\big([-\pi,\theta]\cap A_{p}(t_0) \big) \quad \text{and} \quad I(\theta):=\beta |S_\theta|+\alpha |S^c_\theta|\,.
 $$
The map $\overline I\ni \theta \mapsto I(\theta) $ is continuous and $I(-\pi)=I_p(t_0^+)$ and $I(\pi)=I_p(t_0^-)$. Hence, there exists $\theta_0=\theta_0(p)>-\pi $ such that $I(\theta_0)=|I|$, i.e. there exists a set $S_{ \theta_0}$ such that
 $$
 \beta |S_{ \theta_0}|+\alpha |S^c_{ \theta_0}|=|I|\qquad \text{with }T_p(t_0)\subseteq S_{\theta_0}\subset S_p(t_0).
 $$
 Let $\widehat{p}$ be a maximizer of \eqref{opt}. We set $\widehat t:=t_0(\widehat p)$, $\widehat \theta:=\theta_0(\widehat p)$,  $\widehat I:=S_{\widehat \theta}$ and we introduce the weight:
 \begin{equation*}
 p_{\widehat{\theta}}(x):=\beta \chi_{ \widehat I}(x)+\alpha\chi_{I \setminus \widehat I}(x)\,.
 \end{equation*}
 Clearly, $p_{\widehat{\theta}} \in P_{\alpha,\beta}$ and we may consider the variation of $\bigg(\dfrac{\nu}{\lambda}\bigg)( \widehat{p})$ with respect to the admissible variation $\delta \widehat{p}=p_{\widehat{\theta}}-\widehat{p}$, namely:
$$
\delta \bigg(\dfrac{\nu}{\lambda}\bigg)( \widehat{p})=\int_{\widehat I} g(\widehat{p},x)\, \delta p(x)\,dx+\int_{I\setminus\widehat I} g(\widehat{p},x)\, \delta p(x)\,dx.
$$
 By $\widehat p\in P_{\alpha,\beta}$ and the definition of $\widehat I$ we have
$$
 \delta \bigg(\dfrac{\nu}{\lambda}\bigg)( \widehat{p})\geq t_0\bigg[\int_{\widehat I} [\beta-\widehat{p}(x)]\,dx+\int_{I\setminus\widehat I} [\alpha-\widehat{p}(x)]\,dx\bigg]=t_0\bigg[\int_{I} p_{\widehat{\theta}}(x)\,dx-\int_{ I} \widehat{p}(x)\,dx\bigg] =0.
$$
 The inequality above becomes strict if $\beta >\widehat{p}$ or $\alpha <\widehat{p} $ on a set of positive measure $J\subset \overline I$ such that $ J\setminus \widehat  A_{\widehat t} \neq \emptyset$ where $A_{\widehat t}:=A_{\widehat p}(t_0)$.  In this case, we get $ \delta \bigg(\dfrac{\nu}{\lambda}\bigg)( \widehat{p})>0$, contradicting the fact that $\widehat{p} $ is a maximizer. 
 \par
\bigskip
\par
 \subsection{Proof of Proposition \ref{gmax}}
The first part of the statement follows from the fact that both $g(p, x)$ and 
$$
g'(p, x)=2 \frac{\nu(p)}{\lambda(p)}\big(e_\lambda(x)e'_\lambda(x)-e_\nu(x)e'_\nu(x)\big)
$$
vanish for $x=\pm \pi$ and for $x=\pm a\pi$, since all the eigenfunctions vanish therein.
\par
We now prove the second part of the statement, for which we assume that the two eigenfunctions $e_\lambda$ and $e_\nu$ have different parity. 
Let $e_\lambda$ be even; then, we know that $e'_\lambda(0)=e'''_\lambda(0)=0$. We claim that, up to a change of sign, $e_\lambda(0)>0$ and $e''_\lambda(0)<0$. To prove this, assume by contradiction that $e_\lambda(0)=0$; it follows that either $e''_\lambda(0)>0$ or $e''_\lambda (0)<0$ (otherwise $e_\lambda \equiv 0$). Arguing as in the proof of Lemma \ref{nehari}, in the first case we get $e_\lambda(a\pi)>0$, while in the second case we obtain $e_\lambda(a\pi)<0$; in any case, we reach a contradiction. Hence, we may assume $e_\lambda(0)>0$. Now, if it were $e''_\lambda(0)\geq 0$, arguing once more as in the proof of Lemma \ref{nehari} we would infer $e_\lambda(a\pi)>0$, a contradiction. Hence, we also have $e''_\lambda(0)<0$ and the claim is completely proved.
Since $e_\nu(0)=0$, we have
$$
g(p,0)=\frac{\nu(p)}{\lambda(p)}e_\lambda^2(0)>0,\quad \quad g'(p,0)=2 \frac{\nu(p)}{\lambda(p)}\big(e_\lambda(0)e'_\lambda(0)-e_\nu(0)e'_\nu(0)\big)=0,
$$	
and, thanks to the claim that we have just proved,
$$
g''(p,0)=2 \frac{\nu(p)}{\lambda(p)}\big(e_\lambda(0)e''_\lambda(0)-(e'_\nu(0))^2\big) \leq 2 \frac{\nu(p)}{\lambda(p)}\, e_\lambda(0)e''_\lambda(0)<0\,.
$$
It follows that $g$ has a local maximum point at $x=0$. \par

\par\bigskip\noindent
 \textbf{Acknowledgments.} The authors are members of the Gruppo Nazionale per l'Analisi Matematica, la Probabilit\`a e le loro Applicazioni (GNAMPA) of the Istituto Nazionale di Alta Matematica (INdAM) and are partially supported by the PRIN project 201758MTR2: ``Direct and inverse problems for partial differential equations: theoretical aspects and applications'' (Italy). The first two authors are partially supported by the INDAM-GNAMPA 2019 grant: ``Analisi spettrale per operatori ellittici con condizioni di Steklov o parzialmente incernierate''.

\end{document}